\documentclass[leqno]{amsart}
\usepackage{stmaryrd,graphicx}
\usepackage{amssymb,mathrsfs,amsmath,amscd,amsthm,color}
\usepackage{float}
\usepackage[all,cmtip]{xy}
\DeclareMathAlphabet{\mathpzc}{OT1}{pzc}{m}{it}
\usepackage{amsfonts,latexsym,wasysym,tipa}
\usepackage{quiver}
\newcommand{\im}{\operatorname{Im}}
\newcommand{\lchord}{\llbracket}
\newcommand{\rchord}{\rrbracket}
\renewcommand{\int}{\operatorname{int}}
\newcommand{\lpc}{\operatorname{{\bf lpc}}}
\newcommand{\hull}{\operatorname{hull}}

\newcommand{\ds}{\displaystyle}
\newcommand{\ui}{[0,1]}

\newcommand{\wtd}{\widetilde{d}}
\newcommand{\tX}{\widetilde{X}}

\newcommand{\lb}{\langle}
\newcommand{\rb}{\rangle}

\newcommand{\scrr}{\mathscr{R}}
\newcommand{\scrp}{\mathscr{P}}
\newcommand{\scrq}{\mathscr{Q}}

\newcommand{\mcn}{\mathcal{N}}
\newcommand{\mco}{\mathcal{O}}
\newcommand{\mcp}{\mathcal{P}}

\newcommand{\mcu}{\mathcal{U}}

\newcommand{\scrc}{\mathscr{C}}

\newcommand{\scrf}{\mathscr{F}}
\newcommand{\scrg}{\mathscr{G}}

\newcommand{\scru}{\mathscr{U}}

\newcommand{\scrs}{\mathscr{S}}

\newcommand{\bbd}{\mathbb{D}}

\newcommand{\bbn}{\mathbb{N}}

\newcommand{\bbr}{\mathbb{R}}

\newcommand{\bbz}{\mathbb{Z}}

\newcommand{\tree}{\mathbf{Tr}}
\newcommand{\pc}{\mathbf{Pc}}

\newcommand{\mcs}{\mathcal{S}}
\newcommand{\ov}{\overline}
\newcommand{\wt}{\widetilde}

\DeclareMathOperator{\diam}{diam}

\newtheorem{theorem}{Theorem}[section]
\newtheorem{lemma}[theorem]{Lemma}
\newtheorem{proposition}[theorem]{Proposition}
\newtheorem{claim}[theorem]{Claim}

\newtheorem{corollary}[theorem]{Corollary}
\theoremstyle{definition}\newtheorem{definition}[theorem]{Definition}
\newtheorem{example}[theorem]{Example}

\newtheorem{remark}[theorem]{Remark}

\title{Fundamental Groups of Disjointly Tree-Graded Spaces}

\author[J. Brazas]{Jeremy Brazas}
\address{West Chester University\\ Department of Mathematics\\
West Chester, PA 19383, USA}
\email{jbrazas@wcupa.edu}

\author[C. Kent]{Curtis Kent}
\address{Brigham Young University\\ Department of Mathematics \\
Provo, UT 84602, USA}
\email{curtkent@math.byu.edu}

\subjclass[2010]{Primary 55Q52; Secondary 20F34, 57M10, 20F69, 55P55}
\keywords{disjointly tree-graded space, tree-graded space, $\mathbb{R}$-tree, fundamental group, generalized covering map}
\date{\today}

\begin{document}

\begin{abstract}
Tree-graded spaces are a generalization of $\mathbb{R}$-trees and play an important role in describing the large-scale geometry of relatively hyperbolic groups. We consider a subclass of tree-graded spaces that we call \emph{disjointly tree-graded spaces}, determined by maps to $\mathbb{R}$-trees. We characterize the fundamental group of a disjointly tree-graded space $(X,\mathscr{P})$ in terms of the fundamental groups of its pieces. Our results apply even in cases where neither $X$ nor its pieces are locally simply connected. In particular, we show that if the pieces are uniformly $1$-$UV_0$, then the fundamental group of a disjointly tree-graded space embeds into the inverse limit of the free products of the fundamental groups of finitely many pieces.
\end{abstract}

\maketitle

\tableofcontents

\section{Introduction}

The concept of a ``tree-graded space" was introduced by Dru\textsubrhalfring{t}u and Sapir in \cite{DrutuSapir} as a generalization of $\bbr$-trees that allows one to characterize the large scale geometry of relatively hyperbolic groups. A complete geodesic metric space $X$ is said to be \textit{tree-graded} with respect to a collection $\scrp$ of closed geodesic subsets of $X$ (called \textit{pieces}) if (1) any two distinct pieces meet in at most one point and (2) every simple closed curve (equivalently, any geodesic triangle) in $X$ is contained in one piece. In a tree-graded space, pieces are arranged in a tree-like fashion so that an injective path starting and ending in distinct pieces must pass through a unique linearly ordered family of pieces and each connected subspace of $X$ that is disjoint from $\bigcup\scrp$ is an $\bbr$-tree. The use of tree-graded spaces in geometric group theory stems from the fact that, for a group, the property of all asymptotic cones having non-trivial tree-gradings is a weak form of hyperbolicity. In \cite{HKflats}, it is shown that a $CAT(0)$ group has isolated flats if and only if all its asymptotic cones are tree-graded with pieces isometric to euclidean space. This is a special case of the general results that a finitely generated group is hyperbolic relative to a collection of subgroups if and only if all of its asymptotic cones are tree-graded relative to limits of the peripheral subgroups \cite[Theorem 8.5]{DrutuSapir}.

Fundamental groups of tree-graded spaces play an important role in their applications within geometric group theory. In \cite[Proposition 2.2]{DrutuSapir} it is proved that if the pieces $\scrp$ of $X$ are locally uniformly contractible, then $\pi_1(X)$ is isomorphic to the free product of the fundamental groups of the pieces. The hypothesis ``locally uniformly contractible" restricts one to considering tree-graded spaces where all of the pieces are all large and tame, excluding situations where null-sequences of essential loops may appear. In this paper, we develop tools for understanding fundamental groups of tree-graded spaces where pieces are allowed to have arbitrarily small diameters. Note this situation can occur when considering asymptotic cones of relatively hyperbolic groups.  For example, any finitely generated group that is hyperbolic relative to a Baumslag-Solitar group (with $|p|\neq |q|$, where the relation is $ta^pt^{-1} = a^q$) will have pieces which are not locally simply connected, since Burillo \cite{Burillo1999} showed that the asymptotic cones of Baumslag-Solitar groups have $\pi_1$-essentially embedded infinite earrings.  In fact, Theorem B and Corollary C of \cite{ConnerKent2014} give a condition for a group to have asymptotic cones that are not locally simply connected  and some natural examples satisfying this condition. Additionally, many lacunary hyperbolic groups can by tree-graded by pieces which are not locally simply connected by applying techniques from \cite{OlshanskiiOsinSapir, ThomasVelickovic}.

In the current paper, we loosen some of the geometric restrictions typically imposed on the underlying space $X$ of a tree-graded space $(X,\scrp)$. For instance, we only require that $X$ be a path-connected, locally path-connected, metrizable space. On the other hand, we impose more restrictions on what we permit to be called a ``piece." Making this restriction is important to achieving our goal of characterizing the fundamental group of tree-graded space in terms of the fundamental groups of it's pieces. For instance, an $\bbr$-tree $T$ is tree-graded with respect to the collection of one-point subsets $\{t\}$, $t\in T$ but this decomposition is redundant since $T$ is simply connected. Extending this thought, we typically wish to avoid calling a one-point set $\{x\}$ a piece if it lies in an open subspace of $X$ that is an $\bbr$-tree. To this end, we consider tree-graded spaces $(X,\scrp)$ where the set $\scrp$ of pieces satisfies the following additional conditions: 
\begin{enumerate}
\item Pieces are the path-components of the union $\pc(X)=\bigcup\scrp$ of the pieces;
\item The set $\pc(X)$ is closed in $X$.
\end{enumerate}
We call $\scrp$ a \textit{disjoint tree-grading} on $X$ and we refer to the pair $(X,\scrp)$ as a \textit{disjointly tree-graded space}. In such a space, the set $\tree(X)=X\backslash \pc(X)$ is a disjoint union of open $\bbr$-trees that we call the \textit{tree-portion} of $X$ (see Figure \ref{fig1}). A structural advantage of disjoint tree-gradings is that distinct pieces are always separated by some component of the tree-portion of $X$ and thus it is possible to collapse pieces from a subset $\scrq\subseteq \scrp$ without affecting the pieces in $\scrp\backslash\scrq$. See Remark \ref{comparisonremark} for a detailed comparison of our definition with Dru\textsubrhalfring{t}u and Sapir's definition.

\begin{figure}[t]
\centering \includegraphics[height=2.3in]{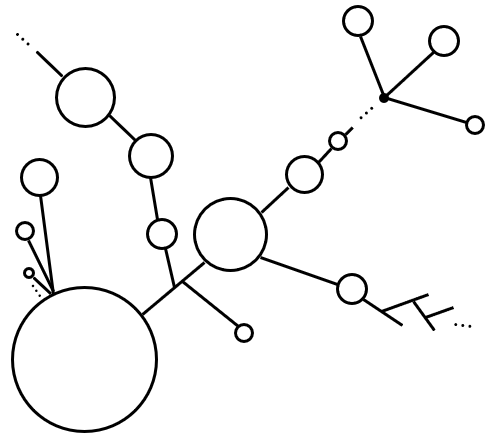}
\caption{\label{fig1}A disjointly tree-graded space where the pieces are illustrated as circles or enlarged points.}
\end{figure}


Our definition is also motivated by an apparent relevance to fundamental groups of Peano continua. If $X$ is a Peano continuum and a proper closed subspace $A\subseteq X$ is such that the quotient $X/A$ is one-dimensional, then it is known that $X/A$ admits a generalized universal covering map $p:E\to X/A$ in the sense of Fischer and Zastrow \cite{FZ07} where $E$ is an $\bbr$-tree. In forthcoming work, the authors prove that the pullback $p':E'\to X$ of $p$ along the quotient map $q:X\to X/A$ is a generalized covering map where $E'$ inherits the structure of a disjointly tree-graded space with pieces homeomorphic to the path-components of $A$. Intuitively, $E'$ is the result of ``unwinding" exactly the subspace $X\backslash A$ (assumed to be one-dimensional) into $\bbr$-tree components so that paths in $X$ lift uniquely to $E'$ relative to a chosen starting point.

When a space $X$ does not have a universal cover it is common to apply shape-theoretic methods wherein one detects the non-triviality of fundamental group elements using maps from $X$ to spaces from some special class. For instance, a space $X$ is said to be \textit{$\pi_1$-shape injective} if for every essential loop $\alpha:S^1\to X$, there exists a polyhedron $K$ and a map $f:X\to K$ such that $f\circ \alpha:S^1\to K$ is essential, that is, if maps to polyhedra detect all fundamental group elements. Spaces that are $\pi_1$-shape injective include one-dimensional metric spaces \cite{CConedim,EK98}, planar sets \cite{FZ07}, trees of manifolds \cite{FG05}, shrinking wedges of CW-complexes \cite{MM}, and many other spaces. Our overall approach is similar in spirit to $\pi_1$-shape injectivity. We detect non-trivial fundamental group elements in disjointly tree-graded spaces using maps to disjointly tree-graded spaces with finitely many pieces. However, our approach is quite a bit more general since we do not require pieces of our disjointly tree-graded spaces to be polyhedra, have tame fundamental groups, or even be $\pi_1$-shape injective themselves.

If $\scrp$ is a disjoint tree-grading on $X$ and $\scrf\subseteq \scrp$ is a subset, then there is a disjointly tree-graded space $X_{\scrf}$ and a metric quotient map $\Gamma_{\scrf}:X\to X_{\scrf}$ where each piece not in $\scrf$ is collapsed to a one-point piece of $X_{\scrf}$ (see Theorem \ref{quotientmetrictheorem}). Our main result allows us to detect the non-triviality of elements of the fundamental group by projecting to such metric quotients with finitely many pieces. Our only hypothesis is that small contractible loops in $\pc(X)$ may be contracted using relatively small null-homotopies, i.e. that $\pc(X)$ satisfies the uniformly $1$-$UV_0$ property (Definition \ref{uniformoneuvzerodef}). Note that this hypothesis permits the existence of null-sequences of essential loops.

\begin{theorem}\label{mainthm1}
Let $(X,\scrp)$ be a disjointly tree-graded space where $\pc(X)$ is uniformly $1$-$UV_0$. Then a loop $\alpha:S^1\to X$ is essential if and only if there exists a finite set of pieces $\scrf\subseteq \scrp$ such that $\Gamma_{\scrf}\circ \alpha:S^1\to X_{\scrf}$ is essential.
\end{theorem}

Example \ref{haexample} shows that Theorem \ref{mainthm1} can fail if $\pc(X)$ is not uniformly $1$-$UV_0$. The proof of Theorem \ref{mainthm1} involves many steps with some parts inspired by the proofs of $\pi_1$-shape injectivity referenced above. An immediate consequence of Theorem \ref{mainthm1} is that we are able to identify the fundamental group of a disjointly-tree graded space as a subgroup of an inverse limit of finite free products of the fundamental groups of the pieces of $X$. To state this precisely, we let $Fin(\scrp)$ be the set of finite subsets of $\scrp$ directed by subset inclusion. Whenever $\scrf\subseteq \scrf'$ in $Fin(\scrp)$, there exists a metric quotient map $\Gamma_{\scrf',\scrf}:X_{\scrf'}\to X_{\scrf}$ collapsing each piece in $\scrf'\backslash\scrf$ to a one-point piece (Remark \ref{levelmapremark}). If $x_0\in X$ and $x_{\scrf}=\Gamma_{\scrf}(x_0)$ for each finite set $\scrf\subseteq \scrp$, we obtain an inverse system with group homomorphism bonding maps $\Gamma_{\scrf',\scrf\#}:\pi_1(X_{\scrf'},x_{\scrf'})\to \pi_1(X_{\scrf},x_{\scrf})$. The metric quotient maps $\Gamma_{\scrf}:X\to X_{\scrf}$, $\scrf\in Fin(\scrp)$ induce a canonical homomorphism $\Phi:\pi_1(X,x_0)\to \varprojlim_{\scrf\in Fin(\scrp)}\pi_1(X_{\scrf},x_{\scrf})$ to the limit given by $\Phi(g)=(\Gamma_{\scrf\#}(g))_{\scrf}$.

\begin{corollary}\label{injectivelimitcor}
Let $(X,\scrp)$ be a disjointly tree-graded space for which $\pc(X)$ is uniformly $1$-$UV_0$. Then the canonical homomorphism \[\Phi:\pi_1(X,x_0)\to\varprojlim_{\scrf\in Fin(\scrp)}\pi_1\left(X_{\scrf},x_{\scrf}\right)\] defined above is injective.
\end{corollary}

For each $\scrf\in Fin(\scrp)$, the group $\pi_1\left(X_{\scrf},x_{\scrf}\right)$ is (non-canonically) isomorphic to the finite free product $\ast_{P\in\scrf}\pi_1(P,x_P)$ for some choice of basepoints $x_P\in P$ (see Corollary \ref{isocor2}). Hence, Corollary \ref{injectivelimitcor} successfully characterizes non-trivial elements of $\pi_1(X,x_0)$ in terms of the fundamental groups of the pieces of $X$. We also note that a disjointly tree-graded space with finitely many non-degenerate pieces all of which are polyhedra is homotopy equivalent to a polyhedron. Hence, the following is a straightforward consequence of Theorem \ref{mainthm1}.

\begin{corollary}
If $(X,\scrp)$ is a disjointly tree-graded space for which $\pc(X)$ is uniformly $1$-$UV_0$ and for which every piece $P\in\scrp$ is a polyhedron, then $X$ is $\pi_1$-shape injective.
\end{corollary}

In Section \ref{sectionapplications}, we use Theorem \ref{mainthm1} to identify situations where grade-preserving maps of disjointly tree-graded spaces (maps that send pieces into pieces) are $\pi_1$-injective. The following theorem is of particular importance to our anticipated applications to fundamental groups of Peano continua.

\begin{theorem}\label{mainapplicationcorollary}
Let $(X,\scrp)$ and $(Y,\scrq)$ be disjointly tree-graded spaces where $\pc(X)$ is uniformly $1$-$UV_0$ and let $f:X\to Y$ be a continuous function with the property that every piece of $X$ is mapped into some piece of $Y$ and such that no two pieces of $X$ are mapped into the same piece of $Y$. Then $f$ is $\pi_1$-injective if and only if $f|_{\pc(X)}:\pc(X)\to\pc(Y)$ is $\pi_1$-injective.
\end{theorem}



\section{Preliminaries}\label{sectionprelims}

All spaces considered in this paper are assumed to be Hausdorff. We take $\ui$, $S^1$, and $\bbd$ to denote the closed unit interval, unit circle, and closed unit disk respectively. A space homeomorphic to $\ui$ is called an \textit{arc} and a space homeomorphic to $S^1$ is called a \textit{simple closed curve}. Generally, the term ``map" will refer to a continuous function. If $f:(X,x_0)\to (Y,y_0)$ is a based map of spaces $f_{\#}:\pi_1(X,x_0)\to \pi_1(Y,y_0)$ will denote the homomorphism induced on fundamental groups. If $f:(X,d)\to (Y,\rho)$ is a function of metric spaces, we say $f$ is \textit{non-expansive} if $f$ does not increase distance, i.e. $\rho(f(x_1),f(x))\leq d(x_1,x_2)$ for all $x_1,x_2\in X$.

A \textit{path} in a space $X$ is a continuous function $\alpha:[a,b]\to X$. When $\alpha(a)=\alpha(b)$, we call $\alpha$ a \textit{loop} and we may regard $\alpha$ as a map $S^1\to X$. A loop $\alpha:S^1\to X$ is \textit{inessential (in $X$)} if there exists a map $g:\bbd\to X$ such that $g|_{S^1}=\alpha$ (equivalently, if $\alpha$ is null-homotopic) and is \textit{essential} otherwise. We let $\alpha\cdot\beta$ denote the concatenation of paths and $\alpha^{-1}$ denote the reverse path of $\alpha$.

A space $X$ is \textit{(uniquely) arcwise-connected} if for all distinct points $x,y\in X$ there is a (unique) arc $A\subseteq X$ with endpoints $x$ and $y$. Every path-connected Hausdorff space is arcwise-connected and it is well-known that a Hausdorff space is uniquely arcwise-connected if and only if $X$ does not contain a simple closed curve as a subspace. This characterization is implied by the following lemma, which is more specific to our purposes.

\begin{lemma}\label{scclemma}
If $X$ is a Hausdorff space, $A_1,A_2\subseteq X$ are arcs sharing the same endpoints $x$ and $y$, and $z\in A_1\backslash A_2$, then there exists a simple closed curve $C\subseteq A_1\cup A_2$ containing (1) the point $z$, (2) a subarc of $A_1$, and (3) a subarc of $A_2$.
\end{lemma}

\begin{proof}
Let $\alpha_1,\alpha_2:\ui\to X$ be injective paths parameterizing $A_1,A_2$ respectively. Let $s$ be the unique element of $\alpha_{1}^{-1}(z)$. The set $C=\alpha_{1}^{-1}(\im(\alpha_2))$ is closed in $\ui$ and we have $\{0,1\}\subseteq C\subseteq \ui\backslash\{s\}$ by assumption. Hence, there exists a connected component $(a,b)$ of $\ui\backslash C$ with $s\in (a,b)$. Since $a,b\in C$ and $\alpha_2$ is injective, we have $\alpha_1(a)=\alpha_2(c)$ and $\alpha_1(b)=\alpha_2(d)$ for unique $c,d\in\ui$. Taking the reverse path of $\alpha_2$ if necessary, we may assume that $c<d$. Now $C=\alpha_1([a,b])\cup \alpha_2([c,d])$ is a simple closed curve satisfying the desired properties.
\end{proof}

\begin{definition}
A topological space $T$ is
\begin{enumerate}
\item a \textit{topological tree} if $T$ is Hausdorff, uniquely arcwise-connected, and locally arcwise-connected,
\item a \textit{topological $\bbr$-tree} if $T$ is a metrizable topological tree,
\item a \textit{dendrite} if $T$ is a compact topological $\bbr$-tree.
\end{enumerate}
\end{definition}

It follows from Lemma \ref{scclemma} that a Hausdorff space $X$ is a topological tree if and only $X$ is connected, locally path-connected and contains no simple closed curve. Note also that a subspace of a topological tree is connected (resp. totally disconnected) if and only if it is path-connected (resp. totally path-disconnected). A metric space $(X,d)$ is an \textit{$\bbr$-tree} if $X$ is uniquely arcwise connected and if each arc in $X$ is the image of an isometric embedding $[a,b]\to X$ of a real closed interval \cite{JTits}. If $(X,d)$ is an $\bbr$-tree, then $X$ is certainly a topological $\bbr$-tree. Far less trivial is the fact that every topological $\bbr$-tree $X$ admits a compatible metric $d$ for which $(X,d)$ is an $\bbr$-tree \cite{MayerOverstTrees}. It is well-known that $\bbr$-trees are contractible. Clearly every topological tree is weakly equivalent to a point. However, it is apparently an open question if all topological trees are contractible (in ZFC).

A \textit{Peano continuum} is a connected, locally connected compact metrizable space. The Hahn-Mazurkiewicz states that a Hausdorff space $X$ is a Peano continuum if and only if there exists a surjective path $\alpha:\ui\to X$. A topological tree is a Peano continuum if and only if it is a dendrite.

\section{Disjointly tree-graded spaces and their parameterizations}\label{sectiondisjointlytreegradedspaces}

\subsection{Basic definitions}

\begin{definition}\label{deftreegraded}
Let $X$ be a path-connected, locally path-connected, metrizable space and $\scrp$ be a collection of non-empty closed subsets of $X$ (called \textit{pieces}) such that the following properties are satisfied:
\begin{enumerate}
\item $\bigcup\scrp$ is closed in $X$,
\item The elements of $\scrp$ are the path-components of $\bigcup\scrp$,
\item Every simple closed curve in $X$ is contained in one piece.
\end{enumerate}
We call $\scrp$ a \textit{disjoint tree-grading} on $X$ and we refer to the pair $(X,\scrp)$ as a \textit{disjointly tree-graded space}.

We refer to $\pc(X)=\bigcup\scrp$ as the \textit{piece-portion} of $X$ and $\tree(X)=X\backslash\pc(X)$ as the \textit{tree-portion} of $X$. We say a piece is \textit{degenerate} if it is a one-point space and \textit{non-degenerate} otherwise.
\end{definition}


\begin{remark}
In Lemma \ref{retractionlemma} below, we will show that a disjointly tree-graded space retracts onto every one of it's pieces. Since disjointly tree-graded spaces are assumed to be locally path-connected, it follows that every piece is locally path-connected.
\end{remark}

\begin{remark}\label{comparisonremark}
When $X$ can be equipped with a complete geodesic metric, our notion of ``disjoint tree-grading" on a space $X$ can be compared with the notion of ``tree-grading" in the sense of Dru\textsubrhalfring{t}u-Sapir \cite{DrutuSapir}. In particular, every disjoint tree-grading $\scrp$ on $X$ is a tree-grading. On the other hand, if $\scrp '$ is a tree-grading on $X$, then $\bigcup\scrp '$ need not be closed, the pieces need not be the path-components of $\bigcup\scrp '$, and it may be that two distinct pieces meet at a point. In general, given a tree-grading $\scrp'$ on $X$, one can construct a disjoint tree-grading $\scrp$ on $X$ by taking the new pieces to be the smallest closed, geodesic subsets $C$ such that $P\subseteq C$ whenever $P\in\scrp '$ and $P\cap C\neq\emptyset$. This process will result in a \textit{non-trivial} disjoint tree-grading $\scrp$ if and only if $X\backslash \bigcup\scrp'$ contains a non-empty open subset, which is an $\bbr$-tree. For example, if $X=\bigcup_{n\in\bbz}D_n$ where $D_n\subseteq \bbr^2$ is the disk of radius $1/2$ centered at $(n,0)$, then $\scrp '=\{D_n\mid n\in\bbz\}$ is a tree-grading on $X$. However, no non-trivial disjoint tree-grading on $X$ exists. This particular example also highlights the fact that one cannot collapse pieces in an ordinary tree-graded space without affecting other pieces (if each disk $D_n\in\scrp '$ is identified to a point, the resulting quotient of $X$ will be a one-point space). 
\end{remark}

\begin{definition}
A \textit{path-diameter metric} on a set $X$ is a metric $\rho$ on $X$ that satisfies \[\rho(a,b)=\inf\{\diam(\im(\alpha))\mid \alpha\text{ is a path from }a\text{ to }b\}.\]
\end{definition}

If $\rho$ is a path-diameter metric on $X$, the open balls $B_{\epsilon}(x)=\{y\in X\mid \rho(x,y)<\epsilon\}$ are path-connected. Throughout the remainder of the paper, we will assume that the metric inducing the topology of a disjointly tree-graded space is a path-diameter metric.  This is justified by the following remark.

\begin{remark}\label{diametermetricremark}
If $d$ is an arbitrary metric on a set $X$ such that the induced topology gives $X$ the structure of a locally path-connected space, then we can define \[\rho(a,b)=\inf\{\diam_{d}(\im(\alpha))\mid \alpha:\ui\to X\text{ is a path from }a\text{ to }b\}\] where the diameter in this definition is taken with respect to $d$. It is straightforward to check that $\rho$ is a metric on $X$, which is compatible with the topology of $X$ and such that the identity function $(X,\rho)\mapsto (X,d)$ is non-expansive. Moreover, since $\diam_{\rho}(\im(\alpha))= \diam_{d}(\im(\alpha))$ for all paths $\alpha:\ui\to X$, $\rho$ is a path-diameter metric on $X$. Hence, any locally path-connected metrizable space may be equipped with a path-diameter metric.
\end{remark}

\subsection{The topological structure of $\tree(X)$}

In the following two subsections, we fix a disjointly tree-graded space $(X,\scrp)$ and a path-diameter metric $d$ on $X$.

\begin{proposition}\label{treexprop}
$\tree(X)$ is a disjoint union of open topological $\bbr$-trees. Moreover, if $U$ is a connected component of $\tree(X)$, then $\ov{U}\backslash U\subseteq \pc(X)$.
\end{proposition}

\begin{proof}
If $U$ is a connected component of $\tree(X)$, then $U$ is open, path-connected, locally path-connected, and contains no simple closed curve. Since $U$ is also metrizable, $U$ is a topological $\bbr$-tree. Moreover if $x\in (\ov{U}\backslash U)\cap \tree(X)$, then $x$ is an element of some other component $V$ of $\tree(X)$. But since $V$ is open and $x\in \ov{U}\cap V$, we have $U\cap V\neq \emptyset$, violating the fact that $U$ and $V$ are distinct connected components of $\tree(X)$.
\end{proof}

\begin{proposition}\label{separateprop}
If $\alpha:\ui\to X$ is an injective path where the endpoints $x=\alpha(0)$ and $y=\alpha(1)$ do not lie in the same piece of $X$, then $\alpha((0,1))\cap \tree(X)\neq\emptyset$ and for every $z\in \alpha((0,1))\cap \tree(X)$, the points $x$ and $y$ lie in distinct components of $X\backslash\{z\}$.
\end{proposition}

\begin{proof}
Since $\tree(X)$ is open and since the pieces of $X$ are the path-components of $\pc(X)$, we must have $\alpha((0,1))\cap \tree(X)\neq\emptyset$. Let $t\in (0,1)$ such that $z=\alpha(t)\in \tree(X)$. Since the component $U$ of $\tree(X)$ containing $z$ is an open topological $\bbr$-tree (Proposition \ref{treexprop}), we may find $0<a<t<b<1$ such that $A_1=\alpha([a,b])\subseteq U$. Since $X$ is locally path-connected, the components of $X\backslash\{z\}$ are path-connected. Hence, if $x,y$ lie in the same component of $X\backslash\{z\}$, then we may find and arc $A_2\subseteq X\backslash\{z\}$ with endpoints $\alpha(a)$ and $\alpha(b)$. Since $z\in A_1\backslash A_2$, Lemma \ref{scclemma} gives a simple closed curve $C$ with $z\in C\subseteq A_1\cup A_2$ but this contradicts (3) in Definition \ref{deftreegraded}.
\end{proof}

\begin{proposition}\label{arcsprop}
If $U$ is a connected component of $\tree(X)$ and $x,y$ are distinct elements of $\ov{U}$, then there is a unique arc $A$ in $X$ with endpoints $x$ and $y$. Moreover, $A$ satisfies $A\backslash\{x,y\}\subseteq U$.
\end{proposition}

\begin{proof}
We prove the proposition in three successive cases.

\noindent Case I: Suppose $x,y\in U$. Let $A_1$ be an arc in $X$ with endpoints $x$ and $y$. If $A_1\nsubseteq U$, then there exists a point $z\in A_1\backslash U$. By Proposition \ref{treexprop}, $U$ is an $\bbr$-tree and so there exists a unique arc $A_2$ in $U$ with endpoints $x$ and $y$. By Lemma \ref{scclemma}, there exists a simple closed curve $C\subseteq A_1\cup A_2$ that contains a subarc of $A_2$. However, this contradicts (3) of Definition \ref{deftreegraded}.

\noindent Case II: Suppose $x\in U$ and $y\in \ov{U}\backslash U$. By Proposition \ref{treexprop}, we have $y\in P$ for some $P\in\scrp$. First we verify the existence of the desired arc. Find a properly nested neighborhood basis $X=W_1\supsetneq W_2\supsetneq W_3\supsetneq \cdots$ at $y$ consisting of path-connected open sets and finds points $w_i\in (W_i\cap U)\backslash W_{i+1}$. Specifically, choose $w_1=x$. Let $\beta_i:\ui\to W_i$ be an injective path from $w_i$ to $w_{i+1}$. By Case I, we have $\im(\beta_i)\subseteq U$ for all $i\in\bbn$. Define path $\beta:\ui\to X$ by defining $\beta$ to be $\beta_i$ on $\left[\frac{i-1}{i},\frac{i}{i+1}\right]$ and $\beta(1)=y$. By construction, $\beta$ is continuous and satisfies $\beta(0)=x$, $\beta([0,1))\subseteq U$, and $\beta(1)=y$. Any choice of injective path $\alpha:\ui\to \im(\beta)$ from $x$ to $y$ parameterizes an arc $A$ with the desired properties. For uniqueness, let $A_1=A$ and suppose that $A_2$ is an arc in $X$ with the same endpoints but for which $A_1\neq A_2$. If there is a point $z\in A_1\backslash A_2$, then Lemma \ref{scclemma} gives a simple closed curve $C$ containing $z$. However $z\in A_1\backslash\{x,y\}\subseteq U$, contradicting (3) in Definition \ref{deftreegraded}. If there is a point $z\in A_2\backslash A_1$, then Lemma \ref{scclemma} gives a simple closed curve $C$ containing a subarc of $A_1$. However, every subarc of $A_1$ must meet $U$, again giving a contradiction.

\noindent Case III: Suppose $x,y\in \ov{U}\backslash U$. Fix $z\in U$. By Case II, find arcs $A_1,A_2\subseteq \ov{U}$ with $\ov{U}\backslash A_1=\{x\}$, $\ov{U}\backslash A_2=\{y\}$, and $z\in A_1\cap A_2$. An arc $A\subseteq A_1\cup A_2$ with endpoints $x$ and $y$ satisfies the desired conditions. Uniqueness follows by the same type of argument used in Case II.
\end{proof}

\begin{proposition}\label{componentclosure}
If $U$ is a connected component of $\tree(X)$, then $\ov{U}$ is a topological $\bbr$-tree that meets each piece of $X$ in at most one point.
\end{proposition}

\begin{proof}
It follows from Proposition \ref{arcsprop} that $\ov{U}$ is uniquely arcwise connected. If $x\in \ov{U}\backslash U$ and $V$ is a path-connected open neighborhood of $x$ in $X$, then Proposition \ref{arcsprop} also implies that $V\cap \ov{U}$ is path-connected. Hence, $\ov{U}$ is locally path-connected and it follows that $\ov{U}$ is a topological $\bbr$-tree.

Suppose $P$ is a piece of $X$ and $x,y$ are distinct points of $\ov{U}\cap P$. By Proposition \ref{arcsprop}, there exists an arc $A_1\subseteq \ov{U}$ with endpoints $x$ and $y$ and such that $A_1\backslash \{x,y\}\subseteq U$. Since pieces are path-connected, there exists an arc $A_2\subseteq P$ with endpoints $x$ and $y$. Then $A_1\cup A_2$ is a simple closed curve that meets $U$, contradicting (3) in Definition \ref{deftreegraded}.
\end{proof}

The above characterization of $\ov{U}$ for a connected component $U$ of $\tree(X)$ allows us to identify convenient representatives of path-homotopy classes.

\begin{definition}\label{treeefficientdef}
A path $\beta:\ui\to X$ is \textit{tree-efficient} if for every connected component $J$ of $\beta^{-1}(\tree(X))$, the restriction $\beta|_{\ov{J}}:\ov{J}\to X$ is an injective path satisfying $\beta|_{\ov{J}}^{-1}(\tree(X))=J$.
\end{definition}

\begin{lemma}\label{efficientlemma}
Every non-constant path $\alpha:\ui\to X$ is path-homotopic to a tree-efficient path $\beta:\ui\to X$. Moreover, we may choose the homotopy $H$ from $\alpha$ to $\beta$ to be the constant homotopy on $\alpha^{-1}(\pc(X))$ and so that $\im(H)\subseteq \im(\alpha)$.
\end{lemma}

\begin{proof}
If $\alpha$ has image in $\pc(X)$, then we take $\alpha=\beta$ and let $H$ be the constant homotopy. Otherwise, for each connected component $J$ of $\alpha^{-1}(\tree(X))$ (which is an open or half-open interval), the path $\alpha|_{\ov{J}}$ maps $\int(J)$ into some connected component $U_J$ of $\tree(X)$ and at least one endpoint of $J$ into $\pc(X)$. We have $\alpha(\ov{J})\subseteq \ov{U_J}$ where $\ov{U_J}$ is a topological $\bbr$-tree by Proposition \ref{componentclosure}.

If $\alpha|_{\ov{J}}$ is a loop based at a point $x\in \pc(X)$, then $\alpha|_{\ov{J}}$ is a loop in the topological $\bbr$-tree $\ov{U_J}$ and is therefore null-homotopic in it's own image. Hence, there exists a map $H_J:\ov{J}\times\ui \to \alpha(\ov{J})$ with $H_J(s,t)=\alpha(s)$ for $s\in \ov{J}$ and $H_J(\partial J\times\ui\cup \ov{J}\times\{1\})=x$. If $\alpha|_{\ov{J}}$ is path with distinct endpoints, we write $\ov{J}=[a,b]$ and define $\beta_J:\ui\to \alpha(\ov{J})$ to be an injective path parameterizing the unique arc in the dendrite $\alpha(\ov{J})$ from $\alpha(a)$ to $\alpha(b)$. Since $\alpha(\ov{J})$ is simply connected, there exists a map $H_J:\ov{J}\times\ui \to \alpha(\ov{J})$ with $H_J(s,0)=\alpha(s)$, $H_J(s,1)=\beta(s)$ for $s\in\ov{J}$ and such that $H_J(\{a\}\times\ui)=\alpha(a)$ and $H_J(\{b\}\times\ui)=\alpha(b)$.

With all of the maps $H_J$ defined, we define a map $H:\ui^2\to X$ so that $H(s,t)=\alpha(s)$ if $s\in \alpha^{-1}(\pc(X))$. If $s\in J$ for some component $J$ of $\alpha^{-1}(\tree(X))$, we define $H(s,t)=H_J(s,t)$. Since $\im(H_J)\subseteq \alpha(\ov{J})$ for each $J$, the continuity of $H$ follows easily from the continuity of $\alpha$. The path $\beta:\ui\to X$ defined by $\beta(s)=H(s,1)$ is tree-efficient and the homotopy constructed has the desired properties.
\end{proof}

\subsection{Metric retractions}

Here, we consider special subspaces $Y$ of $X$ that admit a canonical retraction $r:X\to Y$ with several convenient properties.

\begin{lemma}\label{retractionlemma}
Let $Y\subseteq X$ be a closed, path-connected subspace with the property that if a piece $P$ of $X$ meets $Y$, then $Y\cap P$ is either a one-point space or all of $P$. Then there exists a retraction $r:X\to Y$ defined as follows: if $y\in Y$, then $r(y)=y$ and if $x\in X\backslash Y$, then $r(x)=y$ where $x$ and $y$ are endpoints of an arc $A\subseteq X$ with the property that $A\cap Y=\{y\}$. Moreover, the following hold:
\begin{enumerate}
\item if $x_1,x_2$ lie in the same connected component of $X\backslash Y$, then $r(x_1)=r(x_2)$,
\item if $\alpha:\ui\to X$ is path whose image meets $Y$, then $\im(r\circ\alpha)\subseteq \im(\alpha)$,
\item $r$ is non-expansive when $Y$ is equipped with the subspace metric.
\end{enumerate}
\end{lemma}

\begin{proof}
First, we show that $r$ is well-defined. Given a point $x\in X\backslash Y$, find an injective path $\alpha:\ui\to X$ with $\alpha(0)=x$ and $\alpha(1)\in Y$. Since $Y$ is closed, $s=\min\{t\in\ui\mid \alpha(t)\in Y\}$ exists and $A=\alpha([0,s])$ is an arc with endpoints $x$ and $\alpha(s)\in Y$ and $A\cap Y=\{\alpha(s)\}$. We show that the endpoint in $Y$ is unique. For $i\in\{1,2\}$, suppose $A_i\subseteq X$ is an arc with endpoint $x$ and $y_i\in Y$ and such that $A_i\cap Y=\{y_i\}$. Suppose, to obtain a contradiction that $y_1\neq y_2$. Now $A_1\cup A_2$ is a Peano continuum that meets $Y$ in the two-point set $\{y_1,y_2\}$. Let $B_1\subseteq A_1\cup A_2$ be an arc with endpoints $y_1$ and $y_2$. Since $Y$ is path-connected, we may find an arc $B_2\subseteq Y$ with endpoints $y_1$ and $y_2$. Since $C=B_1\cup B_2$ is a simple closed curve, we must have $C\subseteq P$ for some piece $P$ of $X$. However, since $B_1$ is not contained entirely in $Y$, our hypothesis about how $Y$ meets the pieces of $X$ forces $P$ to be a degenerate piece of $X$ but this contradicts the fact that $\{y_1,y_2\}\subseteq P$. Therefore, we must have $y_1=y_2$ and the function $r$ is well-defined.

Next, we prove (1). Suppose $x_1,x_2$ lie in the same connected component of $X\backslash Y$. Find arcs $A_i$ with endpoints $x_i$ and $y_i\in Y$ and with the property that $A_i\cap Y=\{y_i\}$ so that $r(x_i)=y_i$. Find an arc $B\subseteq X\backslash Y$ with endpoints $x_1\in A_1$ and $x_2\in A_2$. There exists an arc $A_2'\subseteq A_1\cup B$ with endpoints $x_2$ and $y_1$. Since $A_{2}'\cap Y=\{y_1\}$, the well-definedness of $r$ from the previous paragraph gives $y_1=r(x_2)=y_2$.

Next, we check that $r$ is continuous. Let $\{x_n\}_{n\in\bbn}\to x$ be a convergent sequence in $X$. Since the connected components of $X\backslash Y$ are open, the previous paragraph makes the case where $x\in X\backslash Y$ trivial. Suppose that $x\in Y$. Since $r$ is the identity on $Y$, we may focus on the case where $x_n\in X\backslash Y$ for all $n$. Let $V$ be a path-connected open neighborhood of $r(x)=x$ in $X$. Find $N\in\bbn$ such that $x_n\in V$ for all $n\geq N$. For each $n\geq N$, let $B_n\subseteq V$ be an arc with endpoints $x_n$ and $x$. Since $x_n\in X\backslash Y$ and $Y$ is closed, we may find subarcs $A_n\subseteq B_n$, $n\geq N$ with endpoints $x_n$ and $y_n\in Y$ and such that $A_n\cap Y=\{y_n\}$. Thus $r(x_n)=y_n\in A_n\subseteq B_n\subseteq V$ for all $n\geq N$, proving that $\{r(x_n)\}_{n\in\bbn}\to x$.

For (2), suppose that $\alpha:\ui\to X$ is a path whose image meets $Y$. Then $K=\alpha^{-1}(Y)$ is non-empty and closed. if $t\in K$, then $r(\alpha(t))=\alpha(t)$. Pick a point $t\in \ui\backslash K$ and find a closed interval $I$ with endpoints $t$ and $k\in K$ and such that $I\cap K=\{k\}$. Now $\alpha(I)$ is a Peano continuum containing $\alpha(t)$ and for which $\alpha(I)\cap Y=\{\alpha(k)\}$. Thus there exists an arc $A\subseteq \alpha(I)$ with endpoints $\alpha(t)$ and $\alpha(k)$. The definition of $r$ now gives $r(\alpha(t))=\alpha(k)$. This proves that $\im(\alpha\circ r)\subseteq \im(\alpha)$.

For (3), recall that $d$ is a path-diameter metric. Let $x_1,x_2\in X$. If $x_1$ and $x_2$ lie in the same connected component of $X\backslash Y$, then (1) gives $r(x_1)=r(x_2)$ and thus $0=d(r(x_1),r(x_2))\leq d(x_1,x_2)$. Suppose $x_1,x_2$ do not both lie in a single connected component of $X\backslash Y$. Let $\alpha:\ui\to X$ be a path from $x_1$ to $x_2$. Then the image of $\alpha$ must meet $Y$ and (2) implies that $r\circ\alpha:\ui\to X$ is a path from $r(x_1)$ to $r(x_2)$ satisfying $\diam(\im(r\circ\alpha))\leq\diam(\im(\alpha))$. Thus, $d(r(x_1),r(x_2))\leq d(x_1,x_2)$ holds.
\end{proof}

The next corollary follows from combining (1) and (2) of Lemma \ref{retractionlemma}.

\begin{corollary}\label{twocor}
Let $Y_1,Y_2$ be disjoint subspaces of $X$ meeting the hypotheses of Lemma \ref{retractionlemma} and let $r_1:X\to Y_1$ and $r_2:X\to Y_2$ be the retractions from the conclusion of Lemma \ref{retractionlemma}. If $\alpha:\ui\to X$ is any path with $\alpha(0)\in Y_1$ and $\alpha(1)\in Y_2$, then $\{r_1(Y_2),r_2(Y_1)\}\subseteq \im(\alpha)$.
\end{corollary}

\begin{corollary}
If $Y$ is a subspace of $X$ meeting the hypotheses of Lemma \ref{retractionlemma}, then the subspace metric on $Y$ is a path-diameter metric on $Y$.
\end{corollary}

\begin{proof}
Given $y_1,y_2\in Y$ and $\epsilon>0$, we can find a path $\alpha:\ui\to X$ from $y_1$ to $y_2$ with $\diam(\im(\alpha))<d(y_1,y_2)+\epsilon$. Since $\im(\alpha)$ meets $Y$, (2) of Lemma \ref{retractionlemma} implies that $r\circ\alpha:\ui\to Y$ is a path with $\diam(\im(r\circ\alpha))\leq \diam(\im(\alpha))<d(y_1,y_2)+\epsilon$.
\end{proof}

\begin{corollary}\label{connectedpreimagecor}
If $\beta:\ui\to X$ is an injective path, then for every piece $P\in\scrp$, $\beta^{-1}(P)$ is connected.
\end{corollary}

\begin{proof}
If there exists a $P\in\scrp$ for which $\beta^{-1}(P)$ is not connected, consider canonical the retraction $r:X\to P$ from Lemma \ref{retractionlemma}. Since $\beta^{-1}(P)$ is closed and not connected, there exists an open interval $(a,b)$, which is a connected component of $\ui\backslash \beta^{-1}(P)$. Then (1) of Lemma \ref{retractionlemma} implies that $\beta(a)=r(\beta(a))=r(\beta([a,b]))=r(\beta(b))=\beta(b)$, contradicting the assumption that $\beta$ is injective.
\end{proof}

\begin{proposition}\label{compactonepiece}
If $X$ is compact and $\scrp=\{P\}$, i.e. $X$ has a single piece $P$, then $P$ is a deformation retract of $X$.
\end{proposition}

\begin{proof}
Let $\mcu$ be the collection of connected components of $X\backslash P$. For each $U\in\mcu$, we have $\ov{U}\cap P=\{x_U\}$ for some point $x_U$. Since $\ov{U}$ is a dendrite (using Prop. \ref{componentclosure} and the fact that $X$ is compact) and all such spaces are point-wise contractible, we may choose a contraction $H_U:\ov{U}\times\ui\to \ov{U}$ so that $H(x,0)=x$, $H(x,1)=x_U$, and $H(x_U,t)=x_U$ for all $t\in\ui$. Define $H:X\times\ui\to X$ so that $H(x,t)=x$ if $x\in P$ and $H(x,t)=H_U(x,t)$ if $x\in \ov{U}$. Certainly, $H$ is well-defined as a function. The only non-trivial case to check in order to verify the continuity of $H$ is if $\{(x_n,t_n)\}_{n\in\bbn}\to (x,t)$ is a convergent sequence in $X\times\ui$ where $x_n\in U_n$ for pairwise-distinct $U_1,U_2,U_3,\dots\in\mcu$. Let $V$ be a neighborhood of $H(x,t)$ in $X$. Since each element of $\scru$ is open in $X$, we must have $x\in P$ and thus $H(x,t)=x$. Thus there exists $N\in\bbn$ such that $x_n\in V$ for all $n\geq N$. Since $H(x_n,t_n)\in \im(H_{U_n})\subseteq \ov{U_n}$, it suffices to show that $\ov{U_n}\subseteq V$ for all but finitely many $n$. Supposing otherwise, we may find $N<n_1<n_2<n_3<\cdots$ and points $a_i\in U_{n_i}\backslash V$. Since $X$ is compact we may assume that $\{a_i\}_{i\in\bbn}\to a$ in $X$. Since all the $U_{n_i}$ are distinct, we must again have $a\in X$. If $r:X\to P$ is the retraction from Lemma \ref{retractionlemma}, we have $r(x_{n_i})=r(a_i)$ for all $i\in\bbn$. Since $\{r(x_{n_i})\}_{i\in\bbn}\to x$ and $\{r(a_i)\}_{i\in\bbn}\to r(a)=a$, we must have $a=x\in V$. Since $\{a_i\}_{i\in\bbn}\to a$, we have $a_i\in V$ for sufficiently large $i$, which is a contradiction.
\end{proof}

Extending Proposition \ref{compactonepiece} to the non-compact case is not necessary for our purposes but seems possible with a more technical argument. We prove the following consequence of Proposition \ref{compactonepiece} in significant generality since it will be required for a situation in which non-metrizable quotient spaces appear.

\begin{corollary}\label{isocor}
Let $B$ be a path-connected Hausdorff space and $A\subseteq B$ be a closed, path-connected, subspace such that every simple closed curve in $B$ is contained in $A$. Then the inclusion $i:A\to B$ induces an isomorphism on fundamental groups.
\end{corollary}

\begin{proof}
Let $\beta:S^1\to B$ be a loop based at a point in $A$. Then $Y=\im(\beta)$ is a Peano continuum. We claim that $Q=Y\cap A$ is path-connected. If $(a,b)$ is a connected component of $U=\beta^{-1}(B\backslash A)$, then we must have $\beta(a)=\beta(b)$ or previously used arguments would allow us to construct a simple closed curve in $B$ that is not contained in $A$. Define $\alpha:\ui\to A$ so that $\alpha(t)=\beta(t)$ if $t\in \ui\backslash U$ and if $t\in (a,b)$ for some connected component $(a,b)$ of $U$, then $\alpha(t)=\beta(a)$. Then $\beta$ is continuous because $\alpha$ is continuous and $\im(\alpha)=Q$. We conclude that $Q$ is path-connected. It follows that $\{Q\}$ is a disjoint tree-grading on $Y$. By Proposition \ref{compactonepiece}, $Q$ is a deformation retract of $Y$ where the resulting retraction $r:Y\to Q$, which is a homotopy equivalence, gives $r\circ\beta=\alpha$. Thus $\beta$ is homotopic to $\alpha$ in $B$, proving that $i:A\to B$ is $\pi_1$-surjective.

Let $\gamma:S^1\to A$ be a loop such that there exists a map $g:\bbd\to B$ with $g|_{S^1}=\gamma$. Then $Y=\im(g)$ is a Peano continuum. Choose a surjective path $\alpha:\ui\to Y$. Using the argument from the previous paragraph it follows that if $Q=Y\cap A$, then $\{Q\}$ is a disjoint tree-grading on $Y$. Again, the retraction $r:Y\to Q$ from Lemma \ref{retractionlemma} is a homotopy equivalence. The map $r\circ g:\bbd\to A$ satisfies $r\circ g|_{S^1}=r\circ\gamma=\gamma$ and thus $\gamma$ is inessential in $A$. This proves that $i:A\to B$ is $\pi_1$-injective.
\end{proof}

\begin{corollary}\label{isocor2}
Let $(B,b_0)$ be a based path-connected Hausdorff space and suppose $A_1,A_2,\dots ,A_m$ are pairwise-disjoint, closed, path-connected subspaces of $B$ such that every simple closed curve in $B$ is contained in $A_i$ for some $i\in\{1,2,\dots,m\}$. Then there exists paths $\alpha_i:\ui\to B$ from $b_0$ to a point $a_i\in A_i$ such that the inclusion maps $A_i\to B$ induce an isomorphism $\theta:\ast_{i=1}^{m}\pi_1(A_i,a_i)\to \pi_1(B,b_0)$ given by $\theta([\gamma])=[\alpha_i\cdot\gamma\cdot\alpha_{i}^{-1}]$ for $[\gamma]\in \pi_1(A_i,a_i)$.

Moreover, if $\beta_i:\ui\to B$ are arbitrarily chosen paths from $b_0$ to $a_i$, then the homomorphism $\theta':\ast_{i=1}^{m}\pi_1(A_i,a_i)\to \pi_1(B,b_0)$ given by $\theta'([\gamma])=[\beta_i\cdot\gamma\cdot\beta_{i}^{-1}]$ for $[\gamma]\in \pi_1(A_i,a_i)$ is injective.
\end{corollary}

\begin{proof}
For each $i\in\{1,2,\dots,m\}$, let $V_i$ be the connected component of $B\backslash \bigcup_{j\neq i}A_j$ containing $A_i$. Applying Corollary \ref{isocor} to the pair $(V_i,A_i)$, we see that the inclusion $(A_i,a_i)\to (V_i,a_i)$ induces an isomorphism on fundamental groups. The sets $V_i$ cover $B$. Moreover, $V_i\cap V_j$ is uniquely arcwise-connected and therefore simply connected whenever $i\neq j$ and $V_i\cap V_j\neq \emptyset$. Reorder the sets $V_i$ so that $W_k=V_1\cup V_2\cup \cdots \cup V_k$ is connected for each $k\in\{1,2,\dots,m\}$. Now, the recursive application of the standard van Kampen Theorem to the sets $W_k$ and $V_{k+1}$ for $k\in\{1,2,\dots,m-1\}$ shows that $\theta$ is an isomorphism for some paths $\alpha_i$ resulting from the recursion and a final basepoint-change conjugation.

For the second statement, we recall the following property of free products: if $h_1,h_2,\dots, h_m\in G=\ast_{i=1}^{m}\pi_1(A_i,a_i)$, then the homomorphism $\phi:G\to G$ determined by mapping $\phi(g)=h_igh_{i}^{-1}$ for $g\in \pi_1(A_i,a_i)$ is injective. For given paths $\beta_i$, we define $\phi$ by setting $h_i=\theta^{-1}([\beta_i\cdot\alpha_{i}^{-1}])$ for $i\in\{1,2,\dots,n\}$. Since $\theta '=\theta\circ\phi$, it follows that $\theta '$ is injective.
\end{proof}

\subsection{Parameterizations of disjointly tree-graded spaces}\label{sectionparamterization}

Here, we give a structural characterization of disjointly tree-graded spaces that will be crucial in the proof of our main result.

\begin{definition}\label{defparamterization}
Let $(X,\scrp)$ be a disjointly tree-graded space. A \textit{parameterization} of $(X,\scrp)$ is a triple $(T,V,q)$ where $T$ is a topological tree, $V\subseteq T$ is a closed totally disconnected subset, and $q:X\to T$ is a continuous surjection such that $\scrp=\{q^{-1}(v)\mid v\in V\}$ and $|q^{-1}(t)|=1$ if $t\in T\backslash V$.
\end{definition}

First, we show that every disjointly tree-graded space admits a parametrization.

\begin{lemma}\label{paramterizationlemma}
Let $(X,\scrp)$ be a disjointly tree-graded space, $T$ be the quotient space of $X$ by collapsing each piece to a point, and $q:X\to T$ be the quotient map. Then $T$ is Hausdorff, path-connected, locally path-connected and $V=q(\pc(X))$ is closed and totally disconnected.
\end{lemma}

\begin{proof}
Since the properties of being path-connected and locally path-connected space are closed under forming quotient spaces, $T$ has these properties. We check that $T$ is Hausdorff. By definition, $\pc(X)$ is closed in $X$ and $q$ is defined so that $\pc(X)$ is saturated with respect to $q$. Hence, $V$ is closed in $T$. Moreover, $q$ is bijective on the open set $\tree(X)=q^{-1}(T\backslash V)$ and so $q$ maps $\tree(X)$ homeomorphically onto $T\backslash V$. In particular, $T\backslash V$ is a disjoint union of open topological $\bbr$-trees. Suppose $a,b$ are distinct points of $T$. Let $x\in q^{-1}(a)$ and $y\in q^{-1}(b)$. Since $x$ and $y$ do not lie in a single piece of $X$, Proposition \ref{separateprop} gives a point $z\in \tree(X)$ such that $x$ and $y$ lie in distinct components of $X\backslash\{z\}$. Say that $C_x$ and $C_y$ are these connected components respectively. Then $C_x$ and $C_y$ are disjoint, open, and saturated with respect to $q$ (since pieces are path-connected) and thus $q(C_x)$ and $q(C_y)$ are disjoint open neighborhoods of $a$ and $b$ in $T$. Moreover, if $x$ and $y$ lie in distinct pieces of $X$, then $q(C_x)\cap V$ and $q(C_y)\cap V$ are disjoint clopen neighborhoods in $V$. This completes the proof that $T$ is Hausdorff and $V$ is totally disconnected.
\end{proof}

\begin{lemma}\label{arcliftinglemma}
Let $(X,\scrp)$ be a disjointly tree-graded space and $q:X\to T$ be the quotient map collapsing each piece to a point. For every arc $A\subseteq T$, there exists an arc $B\subseteq X$ such that $q(B)=A$.
\end{lemma}

\begin{proof}
Since $V$ is totally disconnected, note that $A\cap (T\backslash V)\neq \emptyset$. Let $\alpha:\ui\to T$ be an injective path parameterizing $A$ and set $a_i=\alpha(i)$ for $i\in\{0,1\}$. Let $\beta:\ui\to X$ be an injective path with $\beta(i)\in q^{-1}(a_i)$ for $i\in\{0,1\}$ and set $B=\im(\beta)$.

Let $U=\alpha^{-1}(T\backslash V)$ and $M=\beta^{-1}(\tree(X))=(q\circ \beta)^{-1}(T\backslash V)$. We first prove that $\alpha(U)=q(\beta(M))$. Let $u\in U$ and $a=\alpha(u)\in T\backslash V$. Then there is a unique $b\in \tree(X)$ with $q(x)=a$. If $b\notin B$, then $B$ must lie in a component $C$ of $X\backslash \{b\}$. But then $A$ lies in the component $q(C)$ of $T\backslash \{b\}$, which cannot occur. Therefore, $b\in B$. Then there exists $m\in M$ such that $\beta(m)=b$ and thus $a=q(\beta(m))\in q(\beta(M))$. Thus $\alpha(U)\subseteq q(\beta(M))$. For the other inclusion, let $m\in M$. We have $b=\beta(m)\in \tree(X)$ and $a=q(b)\in T\backslash V$. If $a\notin A$, then $A$ lies in a single connected component $D$ of $T\backslash \{a\}$. Then $q^{-1}(a_0)\cup q^{-1}(a_1)$ lies in the connected component $C$ of $X\backslash \{b\}$ satisfying $q(C)=D$. Since $\{b_1,b_2\}\subseteq C$ and $b\in B\cap \tree(X)$ violates Proposition \ref{separateprop} (which applies since $\beta(0)$ and $\beta(1)$ do not lie in a single piece of $X$). Therefore, $a\in A$. In particular, $q(\beta(m))=q(b)=a=\alpha(u)$ for unique $u\in U$. This proves the other inclusion.

Since $U$ is dense in $\ui$, we have $A=\alpha(\ov{U})$ and since $\ui$ is compact and $\alpha$ is continuous, we have $\alpha(\ov{U})=\ov{\alpha(U)}$. Thus $A=\ov{\alpha(U)}=\ov{q(\beta(M))}\subseteq q(B)$. We also have $q(\beta(\ov{M}))\subseteq \ov{q(\beta(M))}=\ov{\alpha(U)}\subseteq A$. If $(s,t)$ is a component of $\ui\backslash \ov{M}$, then $\beta$ maps $[s,t]$ into a piece of $X$ and $q(\beta([s,t]))=q(\beta(s))\in q(\beta(\ov{M}))\subseteq A$. This completes the proof that $q(B)=A$.
\end{proof}

\begin{lemma}\label{paramterizationlemma2}
Let $(X,\scrp)$ be a disjointly tree-graded space. If $T$ is the quotient space of $X$ given by collapsing each piece to a point, then $T$ is a topological tree.
\end{lemma}

\begin{proof}
According to Lemma \ref{paramterizationlemma}, it suffices to show that $T$ does not contain a simple closed curve. Suppose that $C\subseteq T$ is a simple closed curve. Let $A_1,A_2$ be arcs such that $C=A_1\cup A_2$ and such that $A_1\cap A_2=\{a_1,a_2\}$ is a two-point set. As before, let $q:X\to T$ denote the quotient map. By Lemma \ref{arcliftinglemma}, there exist arcs $B_1,B_2\subseteq X$ such that $q(B_i)=A_i$ for $i\in\{1,2\}$. For $i,j\in\{1,2\}$, let $b_{i,j}$ be the endpoint of $B_i$ in $B_i\cap q^{-1}(a_j)$. In the case that $a_1\in V$, an initial segment of $B_1$ and $B_2$ may lie in the piece $q^{-1}(a_1)$, By deleting this segment, we may assume that $B_i\cap q^{-1}(a_1)=\{b_{i,1}\}$ for $i\in\{1,2\}$. Making the same adjustment when $a_2\in V$, we may assume that $B_i\cap q^{-1}(a_2)=\{b_{i,2}\}$ for $i\in\{1,2\}$. Since the interiors of $A_1$ and $A_2$ are disjoint, our adjustments to the ends of $B_i$ ensures that the interiors of $B_1$ and $B_2$ are disjoint. Fix $j\in\{1,2\}$. If $b_{1,j}=b_{2,j}$, let $C_j=\{b_{1,j}\}$. On the other hand, if $b_{1,j}\neq b_{2,j}$, then $q^{-1}(a_j)$ is a piece of $X$ and we may find an arc $C_j\subseteq q^{-1}(a_j)$ with endpoints $b_{1,j}$ and $b_{2,j}$. With $C_1$ and $C_2$ defined, $D=B_1\cup B_2\cup C_1\cup C_2$ is a simple closed curve in $X$ for which $q(D)=C$. Thus $D$ must lie in a single piece of $X$, which forces $C$ to be a single point; a contradiction.
\end{proof}

\begin{example}
The topological tree $T$ constructed as the quotient of $X$ in the previous two lemmas need not be a topological $\bbr$-tree. For example let $X=\bbr\times\{0\}\cup \bbz\times [0,1]$ having single piece $P=\bbr\times\{0\}$. Then the quotient space $T=X/P$ is an infinite wedge of arcs with the weak topology and is not first countable.
\end{example}

\begin{example}
The quotient map $q:X\to T$ constructed in the previous two lemmas need not be a closed map. For instance, if $X=[0,1]\times\{0\}\cup\bigcup_{n\in\bbn}\{1/n\}\times[0,1]$ with pieces $P_n=\{1/n\}\times[0,1]$, then we may identify $T=[0,1]\times\{0\}$, $V=\{0\}\cup\{1/n\mid b\in\bbn\}$, and $q:X\to T$ with vertical projection. In this case, $q$ is not a closed map.
\end{example}

Next, we show that a disjointly tree-graded space can be completely recovered from a given parameterization.

\begin{lemma}\label{monotonelemma2}
Let $T$ be a uniquely arcwise-connected Hausdorff space, $V\subseteq T$ be a closed totally disconnected subset, and let $q:Y\to T$ be a map from a space $Y$ with the following properties:
\begin{enumerate}
\item every fiber of $q$ is path-connected,
\item $q$ is injective on $Y\backslash q^{-1}(V)$,
\end{enumerate}
If $\alpha:\ui\to Y$ is an injective path, then $q\circ\alpha:\ui\to T$ is monotone. Moreover, if $\alpha:S^1\to Y$ is an injective loop, then $q\circ\alpha:S^1\to T$ is constant.
\end{lemma}

\begin{proof}
Supposing $\alpha$ is injective, let $a\in T$. If $a\in T\backslash V$, then $\alpha^{-1}(q^{-1}(a))$ is contains a single point whenever it is non-empty. If $a\in V$, then $C=q^{-1}(a)$ is path-connected by assumption. Suppose that $|\alpha^{-1}(C)|>1$. Let $s=\min(\alpha^{-1}(C))$ and $t=\max(\alpha^{-1}(C))$. We claim that $\alpha^{-1}(C)=[s,t]$. Supposing otherwise, $\alpha^{-1}(C)$ is a proper closed subset of $[s,t]$. Choose a connected component $(c,d)$ of $[s,t]\backslash \alpha^{-1}(C)$. Then $\alpha|_{[c,d]}:[c,d]\to Y$ is a path with $\alpha^{-1}(C)=\{c,d\}$ and $q\circ\alpha|_{[c,d]}$ is a loop in $T$ based at $a$ that maps $(c,d)$ into a connected component $T'$ of $T\backslash\{a\}$. Since $T'$ is uniquely arcwise-connected and $V$ is totally disconnected, we may find $c<c'<d'<d$ such that $q\circ \alpha(c')=q\circ\alpha(d')\in T'\backslash V$. Since $q:Y\backslash q^{-1}(V)\to T\backslash V$ is injective, we have $\alpha(c')=\alpha(d')$; a contradiction.

For the second statement, suppose $\alpha:S^1\to Y$ is injective. We may cover $S^1$ by various arcs and apply the first part of the lemma to see that $q\circ\alpha$ is monotone. Additionally, since $T$ is Hausdorff, $q\circ \alpha$ is quotient onto it's image. Since a Hausdorff quotient of $S^1$ by a monotone quotient map is a simple closed curve and $T$ is uniquely arcwise-connected, $q\circ\alpha$ must be constant.
\end{proof}

\begin{lemma}\label{datalemma}
Suppose $X$ is a path-connected, locally path-connected metrizable space, $T$ is a topological tree, $V\subseteq T$ is a closed totally disconnected subspace, and $q:X\to T$ is a quotient map with path-connected fibers and such that the fiber of each point in $T\backslash V$ is degenerate. Then $\scrp=\{q^{-1}(v)\mid v\in V\}$ is a disjoint tree-grading on $X$ for which $(T,V,q)$ is a parameterization.
\end{lemma}

\begin{proof}
The elements of $\scrp$ are pairwise-disjoint by construction. Since $V$ is closed in $T$, $q^{-1}(V)=\bigcup\scrp$ is closed in $X$. Since $V$ is totally disconnected and each element of $\scrp$ is path-connected, it follows that the elements of $\scrp$ are the path-components of $\bigcup\scrp$. Moreover, since $T$, $V$, and $q:X\to T$ meet the hypotheses of Lemma \ref{monotonelemma2}, we see that if $\alpha:S^1\to X$ parameterizes a simple closed curve, then $q\circ\alpha:S^1\to T$ must be constant. Thus $\im(\alpha)$ must lie in some non-degenerate fiber of $q$, namely, some element of $\scrp$.
\end{proof}

The above lemmas combine to give the following statement, which shows that one could give an equivalent definition of disjointly tree-graded spaces in terms of parameterizations.

\begin{theorem}\label{paramterizationtheorem}
Suppose $X$ is a path-connected, locally path-connected metrizable space. Every disjoint tree-grading $\scrp$ on $X$ admits a parameterization $(T,V,q)$ for $(X,\scrp)$ where $q$ is a topological quotient map. Conversely, if $(T,V,q)$ is a triple where $T$ is a topological tree, $V\subseteq T$ is a closed totally disconnected subspace, and $q:X\to T$ is a continuous surjection with path-connected fibers and such that the fiber of each point in $T\backslash V$ is degenerate, then $\scrp=\{q^{-1}(v)\mid v\in V\}$ is a disjoint tree-grading on $X$ for which $(T,V,q)$ is a parameterization.
\end{theorem}

\begin{remark}
Lemma \ref{paramterizationlemma2} shows that we may always choose the map $q:X\to T$ in a parameterization $(T,V,q)$ to be a topological quotient map if we wish to do so. In Section \ref{sectionmetricquotient}, we will show that we may also choose $T$ to be a topological $\bbr$-tree and $q$ to be non-expansive. Since we do not require that $q$ is always a quotient map, the topological tree $T$ in a parameterization need not be uniquely determined up to homeomorphism. The ability to recognize a disjoint tree-grading on a space without making specific demands on $T$ and $q$ will be important later on.
\end{remark}

\section{Grade-preserving maps and graded subspaces}\label{sectiongradedsubspaces}

The following is the appropriate notion of a morphism of tree-graded spaces.

\begin{definition}\label{morphismdef}
Let $(X_1,\scrp_1)$ and $(X_2,\scrp_2)$ be disjointly tree-graded spaces. A continuous function $f:X_1\to X_2$ is a \textit{grade-preserving map} if $f(\pc(X_1))\subseteq \pc(X_2)$. We may denote such a map as a morphism $f:(X_1,\scrp_1)\to(X_2,\scrp_2)$ of pairs.
\end{definition}

Note that we do not require that a grade-preserving map be non-expansive. This is the reason why a topological quotient spaces and required in various parts of this paper.

Suppose $f:(X_1,\scrp_1)\to(X_2,\scrp_2)$ is a grade-preserving map. Then for every piece $P\in \scrp_1$, there is some $Q\in \scrp_2$ such that $f(P)\subseteq Q$. Moreover, if these disjointly tree-graded spaces have respective parameterizations $(T_1,V_1,q_1)$ and $(T_2,V_2,q_2)$ where $q_1$ is chosen to be a quotient map, then there exists a unique map $g:T_1\to T_2$ such that $g\circ q_1=q_2\circ f$ and $g(V_1)\subseteq V_2$, i.e. the following diagram commutes where $i_1,i_2$ denote inclusion maps.
\[\xymatrix{
X_1 \ar[r]^-{q_1} \ar[d]_-{f} & T_1 \ar[d]^-{g}  & V_1 \ar[d]^-{g|_{V_1}} \ar[l]_-{i_1} \\
X_2 \ar[r]_-{q_2} & T_2  & V_2 \ar[l]^-{i_2}
}\]
Since $q_1$ may always be chosen to be quotient (Theorem \ref{paramterizationtheorem}), a grade-preserving map will always give rise to such a commutative diagram.

\begin{definition}\label{gradedsubspacedef}
Let $(X,\scrp)$ be a disjointly tree-graded space. A \textit{graded subspace} of $(X,\scrp)$ is a disjointly tree-graded space $(Y,\scrq)$ where $Y$ is a path-connected, locally path-connected subspace of $X$ and $\mathscr{Q}=\{P\cap Y\mid P\in\scrp\text{ s.t. }P\cap Y\neq \emptyset\}$.
\end{definition}

\begin{remark}\label{subspacewelldefinedremark}
The set $\scrq$ in Definition \ref{gradedsubspacedef} is always a well-defined tree-grading on $Y$. Since $Y$ is assumed to be path-connected, it follows from Lemma \ref{retractionlemma} that the retraction $r:X\to P$ satisfies $r(Y)=P\cap Y$ whenever $P\in\scrp$ and $P\cap Y\neq \emptyset$. Thus each element of $\scrq$ is path-connected. The other conditions are straightforward to check.
\end{remark}


\begin{definition}
Let $(Y,\scrq)$ be a graded subspace of $(X,\scrp)$.
\begin{enumerate}
\item We say that $(Y,\mathscr{Q})$ is a \textit{full} graded subspace of $(X,\scrp)$ if $\scrq\subseteq \scrp$.
\item We say that $(Y,\mathscr{Q})$ is a \textit{sectional} graded subspace of $(X,\scrp)$ if whenever $Q\in\scrq$, $P\in\scrp$, and $Q\subseteq P$, we have either $|Q|=1$ or $Q=P$.
\end{enumerate}
\end{definition}

%

A special case of Lemma \ref{retractionlemma} is the following lemma.

\begin{lemma}\label{retractionlemma3}
If $(Y,\scrq)$ is a sectional graded subspace of $(X,\scrp)$ where $Y$ is closed in $X$, then there is a canonical non-expansive retraction $r:X\to Y$.
\end{lemma}


In the next proposition, we show that if $(Y,\scrq)$ is a graded subspace of $(X,\scrp)$, then we can always enlarge $Y$ by replacing each non-degenerate piece of $Y$ with the piece of $X$ that contains it.

\begin{proposition}\label{expansionprop}
Suppose that $(Y,\scrq)$ is a graded subspace of $(X,\scrp)$ and $\scrr\subseteq \scrq$. Define $\scrp_0=\{P\in \scrp\mid \exists R\in\scrr\text{ s.t. }R\subseteq P\}$ and let $Y_0=Y\cup \bigcup\scrp_0$ and $\scrq_0=(\scrq\backslash \scrr)\cup \scrp_0$. Then $\scrq_0$ is a disjoint tree-grading on $Y_0$ such that we have graded subspace inclusions $(Y,\scrq)\subseteq (Y_0,\scrq_0)\subseteq (X,\scrp)$. Moreover, if $Y$ is closed in $X$, then $Y_0$ is closed in $X$.
\end{proposition}

\begin{proof}

It is straightforward to check that $(Y_0,\scrq_0)$ is a graded subspace of $(X,\scrp)$. For the last statement, suppose $Y$ is closed in $X$. Consider a convergent sequence $\{a_n\}_{n\in\bbn} \to x$ in $X$ where $a_n\in Y_0$ for all $n\in\bbn$. If infinitely many $a_n$ lie in $Y$, then $x\in Y\subseteq Y_0$. So we may assume that $a_n\in (\bigcup\scrp_0)\backslash Y$ for all $n\in\bbn$. If infinitely many $a_n$ lie in some $P\in\scrp_0$, then $x\in P\subseteq Y_0$ since $P$ is closed in $X$. So we may focus on the case where $a_n\in P_n$ for distinct $P_1,P_2,P_3,\dots\in \scrp_0$. Let $r_n:X\to P_n$ denote the canonical retraction from Lemma \ref{retractionlemma}. Let $c_{n+1}=r_{n+1}(P_n)$ and $b_n=r_n(P_{n+1})$ for all $n\in\bbn$. By Corollary \ref{twocor}, any injective path $\gamma:\ui\to X$ with $\gamma(0)\in P_n$ and $\gamma(1)\in P_{n+1}$ must pass through the points $b_n$ and $c_{n+1}$. Since $P_n\in\scrp_0$, there exists $R_n\in\scrr$ with $R_n\subseteq P_n$. Thus we may find point $y_n\in P_{n}\cap Y$ for each $n\in\bbn$. Since there exists an injective path $\alpha_n:\ui\to Y$ with $\alpha_n(0)=y_n$ and $\alpha_n(1)=y_{n+1}$, we must have $\{b_n,c_{n+1}\}\subseteq \im(\alpha_n)\subseteq Y$ for all $n\in\bbn$. We claim that $\{b_n\}_{n\in\bbn}\to x$ in $X$. Let $U$ be a path-connected open neighborhood of $x$ in $U$ and find $N\in\bbn$ such that $a_n\in U$ for all $n\geq N$. For $n\geq N$, let $\beta_n:\ui\to U$ be an injective path with $\beta_n(0)=a_{n}$ and $\beta_n(1)=a_{n+1}$. Then we have $b_n\in \im(\beta_n)\subseteq U$ for all $n\geq N$, completing the proof. Since $\{b_n\}_{n\in\bbn}\to x$ where $b_n\in Y$ for all $n$ and where $Y$ is closed, we have $x\in Y\subseteq Y_0$, completing the proof.
\end{proof}

\begin{definition}\label{expansiondef}
If $(Y,\scrq)$ is a graded subspace of $(X,\scrp)$ and $\scrr\subseteq \scrq$, we refer to the disjointly tree-grade $(Y_0,\scrq_0)$ constructed in Proposition \ref{expansionprop} as the \textit{expansion of $(Y,\scrq)$ in $(X,\scrp)$ at $\scrr$}.
\end{definition}

\begin{remark}\label{expansionremark}
Let $(Y_0,\scrq_0)$ be the expansion of $(Y,\scrq)$ in $(X,\scrp)$ at $\scrr$.
\begin{enumerate}
\item If $\scrr$ contains all of the non-degenerate pieces of $Y$, then $(Y_0,\scrq_0)$ is a sectional graded subspace of $(X,\scrp)$.
\item If $\scrr$ contains exactly the non-degenerate pieces of $Y$, then the inclusion $Y\to Y_0$ determines a bijection on non-degenerate pieces.
\item If $\scrq\backslash\scrp\subseteq \scrr$, then $(Y_0,\scrq_0)$ is a full graded subspace of $(X,\scrp$).
\end{enumerate}
\end{remark}

\begin{lemma}\label{pclemma}
If $(X,\scrp)$ is a disjointly tree-graded space where $X$ is a Peano continuum, then
\begin{enumerate}
\item $X$ has at most countably many non-degenerate pieces of null diameter,
\item $\tree(X)$ has at most countably many connected components.
\end{enumerate}
\end{lemma}

\begin{proof}
Let $d$ be a metric inducing the topology of $X$. Let $\scru$ be any finite open cover of $X$ by path-connected open sets. We claim that all but finitely many pieces lie in some element of $\scru$. Supposing otherwise, we have an infinite sequence of distinct pieces $P_1,P_2,P_3,\dots\in\scrp$ such that each $P_i$ does not lie in any single element of $\scru$. If $\lambda>0$ is a Lebesgue number for $\scru$, then $\diam(P_i)\geq \lambda$ for all $i\geq 1$. Find $a_i,b_i\in P_i$ with $d(a_i,b_i)\geq \lambda$. Since $X$ is compact, we may replace $\{P_i\}_{i\in\bbn}$ with a subsequence so that $\{a_i\}_{i\in\bbn}\to a$ and $\{b_i\}_{i\in\bbn}\to b$ for some $a,b\in X$. Then $d(a,b)\geq \lambda/3$. Find path-connected open sets $U_a,U_b$ such that $a\in U_a\subseteq B_{\lambda/6}(a)$, $b\in U_b\subseteq B_{\lambda/6}(b)$. Fix $i$ sufficiently large so that $a_i,a_{i+1}\in U_a$ and $b_i,b_{i+1}\in U_b$. Now we can find injective paths $\alpha:\ui\to U_a$ from $a_i$ to $a_{i+1}$ and $\beta:\ui\to U_b$ from $b_i$ to $b_{i+1}$. Let $r_i:X\to P_i$ be the retraction constructed in Lemma \ref{retractionlemma}. Lemma \ref{retractionlemma} implies that $r_i(P_{i+1})$ is a single point $z\in P_i$. Moreover, since $\alpha$ and $\beta$ both start in $P_i$ and end in $P_{i+1}$, we have $z\in \im(\alpha)\cap \im(\beta)$. However, this is a contradiction since $\alpha$ and $\beta$ were constructed to have disjoint images. We conclude that for any finite open cover $\scru$ of $X$ by path-connected open sets, all but finitely many pieces of $X$ lie in some element of $\scru$.

Now, find a sequence $\scru_n$ of finite open covers of $X$ by path-connected open sets of diameter less than $2^{-n}$ and Lebesgue numbers $\lambda_n>0$ for $\scru_n$ such that $\{\lambda_n\}_{n\in\bbn}\to 0$. Let $\scrp_n$ be the set of a pieces $P\in\scrp$, which are not contained in some element of $\scru_n$. By the previous paragraph, $\scrp_n$ is finite for all $n\in\bbn$ and $\scrp_{\infty}=\bigcup_{n\in\bbn}\scrp_n$ is a countable subset of $\scrp$. Note that if $P\in \scrp$ is a non-degenerate piece of $X$, then $\diam(P)>0$ and there exists $n$ sufficiently large so that $P\in \scrp_n$. Thus the countable set $\scrp_{\infty}$ is the set of non-degenerate pieces of $X$ and, when enumerated, the diameters of these pieces form a null sequence. This completes the proof of (1).

Since $\tree(X)$ is a disjoint union of pairwise-disjoint connected open sets and $X$ is separable, $\tree(X)$ must have countably many connected components. Therefore, (2) holds.
\end{proof}

\begin{lemma}\label{expansionofpclemma}
Let $(X,\scrp)$ be a disjointly tree-graded space, $K$ be a Peano continuum and $f:K\to X$ be a map. Then there exists a sectional graded-subspace $(Y,\scrq)$ of $(X,\scrp)$ satisfying the following:
\begin{enumerate}
\item $\im(f)\subseteq Y$,
\item $Y$ is closed in $X$.
\item $Y$ has countably many non-degenerate pieces,
\item $\tree(Y)$ has countably many connected components,
\item $f$ is null-homotopic in $Y$ if and only if $f$ is null-homotopic in $X$.
\end{enumerate}
\end{lemma}

\begin{proof}
Since $C=\im(f)$ is a Peano continuum, the collection $\scrc=\{P\cap C\mid P\in\scrp \text{ and }P\cap C\neq\emptyset\}$ is a disjoint tree-grading on $C$ such that $(C,\scrc)$ is a graded subspace of $(X,\scrp)$. Moreover, Lemma \ref{pclemma} implies that $C$ has at most countably many non-degenerate pieces and $\tree(C)$ has countably many connected components. Let $\scrr$ denote the set of non-degenerate pieces of $C$ and let $(Y,\scrq)$ be the expansion of $(C,\scrc)$ in $(X,\scrp)$ at $\scrr$ (recall Definition \ref{expansiondef}). By Part (1) of Remark \ref{expansionremark}, $(Y,\scrq)$ is a section graded subspace of $(X,\scrp)$.

With the definition of $(Y,\scrq)$ complete, note that (1) holds by construction and (2) holds by the second statement of Proposition \ref{expansionprop}. (3) follows from Part (2) of Remark \ref{expansionremark}. (4) follows from the fact that the expansion construction does not affect the tree-portion, i.e. $\tree(C)=\tree(Y)$. For (5), note that any null-homotopy $H:K\times\ui\to Y$ of $f$ in $Y$ also gives a null-homotopy of $f$ in $X$ (since $Y\subseteq X$). Since we have established that $Y$ is closed in $X$ and $(Y,\scrq)$ is a sectional graded subspace of $(X,\scrp)$, Lemma \ref{retractionlemma3} applies, giving a canonical retraction $r:X\to Y$. Hence, if $H:K\times\ui\to X$ is a null-homotopy of $f$ in $X$, then $r\circ H$ is a null-homotopy of $f$ in $Y$.
\end{proof}

\section{Metric quotients of disjointly tree-graded spaces}\label{sectionmetricquotient}

Fix a disjointly tree-graded space $(X,\scrp)$ and let $\scrq\subseteq \scrp$. Let $X_{\scrq}$ denote the quotient set obtained by identifying each $P\in\scrp\backslash \scrq$ to a single point and let $\Gamma_{\scrq}:X\to X_\scrq$ denote the surjection making these identifications. For each $P\in\scrp\backslash\scrq$, let $y_P$ denote the point in $X_{\scrq}$, which is the image of $P$ under $\Gamma_{\scrq}$. For the sake of simplicity, we identify $X\backslash \bigcup\{P\mid P\in \scrp\backslash \scrq\}$ with $X_{\scrq}\backslash\{y_P\mid P\in\scrp\backslash\scrq\}$ so that $\Gamma_{\scrq}$ is the identity on $X\backslash \bigcup\{P\mid P\in \scrp\backslash \scrq\}$ and so that $\Gamma_{\scrq}(P)=y_P$ when $P\in\scrp\backslash \scrq$.

Because we do not demand that grade-preserving maps are non-expansive, we have occasional need to view $X_{\scrq}$ the topological quotient space of $X$. We write $X_{\scrq}^{qt}$ to denote $X_{\scrq}$ equipped with the quotient topology inherited from $\Gamma_{\scrq}$. If $\scrq=\emptyset$, then $X_{\scrq}^{qt}$ is precisely the topological tree from Section \ref{sectionparamterization}. The same argument used in the proof of Lemma \ref{paramterizationlemma} shows that $X_{\scrq}^{qt}$ is Hausdorff for any $\scrq\subseteq \scrp$.

We also wish to endow $X_{\scrq}$ with a disjoint tree grading $\scrp_{\scrq}=\scrq\cup\{\{y_P\}\mid P\in\scrp\backslash\scrq\}$ for which the map $\Gamma_{\scrq}:X\to X_{\scrq}$ would be grade-preserving. However, $X_{\scrq}^{qt}$ need not be first countable. Thus we must explicitly define a path-diameter metric $\rho$ on $X_{\scrq}$ such that $\Gamma_{\scrq}:X\to X_{\scrq}$ is a metric quotient map in the following sense.

\begin{definition}\label{defmetricquotientmap}
A function $q:(X,d)\to (Y,\rho)$ of metric spaces is a \textit{metric quotient map} if $q$ is a non-expansive surjection with the property that for every non-expansive map $g:(X,d)\to (X',d')$ where $g$ is constant on the fibers of $q$, there exists a unique non-expansive map $f:(Y,\rho)\to (X',d')$ such that $f\circ q=g$.
\end{definition}

This section is devoted to proving the following theorem.

\begin{theorem}\label{quotientmetrictheorem}
For any $\scrq\subseteq \scrp$, there exists a path-diameter metric $\rho$ on $X_{\scrq}$ so that when $X_{\scrq}$ is given the topology induced by $\rho$, $\scrp_{\scrq}$ is a disjoint tree-grading on $X_{\scrq}$ and $\Gamma_{\scrq}:(X,d)\to (X_{\scrq},\rho)$ is a grade-preserving, metric quotient map.
\end{theorem}

Since $\scrq$ is fixed, we simplify our notation by writing $Y$ for $X_{\scrq}$ and $\Gamma:X\to Y$ for the map $\Gamma_{\scrq}$. Let $d$ be a path-diameter metric inducing the topology of $X$. We take $\rho:Y\times Y\to \bbr$ be the standard quotient pseudometric \cite{MetricGeometry} on $Y$ inherited from $(X,d)$ by the map $q$: $\rho(y,z)$ is the infimum of all sums $\sum_{i=1}^{n}d(a_{i},b_i)$ taken over the set of all finite chains $a_1,b_1,a_2,b_2,\cdots,a_n,b_n$ in $X$ where $\Gamma(a_1)=y$, $\Gamma(b_n)=z$, and $\Gamma(b_i)=\Gamma(a_{i+1})$ for $1\leq i\leq n-1$.

\begin{proposition}
The quotient pseudometric $\rho$ on $Y$ is a metric.
\end{proposition}

\begin{proof}
We check that $\rho$ is positive definite. Suppose $y$ and $z$ are distinct points in $Y$. Then there exists a point $s_1\in \tree(X)$ such that $\Gamma^{-1}(y)$ and $\Gamma^{-1}(z)$ lie in distinct connected components of $X\backslash\{s_1\}$. Since $\tree(X)$ is open in $X$, we may find $\epsilon>0$ such that $\ov{B_{\epsilon}(s_1)}\subseteq \tree(X)\backslash (\Gamma^{-1}(y)\cup \Gamma^{-1}(z))$. Since $d$ is a path-diameter metric, $\ov{B_{\epsilon}(s_1)}$ is uniquely-arcwise-connected. In particular, there is an arc $A$ with endpoints $s_1$ and $s_0\in \partial B_{\epsilon}(s_1)$ such that the image of any path $\alpha:\ui\to X$ from a point in $\Gamma^{-1}(y)$ to a point in $\Gamma^{-1}(z)$ must contain $A$ (passing through $s_0$ first and then $s_1$). Note that $\diam(A)\geq d(s_0,s_1)= \epsilon$.

Consider a finite chain $a_1,b_1,a_2,b_2,\cdots,a_n,b_n$ in $X$ where $\Gamma(a_1)=y$, $\Gamma(b_n)=z$, and $y_i=\Gamma(b_i)=\Gamma(a_{i+1})$ for $1\leq i\leq n-1$. We will show that $\sum_{i=1}^{n}d(a_i,b_i)\geq \epsilon$. Note that if $b_j=a_{j+1}$, then $\sum_{i=1}^{n}d(a_i,b_i)$ is greater than or equal to \[ d(a_1,b_1)+\cdots +d(a_{j-1},b_{j-1})+d(a_{j},b_{j+1})+d(a_{j+2},b_{j+2})+\cdots+d(a_n,b_n)\]
Thus, we may assume that $b_i\neq a_{i+1}$ for all $1\leq i\leq n-1$. Note that this implies that no $a_i$ or $b_i$ lies in $\tree(X)$. Let $\alpha_i:\ui\to X$ be injective paths from $a_i$ to $b_i$ for $1\leq i\leq n$. Since the fibers of $\Gamma$ are path-connected, we may find injective paths $\beta_i:\ui\to \Gamma^{-1}(y_i)$ for $1\leq i\leq n-1$ from $b_i$ to $a_{i+1}$. Let $\alpha:\ui\to X$ be the concatenation path $\alpha_1\beta_1\alpha_2\beta_2\cdots \beta_{n-1}\alpha_n$. According to the first paragraph, $A\subseteq \im(\alpha)$. Since the paths $\beta_i$ lie in $\pc(X)$, there must exist $k\in\{1,2,\dots,n\}$ such that $A\subseteq \im(\alpha_k)$. Moreover, since $d$ is a path-diameter metric and since the image of every path from $a_k$ to $b_k$ must contain $A$, we have $d(a_k,b_k)\geq \diam(A)$. Thus $\sum_{i=1}^{n}d(a_i,b_i)\geq d(a_k,b_k)\geq \epsilon$. We conclude that $\rho(x,y)\geq \epsilon$. Since $\rho$ is positive definite, it is a metric on $Y$.
\end{proof}

\begin{proposition}
$\rho$ is a path-diameter metric on $Y$.
\end{proposition}

\begin{proof}
Suppose to obtain a contradiction that
\[\rho(y,z)<\delta=\inf\{\diam(\im(\beta))\mid \beta\text{ is a path from }y\text{ to }z\}.\] Then there exists a finite chain $a_1,b_1,a_2,b_2,\cdots,a_n,b_n$ in $X$ where $\Gamma(a_1)=y$, $\Gamma(b_n)=z$, and $y_i=\Gamma(b_i)=\Gamma(a_{i+1})$ for $1\leq i\leq n-1$ and such that $\sum_{i=1}^{n}d(a_i,b_i)<\delta$. Since $d$ is a path-diameter metric, there exists paths $\alpha_i:\ui\to X$ from $a_i$ to $b_i$ such that $\sum_{i=1}^{n}\diam(\im(\alpha_i))<\delta$. Setting $\beta_i=\Gamma\circ\alpha_i$ for each $i$, the concatenation $\beta=\beta_1\beta_2\cdots\beta_n$ is a path from $y$ to $z$ in $Y$. Then $\sum_{i=1}^{n}\diam(\im(\alpha_i))<\diam(\im(\beta))$ and so there exists $y_0,y_1\in \im(\beta)$ such that $\sum_{i=1}^{n}\diam(\im(\alpha_i))< \rho(y_0,y_1)$. Find $1\leq j\leq k\leq 1$, $s\in [0,1)$, and $t\in (0,1]$ such that $\Gamma\circ\alpha_j(s)=y_0$ and $\Gamma\circ\alpha_k(t)=y_1$. Let $a_{j}'=\alpha_j(s)$ and $b_{k}'=\alpha_k(t)$. Now $a_{j}',b_{j},a_{j+1},b_{j+2},\dots,a_{k-1},b_{k-1},a_{k},b_{k}'$ is a chain of points in $X$ where $\Gamma(a_j')=y_0$, $\Gamma(b_k')=y_1$, and $\Gamma(b_i)=\Gamma(a_{i+1})$ for $j\leq i\leq k-1$. Additionally, we have
\begin{eqnarray*}
 &&d(a_j',b_j)+\sum_{i=j+1}^{k-1}d(a_i,b_i)+d(a_k,b_k')\\
 &\leq & \diam(\alpha_j([s,1]))+\sum_{i=j+1}^{k-1}\diam(\im(\alpha_i))+\diam(\alpha_k([0,t]))\\
&\leq & \sum_{i=1}^{n}\diam(\im(\alpha_i))\\
&<& \rho(y_0,y_1).
\end{eqnarray*}
However, this contradicts the definition of $\rho$. Note that if $j=k$ we obtain the same contradiction with a simpler statement.
\end{proof}

Our use of the quotient pseudometric $\rho$ on $Y$ ensures that $\Gamma:(X,d)\to (Y,\rho)$ is a metric quotient map. It is clear from the construction of $Y$ that $\Gamma$ is grade-preserving. This completes the proof of Theorem \ref{quotientmetrictheorem}.

\begin{remark}\label{levelmapremark}
If $(X,\scrp)$ is a disjointly tree-graded space and $\scrr\subseteq \scrq\subseteq \scrp$, then the universal property of the metric quotient map $\Gamma_{\scrq}:X\to X_{\scrq}$ induces a metric quotient map $\Gamma_{\scrq,\scrr}:X_{\scrq}\to X_{\scrr}$ such that the following diagram commutes.
\[\xymatrix{
& X \ar[dl]_-{\Gamma_{\scrq}} \ar[dr]^-{\Gamma_{\scrr}}\\
X_{\scrq}  \ar[rr]_-{\Gamma_{\scrq,\scrr}} && X_{\scrr}
}\]
\end{remark}

We conclude this section with two useful consequences of Theorem \ref{quotientmetrictheorem}. First, note that in the case where $\scrq=\scrp$, $X_{\scrp}$ is a topological $\bbr$-tree and $\Gamma_{\scrp}$ is a continuous surjection whose non-degenerate fibers are precisely the non-degenerate pieces of $X$. Combined with Theorem \ref{paramterizationtheorem}, we have the following.

\begin{corollary}
Every disjointly tree-graded space $(X,\scrp)$ admits a parameterization $(T,V,q)$ where $T=X_{\scrp}$ is a topological $\bbr$-tree and $q=\Gamma_{\scrp}$ is a metric quotient map.
\end{corollary}

\begin{lemma}\label{identitypoioneisomorphismlemma}
Let $(X,\scrp)$ be a disjointly tree-graded space, $x_0\in X$, and $\scrf\subseteq \scrp$ be a finite subset of pieces. The continuous identity function $id:X_{\scrf}^{qt}\to X_{\scrf}$ induces an isomorphism on fundamental groups.
\end{lemma}

\begin{proof}
Let $\scrf=\{P_1,P_2,\dots,P_m\}$. Both $X_{\scrf}^{qt}$ and $X_{\scrf}$ meet the hypotheses of Corollary \ref{isocor2}. Moreover, the open covers used in the proof of Corollary \ref{isocor2} are identical for these two spaces. Hence, we may choose a sequence of conjugating paths $\alpha_i:\ui\to X_{\scrf}^{qt}$ for the isomorphism $\ast_{i=1}^{m}\pi_1(P_i,x_i)\to\pi_1(X_{\scrf}^{qt},x_0)$ and use these same paths for the isomorphism $\ast_{i=1}^{m}\pi_1(P_i,x_i)\to\pi_1(X_{\scrf},x_0)$. The result follows.
\end{proof}

%
%

\section{Relatively small contractions of loops: the $1$-$UV_0$ Property}\label{sectionsmallcontractionsofloops}

The next two definitions are properties formalize the idea that ``small" inessential loops contract by ``small" null-homotopies. The first appears in \cite[Lemma 4.1]{CMRZZsmall} and is was named in \cite[Def. 6.6]{BFTestMap}.

\begin{definition}
A space $X$ is $1$-$UV_0$ at $x\in X$ if for open neighborhood $U$ of $x$, there exists a neighborhood $V$ of $x$ such that $V\subseteq U$ and such that for every loop $\alpha:S^1\to V$ that is inessential in $X$, $\alpha$ is also inessential in $U$. We say that $X$ is $1$-$UV_0$ if it $1$-$UV_0$ at all of its points.
\end{definition}

\begin{definition}\label{uniformoneuvzerodef}
A metric space $(X,d)$ is \textit{uniformly $1$-$UV_0$} if for every $\epsilon>0$, there exists $\delta\in (0,\epsilon)$ such that for every inessential loop $\alpha:S^1\to X$ with $\diam(\im(\alpha))<\delta$, there exists a map $g:\bbd\to X$ such that $\diam(\im(g))<\epsilon$ and $g|_{S^1}=\alpha$.
\end{definition}

The $1$-$UV_0$ property is a topological property whereas the uniformly $1$-$UV_0$ property depends on a choice of metric on $X$. All one-dimensional metric spaces are uniformly $1$-$UV_0$ since every null-homotopic loop in such a space contracts in it's own image \cite[Theorem 3.7]{CConedim}.

\begin{example}
For each $k\in\bbn$, let $D_k\subseteq \bbr^2$ be the closed disk of radius $1/3$ centered at $(k,0)$ and set $X=\bigcup_{k\in\bbn}D_k$ with the disjoint union topology. Certainly, $X$ is $1$-$UV_0$. If $d_1$ is the Euclidean metric on $X$, then $(X,d_1)$ is uniformly $1$-$UV_0$. However, one can also define a metric $d_2$ on $X$, which (1) induces the topology of $X$, (2) satisfies $\diam(D_k)=2/3$ for all $k\in\bbn$, and (3) satisfies $\ds\lim_{k\to\infty}\diam(\partial D_k)=0$. For any such metric $d_2$, the metric space $(X,d_2)$ is not uniformly $1$-$UV_0$.
\end{example}

The previous example motivates the following construction of a disjointly tree-graded space $(X,\scrp)$ where $\pc(X)$ is not (uniformly) $1$-$UV_0$.

\begin{example}\label{haexample}
For $n\geq 1$, let $C_n\subseteq \bbr^3$ denote the cone whose base is the circle in the $xy$-plane of radius $\frac{1}{2^{n+1}}$ centered at $(\frac{1}{2^n},0,0)$ and with vertex $(\frac{1}{2^n},0,1)$. Connecting $C_n$ and $C_{n+1}$ with a line segment $A_n$ in the $x$-axis, we let $X=\{(0,0,0)\}\cup \bigcup_{n\in\bbn}A_n\cup C_n$ (see Figure \ref{fig2}). Let $\scrp$ be the set of cones $C_n$ and $\{(0,0,0)\}$. Then $\scrp$ is a disjoint tree-grading on $X$ such that $\pc(X)$ is not uniformly $1$-$UV_0$. If $\scrf\subseteq \scrp$ is finite, it is clear that the resulting quotient $X_{\scrf}$ where all but finitely many cones are collapsed to points, is simply connected. However, $X$ itself has fundamental group isomorphic to that of the harmonic archipelago \cite{hojka} and is uncountable. Thus the conclusion of Theorem \ref{mainthm1} fails for this space.
\end{example}

\begin{figure}[H]
\centering \includegraphics[height=1.8in]{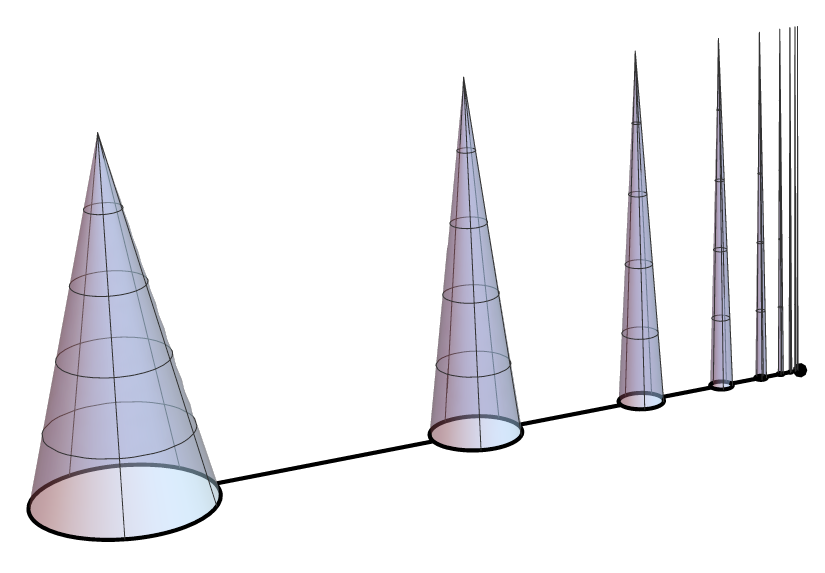}
\caption{\label{fig2}A disjointly tree-graded space $(X,\scrp)$ where $\pc(X)$ is not uniformly $1$-$UV_0$.}
\end{figure}

\begin{lemma}
Let $(X,d)$ be a metric space. If $(X,d)$ is uniformly $1$-$UV_0$, then $X$ is $1$-$UV_0$. Moreover, if $X$ is compact, then the converse holds.
\end{lemma}

\begin{proof}
The first statement is clear. Suppose $X$ is compact metric and $1$-$UV_0$. Let $\epsilon>0$. For each $x\in X$, find $\delta_x\in (0,\epsilon)$ such that every loop $\alpha:S^1\to B_{\delta_x/2}(x)$ that is inessential in $X$ is also inessential in $B_{\epsilon/2}(x)$. Let $\delta>0$ be a Lebesgue number for the cover $\{B_{\delta_x/2}(x)\mid x\in X\}$ of $X$. Suppose $\alpha:S^1\to X$ is an inessential map with $\diam(\im(\alpha))<\delta$. Then $\im(\alpha)\subseteq B_{\delta_x/2}(x)$ for some $x\in X$ and by our choice of cover, we have a map $g:\bbd\to B_{\epsilon/2}(x)$ with $g|_{S^1}=\alpha$. It follows that $\diam(\im(g))<\epsilon$.
\end{proof}

We will need an alternative characterization of the $1$-$UV_0$ property in order to construct continuous maps by gluing together infinite sequences of continuous maps with ``shrinking" images.

\begin{definition}
Let $W,X$ be spaces and $x\in X$. We say a sequence $\{f_k\}_{k\in\bbn}$ of maps $f_k:W\to X$ \textit{converges to} $x$ if for every neighborhood $U$ of $x$ in $X$, we have $\im(f_k)\subseteq U$ for all but finitely many $k\in\bbn$.
\end{definition}

The following provides a useful characterization of the $1$-$UV_0$ property similar to that in \cite[Prop. 6.8]{BFTestMap}. The proof is straightforward and is left as an exercise.

\begin{lemma}\label{uvsequencelemma}
If $X$ is a metrizable space, then the following are equivalent.
\begin{enumerate}
\item $X$ is $1$-$UV_0$
\item for every sequence $\{\alpha_k\}_{k\in\bbn}$ of inessential loops $\alpha_k:S^1\to X$ that converges to a point $x\in X$, there exists a sequence $\{g_k\}_{k\in\bbn}$ of maps $g_k:\bbd\to X$ that converges to $x$ and such that $(g_k)|_{S^1}=\alpha_k$ for all $k\in\bbn$.
\end{enumerate}
Moreover, if $X$ is locally path-connected, then the loops $\alpha_k$ in (2) may be assumed to be based at $x$.
\end{lemma}

%

Using (2) in Lemma \ref{uvsequencelemma} as the definition of $1$-$UV_0$, the next statement has a simple proof.

\begin{corollary}\label{oneuvzeroretractioncor}
If $X$ is a $1$-$UV_0$ metrizable space and $Y\subseteq X$ is a retract of $X$, then $Y$ is $1$-$UV_0$.
\end{corollary}


\begin{lemma}\label{oneuvzerofinitetreelemma}
Let $(X,\scrp)$ be a disjointly tree-graded space with finitely many non-degenerate pieces. Then $X$ is $1$-$UV_0$ if and only if all of it's pieces are $1$-$UV_0$.
\end{lemma}

\begin{proof}
Since each piece of $X$ is a retract of $X$, one direction follows from Corollary \ref{oneuvzeroretractioncor}. For the converse, we may assume that $\scrp=\{P_1,P_2,\dots,P_n\}$ and that $P_i$ is $1$-$UV_0$ for each $i$. We apply Lemma \ref{uvsequencelemma} to show that $X$ is $1$-$UV_0$. Let $\{\alpha_k\}_{k\in\bbn}$ be a sequence of inessential loops $\alpha_k:S^1\to X$ that converges to a point $x$ in $X$.

If $x\in \tree(X)$, then there exists $K\in\bbn$ such that $\im(\alpha_k)\subseteq \tree(X)$ for all $k\geq K$. Since each component of $\tree(X)$ is a topological $\bbr$-tree, whenever $\im(\alpha_k)\subseteq \tree(X)$, there exists a map $g_k:\bbd\to \im(\alpha_k)$ such that $(g_k)|_{S^1}=\alpha_k$. Choosing $g_k$, $k<K$ arbitrarily, we obtain a sequence of maps $g_k:\bbd\to X$ that converges to $x$ and such that $(g_k)|_{S^1}=\alpha_k$ for all $k\in\bbn$.

If $x\in P_j$ for some $j$, Let $U$ be the component of $X\backslash \bigcup_{i\neq j}P_i$ that contains $P_j$. Then $U$ is open and is $\scrp'=\{P_j\}$ is a disjoint tree-grading on $U$. Since $U$ is open, we may assume that $\im(\alpha_k)\subseteq U$ for all $k$. If $\im(\alpha_k)$ lies in the disjoint union of $\bbr$-trees $U\backslash P_j$, then there is a map $g_k:\bbd\to \im(\alpha_k)$ with $(g_k)|_{S^1}=\alpha_k$.
Let $T=\{k\in \bbn\mid im(\alpha_k)\cap P_j\neq \emptyset\}$. If $T$ is finite, then the remainder of the proof is clear since we have defined $g_k$ for all $k\in\bbn\backslash T$. Suppose $T$ is infinite. If $k\in T$, then we may assume that $\alpha_k$ is based in $P_j$. By Lemma \ref{efficientlemma}, $\alpha_k$ is path-homotopic to a tree-efficient loop $\beta_k:S^1\to U$ by a homotopy $H_k:S^1\times\ui\to U$ with image in $\im(\alpha_k)$. However, since $U$ only has one piece, we have $\im(\beta_k)\subseteq P_j$. Hence, $\{\beta_k\}_{k\in T}$ is a sequence of inessential loops in $P_j$ that converges to $x$. Since $P_j$ is assumed to be $1$-$UV_0$, there exists a sequence $\{g_k\}_{k\in T}$ of maps $g_k:\bbd\to P_j$ such that $\{g_k\}_{k\in T}$ converges to $x$ and $(g_k)|_{S^1}=\beta_k$. Since $\im(H_k)\subseteq \im(\alpha_k)$, composing homotopies allows us to construct a sequence $\{h_k\}_{k\in T}$ of maps $h_k:\bbd\to  \im(\alpha_k)\cup\im(g_k)$ where $(h_k)|_{S^1}=\alpha_k$ for all $k\in T$. With $g_k$ defined for all $k\in\bbn$, we see that $\{h_k\}_{k\in\bbn}$ converges to $x$ in $X$.
\end{proof}



Much like Lemma \ref{uvsequencelemma}, the following provides a useful characterization of the uniformly $1$-$UV_0$ property and the proof is left as an exercise.

\begin{lemma}\label{uniformuvsequencelemma}
If $(X,d)$ is a metric space, then the following are equivalent:
\begin{enumerate}
\item $X$ is uniformly $1$-$UV_0$,
\item For every sequence $\{\alpha_k\}_{k\in\bbn}$ of inessential loops $\alpha_k:S^1\to X$ such that $\ds\lim_{k\to\infty}\diam(\im(\alpha_k))=0$, there exists a sequence $\{g_k\}_{k\in\bbn}$ of maps $g_k:\bbd\to X$ such that $\ds\lim_{k\to\infty}\diam(\im(g_k))=0$ and such that $(g_k)|_{S^1}=\alpha_k$ for all $k\in\bbn$.
\end{enumerate}
\end{lemma}

%

Using (2) from Lemma \ref{uniformuvsequencelemma} for the definition of uniformly $1$-$UV_0$, the next corollary has a simple proof.

\begin{corollary}\label{uniformretractcor}
If $(X,d)$ is a uniformly $1$-$UV_0$ metric space and $Y\subseteq X$ is a subspace for which there exists a non-expansive retraction $r:X\to Y$, then $Y$ is uniformly $1$-$UV_0$.
\end{corollary}


\begin{lemma}\label{sectionaluniformlemma}
Let $(X,\scrp)$ be a disjointly tree-graded space and $d$ be a compatible path-diameter metric for $X$. If $\pc(X)$ is uniformly $1$-$UV_0$ and $(Y,\scrq)$ is a sectional graded subspace of $(X,\scrp)$, then $\pc(Y)$ is uniformly $1$-$UV_0$.
\end{lemma}

\begin{proof}
Let $\epsilon>0$. Find $\delta\in(0,\epsilon)$ such that if $\alpha:S^1\to \pc(X)$ is an inessential loop with $\diam(\im(\alpha))<\delta$, then there exists a map $g:\bbd\to \pc(X)$ such that $g|_{S^1}=\alpha$ and $\diam(\im(g))<\epsilon$. Let $\alpha:S^1\to \pc(Y)$ be an inessential loop with $\diam(\im(\alpha))<\delta$. Let $Q$ be the piece of $Y$ for which $\im(\alpha)\subseteq Q$. Since $(Y,\scrq)$ is a sectional graded subspace, either $Q$ is a one-point space or $Q$ is a piece of $X$. If $\alpha$ is constant at a point $x$, then the constant map $g:\bbd\to \pc(Y)$ at $x$ suffices. Suppose that $\alpha$ not constant. Then $Q$ is a non-degenerate piece of both $X$ and $Y$. By our choice of $\delta$, there exist a map $g:\bbd\to \pc(X)$ such that $g|_{S^1}=\alpha$ and $\diam(\im(g))<\epsilon$. Since $Q$ is the path-component of $\pc(X)$ containing $\im(\alpha)$, we must have $\im(g)\subseteq Q\subseteq \pc(Y)$.
\end{proof}

The next lemma highlights the exact way in which we will use the uniformly $1$-$UV_0$ property. It will be important for constructing null-homotopies of loops in tree-graded spaces.

\begin{lemma}\label{extensionlemma}
Let $K$ be a compact subset of $\bbd$ such that $S^1\subseteq K$ and let $\mco=\{O_1,O_2,O_3,\dots\}$ be an enumerated (possibly finite) collection of connected components of $\bbd\backslash K$ such that every set $O_n\in\mco$ is convex. Let $(X,d)$ be a metric space and $A$ be a uniformly $1$-$UV_0$ metric subspace of $X$. If $g:K\to X$ is a map such that
\begin{enumerate}
\item $g(\partial O_n)\subseteq A$ for all $n$,
\item $g|_{\partial O_n}:\partial O_n\to A$ is null-homotopic in $A$ for all $n$,
\item $\ds\lim_{n\to\infty}\diam(g(\partial O_n))=0$ if $\mco$ is infinite,
\end{enumerate}
then there exists a map $H:K\cup (\bigcup_{n=1}^{\infty}O_n)\to X$ such that $H|_{K}=g$. Moreover, if $\bigcup_{n=1}^{\infty}O_n=\bbd\backslash K$, then $g|_{S^1}:S^1\to X$ is inessential.
\end{lemma}

\begin{proof}
Since each set $O_n\in\mco$ is open and convex, there is a homeomorphism of pairs $(\ov{O_n},\partial O_n)\cong (\bbd,S^1)$. Hence, $\alpha_n=g|_{\partial O_n}:\partial O_n\to A$ may be regarded as an inessential loop that extends to a map on $\ov{O_n}$ (see Figure \ref{fig4}). The conclusion is clear if $\mco$ is finite. Suppose $\mco$ is infinite. Since $\ds\lim_{n\to\infty}\diam(g(\partial O_n))=0$ in the uniformly $1$-$UV_0$ metric subspace $A$, Lemma \ref{uniformuvsequencelemma} implies that there exists maps $g_{n}:\ov{O_n}\to A$, $n\in\bbn$ such that $(g_n)|_{\partial O_n}=g|_{\partial O_n}$ for all $n\in\bbn$ and such that $\ds\lim_{n\to\infty}\diam(\im(g_n))=0$. Define $H:K\cup (\bigcup_{n=1}^{\infty}O_n)\to X$ to agree with $g$ on $K$ and to agree with $g_n$ on $\ov{O_n}$.

To see that $H$ is continuous, consider a convergent sequence $\{a_k\}_{k\in\bbn}\to a$ in $K\cup (\bigcup_{n=1}^{\infty}O_n)$ and let $\epsilon>0$. We will show that $d(H(a_k),H(a))<\epsilon$ for all but finitely many $k$. Since $H$ is continuous on $K$, it suffices to consider the case where for each $k\in\bbn$, we have $a_k\in O_{n_k}$ for some $n_k\in\bbn$. If $S=\{n_k\mid k\in\bbn\}$ is finite, then our conclusion follows from the continuity of $g$ and each $g_n$. Supposing that $S$ is infinite, we must have $a\in K$. Let $b_k$ a point in the intersection of $\partial O_{n_k}$ and the line segment in $\bbd$ with endpoints $a$ and $a_k$. Then $\{b_k\}_{k\in\bbn}\to a$ in $K$ and the continuity of $H|_{K}$ gives that $\{H(b_k)\}_{k\in\bbn}\to H(a)$. Find $N\in\bbn$ such that $\diam(\im(g_n))<\frac{\epsilon}{2}$ for all $n\geq N$. In particular, we have $d(H(a_k),H(b_k))=d(g_{n_k}(a_k),g_{n_k}(b_k))<\frac{\epsilon}{2}$ whenever $n_k\geq N$.
Find $K_0$ such that $d(H(b_k),H(a))<\frac{\epsilon}{2}$ for all $k\geq K_0$. When $n_k\geq N$ and $k\geq K_0$, we have $d(H(a_k),H(a))\leq d(H(a_k),H(b_k))+d(H(b_k),H(a))<\epsilon$. If $n<N$ and $T_n=\{k\in\bbn\mid n_k=n\}$ is infinite, then the continuity of $g_n$ implies that $\{H(a_k)\}_{k\in T_n}\to H(a)$ and so we may find $K_n\in\bbn$ such that $d(H(a_k),H(a))<\epsilon$ whenever $k\in T_n$ and $k\geq K_n$. If $n<N$ and $T_n$ is finite, let $K_n=\max(T_n)+1$. Set $K=\max\{K_n\mid 0\leq n<N\}$. If $k\geq K$, the we have $d(H(a_k),H(a))<\epsilon$, completing the proof.

For the last statement, note that if $\mco$ consists of all of the connected components of $\bbd\backslash K$, then we obtain a map $H:\bbd\to X$ such that $H|_{S^1}=g|_{S^1}$. Thus $g_{S^1}$ is inessential.
\end{proof}

\begin{figure}[H]
\centering \includegraphics[height=2.3in]{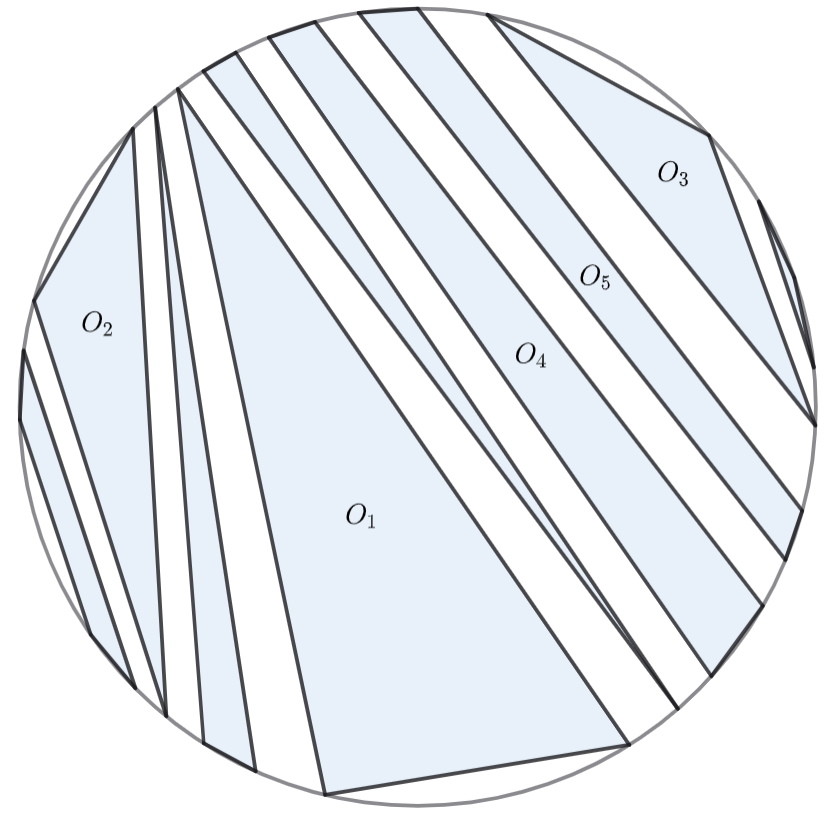}
\caption{\label{fig4} Open convex open sets $O_1,O_2,O_3,\dots$ enumerating some of the connected components of $\bbd\backslash K$. The map $g:K\to X$ sends the boundaries $\partial O_n$ into $A$ by null-homotopic loops that shrink in diameter as $n\to\infty$.}
\end{figure}

\section{Null-homotopies of loops in disjointly tree-graded spaces}\label{sectionconstructingnullhomotopies}

The next lemma is the most technical result in this paper. We use the notation $\lchord a,b\rchord$ to denote the chord of $S^1$ with endpoints $a,b\in S^1$ and $\hull(A)$ to denote the convex hull of a set $A\subseteq \bbr^2$.

\begin{lemma}\label{kextensionlemma}
Let $(X,\scrp)$ be a disjointly tree-graded space and $\alpha:S^1\to X$ be a tree-efficient loop. Then there exists a compact set $K$ with $S^1\subseteq K\subseteq \bbd$ and a map $g:K\to X$ such that if $\mco=\{O_1,O_2,O_3,\dots\}$ is an enumeration of the (possibly finite) collection of connected components of $\bbd\backslash K$, then the following hold:

    \begin{enumerate}
        \item $g|_{S^1}=\alpha$ and $\im(g)=\im(\alpha)$,
        \item every element of $\mco$ is convex,
        \item there exists (not necessarily distinct) pieces $P_1,P_2,P_3,\dots\in\scrp$ such that $g(\partial O_n)\subseteq P_n$ for each $O_n\in\mco$, and
        \item\label{item: null sequence} $\ds\lim_{n\to\infty}\diam(g(\partial O_n))=0$ if $\mco$ is infinite.
    \end{enumerate}

Moreover,  if $\alpha$ is inessential, then, for every $n$, $g|_{\partial O_n}:\partial O_n\to P_n$ is inessential in $P_n$.
\end{lemma}

\begin{proof}
    Let $(T,V,q)$ be a parameterization for $(X,\scrp)$.  Suppose that $\alpha: S^1 \to X$ is a tree-efficient loop.  Since $T$ is a topological tree, $q\circ\alpha$ is nullhomotopic. Let $f_0:\bbd \to T$ be an extension of $q\circ\alpha$.  We may then modify $f_0$ to obtain a map $f':\bbd \to T$ by making $f'$ constant on the closure of bounded components of the complements of maximal point preimages of $f_0$ that separate $\mathbb R^2$ (for details, see proof of \cite[Theorem 3.7]{CConedim} or Claim \ref{clm: maximal separate} for an analogous proof on the two-sphere.)  Then no point preimage of $f'$ separates $\mathbb R^2$ and $f'|_{S^1}= f_0|_{S^1} = q\circ\alpha$.  For each component $C$ of $f'^{-1}(t)$, there exists a unique ``bounding component" $C'$ of $f_0^{-1}(t)$ such that $C$ is the union of $C'$ with the bounded components of $\mathbb R^2\backslash C'$ and $C\cap S^1 = C'\cap S^1$. (If $C'$ doesn't separate $\mathbb R^2$, then $C = C'$).   Set
     \[H'= \bigl\{ C \mid \exists\ t\in T \text{ such that } C \text{ is a component of } f'^{-1}(t)\bigr\} \text{ and}\] \[H= \bigl\{ \hull(C\cap S^1)\mid \exists\ t\in T \text{ such that } C \text{ is a component of } f'^{-1}(t)\bigr\}.\] By the proof of \cite[Theorem 3.7]{CConedim} and \cite[Subsection 3.2.1]{CConedim} respectively, $H'$ and $H$ are upper semicontinuous decompositions of $\bbd$ such that $\bbd/H'$ and $\bbd/H$ are dendrites.

    Let $\rho: \bbd \to \bbd/H$ and $\rho': \bbd \to \bbd/H'$ be the closed, monotone, quotient maps. Then there exist maps $\eta': \bbd/H' \to T$ and $\eta: \bbd/H \to T$ such that $\eta\circ\rho|_{S^1} = q\circ\alpha = \eta'\circ\rho'|_{S^1}$ and $\eta'\circ\rho'= f'$, which are uniquely defined and continuous by the universal property of quotient maps.  Let $f = \eta\circ\rho$.

    By construction, $\rho|_{S^1}$ is a quotient map and $\rho'|_{S^1}$ is constant on the fibers of $\rho'$. Hence, there exists an induced continuous map $\theta: \bbd/H \to \bbd/H'$. There is a bijection between elements of $H$ and elements of $H'$ that intersect the boundary given by $h = \hull(C\cap S^1) \mapsto C$. Then, it is immediate that $\theta$ is injective and $\eta= \eta'\circ \theta$. Thus we have the following commutative diagram.

\[\begin{tikzcd}
	\bbd && {S^1} && \bbd \\
	& {\bbd/H} & {} & {\bbd/H'} \\
	&& T
	\arrow["\rho"', from=1-1, to=2-2]
	\arrow["f"', curve={height=30pt}, from=1-1, to=3-3]
	\arrow[hook', from=1-3, to=1-1]
	\arrow[hook, from=1-3, to=1-5]
	\arrow["{q\circ\alpha}"{pos=0.2}, from=1-3, to=3-3]
	\arrow["{\rho'}", from=1-5, to=2-4]
	\arrow["{f'}", curve={height=-30pt}, from=1-5, to=3-3]
	\arrow["\theta"{pos=0.7}, hook, no head, from=2-2, to=2-3]
	\arrow["\eta"', from=2-2, to=3-3]
	\arrow[from=2-3, to=2-4]
	\arrow["{\eta'}", from=2-4, to=3-3]
\end{tikzcd}\]

    Since $\rho$ and $\rho'$ are a monotone closed maps, the preimage under either map, of a connected  set is connected, see \cite[Proposition I.4.2]{Daverman2007}. If $A$ is a connected subset of $\eta'^{-1}(t)$ for some $t\in T$, then $\rho'^{-1}(A)$ is connected and must be contained in a single element of $H'$.  Thus $A$ is a point.  Therefore, $\eta'$ is a light map.  Since $\eta= \eta'\circ \theta$, $\eta$ is also a light map.

    Suppose that $A$ is a connected subset of $\bbd/H'$ contained in $\eta'^{-1}(V)$.  Since $\eta'$ is a light map and $V$ is totally disconnected, $A\subseteq \eta'^{-1}(v)$ for some $v\in V$.  Then the connected set $\rho'^{-1}(A)$ must be contained in an element of $H'$ and $A$ is a point. Thus every non-degenerate connected subset of $\bbd/H'$ intersects $\eta'^{-1}(T\backslash V)$ and any two points in $\bbd/H'$ can be separated by a point in $\eta'^{-1}(T\backslash V)$.

    \begin{claim}\label{clm: separate components}
        Let $h_1$ and $h_2$ be distinct elements of $H$ and $C_1$ and $C_2$ the corresponding elements of $H'$, i.e., $h_i = \hull(C_i\cap S^1)$ for $i= 1,2$.  Then there exists a $t\in T\backslash V$ and $C$ a component of $f'^{-1}(t)$ such that $C$ separates $C_1$ and $C_2$ in $\bbd$ and $\hull(C\cap S^1)$ separates $h_1$ and $h_2$ in $\bbd$. In addition, the unique bounding component $C'$ of $f_0^{-1}(t)$ contained in $C$ also separates $C_1$ and $C_2$ in $\bbd$.
    \end{claim}

    \begin{proof}[Proof of claim.]
      Suppose that $h_1, h_2\in H$ are distinct elements of $H$.  Choose $C_1, C_2\in H'$ such that $C_i\cap S^1 = h_i\cap S^1$.  (Recall that $C_1$ and $C_2$ are uniquely determined.)  Thus $\rho'(C_1)$ and $\rho'(C_2)$ are distinct points of $\bbd/H'$, which means that there exists $d\in\bbd/H'$ such that $\eta'(d)\in T\backslash V$ and such that $d$ separates $\rho'(C_1)$ and $\rho'(C_2)$. By construction, $C= \rho'^{-1}(d)$ is a component of $f'^{-1}\bigl(\eta'(d)\bigr)$ and $C$ separates $C_1$ and $C_2$ in $\bbd$. In particular, $C\cap S^1\neq\emptyset$.

      Notice that $\rho\bigl(\hull(C\cap S^1)\bigr) = \theta^{-1}(d)$ and that $\theta^{-1}(d)$ separates $\theta^{-1}\bigl(\rho'(C_1)\bigr)=\rho(h_1)$ and $\theta^{-1}\bigl(\rho'(C_2)\bigr)=\rho(h_2)$ in $\bbd/H$.  Thus $\hull(C\cap S^1)$ must separate $h_1$ and $h_2$ in $\bbd$.

      The set of components of  $\bbd \backslash C'$ is the set of the components of $\bbd \backslash C$ together with the bounded components of $\mathbb R^2\backslash C'$.  Thus $C'$ must also separates $C_1$ and $C_2$ in $\bbd$.
    \end{proof}

    Notice that each element of $H$ is contained in an element of \[\tilde H = \bigl\{ C \mid \exists\ t\in T \text{ such that } C \text{ is a component of } f^{-1}(t)\bigr\}.\]  By the previous claim, no two distinct elements of $H$ can be contained in the same component of $f^{-1}(t)$ for any $t\in T$. Thus   $H= \tilde H$.

    Let $H_{\tree}= \bigl\{ h\in H\mid f(h)\in T\backslash V\bigr\}$ and $H_{\tree}'= \bigl\{ h\in H'\mid f'(h)\in T\backslash V\bigr\}$.
\begin{claim}\label{clm: constant on endpoints}
    For every $h\in H$ and every component $A$ of $S^1\backslash h$, $\alpha$ is constant on the endpoints of $A$.
\end{claim}

\begin{proof}[Proof of Claim \ref{clm: constant on endpoints}]

    Let $h,A$ be as in the statement of the claim, and $a,b$ be the endpoints of $A$.  Since elements of $H$ are convex, $\hull(A)$ is a $\rho$-saturated open subset of $\bbd$ and, hence, $\rho\bigl(\hull(A)\bigr)$ is an open connected subset of $\bbd/H$.  Choose a nested sequence of compact, connected subsets $S_n$ of $A$ such that $\bigcup_{n\in\bbn} S_n = A$.
    Choose $e_n \in \rho(A)$ separating $\rho(h)$ and $\rho\bigl(\hull(S_n)\bigr)$ such that $\eta(e_n)\in T\backslash V$.

    Fix $\epsilon>0$.  Then there exists a $\delta>0$ such that $d(x,y)< \delta$ implies that $d\bigl(\alpha(x),\alpha(y)\bigr)< \epsilon/2$. Choose $n\in\bbn$ sufficiently large so that $ \hull(S_n)\cap B_\delta(a)\neq \emptyset $ and  $\hull(S_n)\cap B_\delta(b)\neq \emptyset$ (see Figure \ref{figendpoints}).  Since $e_n$ separates $\rho(h)$ and $\rho\bigl(\hull(S_n)\bigr)$, $\rho^{-1}(e_n)$ must intersect both $B_\delta(a)$ and $B_\delta(b)$.  By construction, $\alpha\bigl( \rho^{-1}(e_n)\bigr)$ is a point. Thus $d\bigl(\alpha(a), \alpha(b)\bigr)\leq d\bigl(\alpha(a),\alpha\bigl( \rho^{-1}(e_n)\bigr)\bigr)+d\bigl(\alpha\bigl( \rho^{-1}(e_n)\bigr),\alpha(b)\bigr)< \epsilon$.  Since $\epsilon$ was arbitrary,  $\alpha(a) = \alpha(b)$.
\end{proof}

\begin{figure}[H]
\centering \includegraphics[height=1.8in]{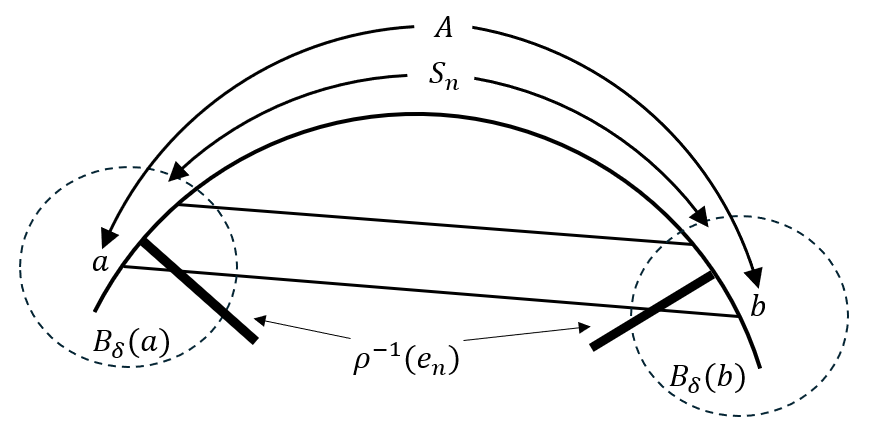}
\caption{\label{figendpoints} $e_n$ is chosen in $\eta^{-1}(T\backslash V)$ to separate $\rho(h)$ and $\rho\bigl(\hull(S_n)\bigr)$.}
\end{figure}

  Every element of $H$ is a point on $S^1$, a chord, or a topological disc.  If $h$ is a disc, since $h = \hull(C\cap S^1)$ for some $C\in H'$, then the boundary of $h$ consists of subsets of $S^1$ and chords in $\bbd$.  Let  $H_{\pc}= \bigl\{ h\in H\mid  f(h)\in V \text{ and } h \text{ is a topological disc}\}$. Observe that  if $ h\in H\backslash H_\pc $, then $\alpha$ is constant on $h\cap S^1$.  (If $h$ is a point, the observation is trivial.  If $h$ is a chord mapped into $V$ by $f$, then apply Claim \ref{clm: constant on endpoints}.  If $h\in H_\tree$, then the observation is true since $q$ is injective on $\tree(X)$ and $f|_{S^1} = q\circ \alpha$.)


  Let $K = \bbd\backslash \bigcup\limits_{h\in H_\pc} \int_\bbd(h)$.

  We will now construct an upper semicontinuous decomposition $G$ of $K$ that refines $H$ and such that $K/G$ maps into $X$.  For each $h\in H_\pc$, let $H_{\partial h}$ be the maximal connected subsets $J$ of $\partial h$ such that $\alpha$ is constant on $J\cap S^1$.  An endpoint of such a $J$ is either an endpoint of a chord in $\partial h$ or a limit point of $J\cap S^1$. It then follows from Claim \ref{clm: constant on endpoints} that the elements of $H_{\partial h}$ are closed. It is immediate that  $H_{\partial h}$ is a monotone decomposition of $\partial h$ and hence $\partial h/H_{\partial h}$ is either a point (if $\alpha$ was constant on $h\cap S^1$) or a simple closed curve.

  Consider the decompositions $G = \Bigl(H\backslash H_\pc \Bigr) \cup  \Bigl(\bigcup\limits_{h\in H_\pc} H_{\partial h} \Bigr)$ and $G_{S^1} = \{g\cap S^1\mid g\in G\}$ of $K$ and $S^1$ respectively. Let $\rho_1:S^1\to S^1/G_{S^1}$ and $\rho_2:K\to K/G$ be the quotient maps.

  \begin{claim}
    $G$ and $G_{S^1}$ are upper semicontinuous decompositions of $K$ and $S^1$ respectively.
  \end{claim}

  \begin{proof}
    Suppose that $g_1$ and $g_2$ are distinct elements of $G$ that are contained in distinct elements $h_1$ and $h_2$ of $H$.  Since $H$ is an upper semicontinuous decomposition of $\bbd$, there exists $U$ and $V$ disjoint $\rho$ saturated subsets of $\bbd$ containing $h_1$ and $h_2$ respectively.  Then $U\cap K$ and $U\cap K$ are disjoint $\rho_2$-saturated subsets of $K$ containing $g_1$ and $g_2$ respectively.  Thus $\rho_2(g_1)$ and $\rho_2(g_2)$ can be separated by disjoint open sets.

    Suppose that $g_1, g_2\in G$ are distinct and both contained in a single element $h$ of $H$.  Then $h\in H_\pc$ and, thus, $h$ is a convex topological disc in $\bbd$.  Since $g_1$ and $g_2$ are distinct closed subintervals (possibly degenerate) of $\partial h$, there exists $g, g'\in G$ also contained in $\partial h$ that separate $g_1$ from $g_2$ in $\partial h$.    Let $x\in g\cap S^1$ and $x'\in g'\cap S^1$ and $U,V$ be the two components of $\bbd\backslash \lchord x,y\rchord $.  Then $U\cap K\backslash\{g\cup g'\}$ and $V\cap K\backslash\{g\cup g'\}$ are disjoint open $\rho_2$ saturated sets one of which contains $g_1$ and the other contains $g_2$.  Thus $\rho_2(g_1)$ and $\rho_2(g_2)$ can be separated by disjoint open sets.

    Thus $K/ G$ is Hausdorff.  Since $K$ is compact, $\rho_2$ is a closed map which implies that $G$ is an upper semicontinuous decomposition.  Then Exercise 2 on page 62 of \cite{CCcodiscrete} implies that $G_{S^1}$ is also an upper semicontinuous decomposition.
  \end{proof}

    It is an exercise to see that the inclusion of $S^1$ into $K$ induces a canonical homeomorphism of $\psi: S^1/G_{S^1}\to K/G$ such that $\rho_2|_{S^1}=\psi\circ \rho_1$. Since $\rho:\bbd\to \bbd/H$ is constant on the elements of $G_{S^1}$, there is a quotient map $q':S^1/G_{S^1}\to \bbd/H$ such that $q'\circ \rho_1=\rho|_{S^1}$. Let $q''=q'\circ \psi^{-1}$. Since every element of $G$ is contained in some element of $H$, we have $q''\circ \rho_2=\rho|_{K}$.

By construction, $\alpha$ is constant on elements of $G_{S^1}$. Hence, there is unique continuous map $g':S^1/G_{S^1}\to X$ such that $\alpha = g'\circ\rho_{1}$ (see Figure \ref{rhoone}). Let $g: K\to X$ be the continuous map $g= g'\circ \psi^{-1}\circ \rho_2$. Thus $g|_{S^1}= \alpha$ and $q\circ g = f|_K$. Since $\rho_1$ is an epimorphism, we have $\eta\circ q'=q\circ g'$. The following diagram commutes and Condition (1) is satisfied.

\[\begin{tikzcd}
	{S^1} & K & \bbd \\
	{S^1/G_{S^1}} & {K/G} \\
	X & {\bbd/H} \\
	T
	\arrow[hook, from=1-1, to=1-2]
	\arrow["{\rho_1}", two heads, from=1-1, to=2-1]
	\arrow["\alpha"', curve={height=30pt}, from=1-1, to=3-1]
	\arrow[hook, from=1-2, to=1-3]
	\arrow["{\rho_2}", two heads, from=1-2, to=2-2]
	\arrow["\rho", curve={height=-12pt}, from=1-3, to=3-2]
	\arrow["\psi", dashed, two heads, from=2-1, to=2-2]
	\arrow["{g'}", dashed, from=2-1, to=3-1]
	\arrow["{q'}"', from=2-1, to=3-2]
	\arrow["{q''}",from=2-2, to=3-2]
	\arrow["q", from=3-1, to=4-1]
	\arrow["\eta", from=3-2, to=4-1]
\end{tikzcd}\]

\begin{figure}[H]
\centering \includegraphics[height=3.2in]{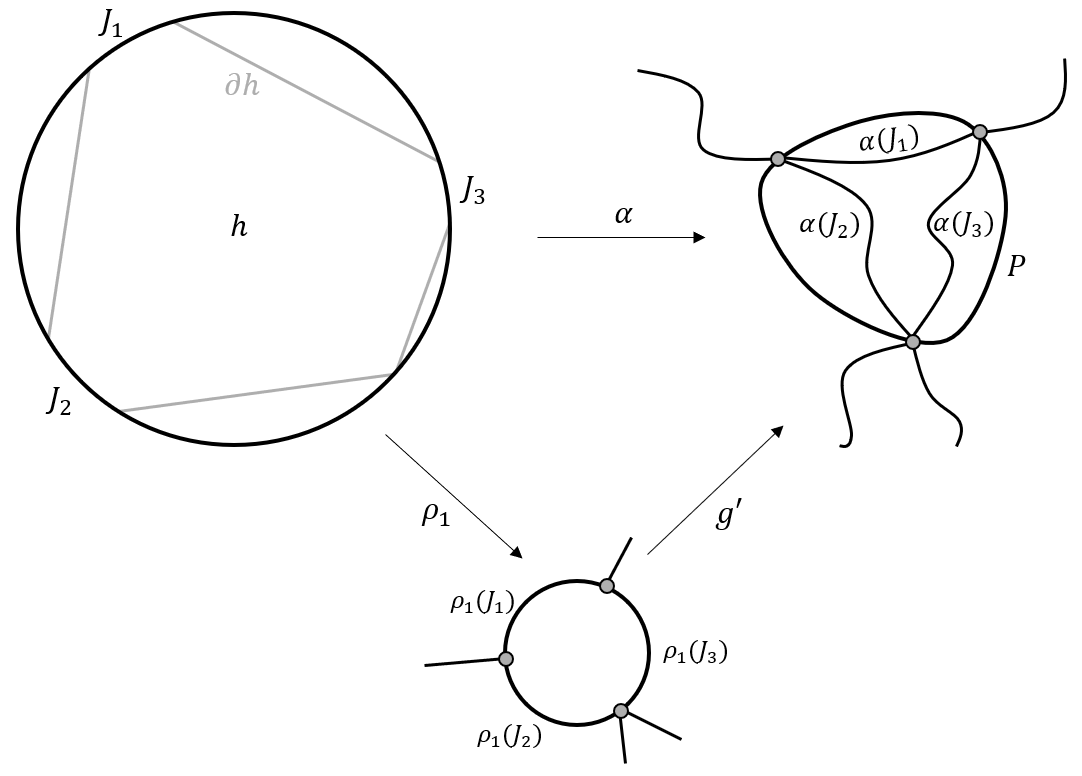}
\caption{\label{rhoone} A set $h\in H_{\pc}$ such that the components $J_1,J_2,J_3$ of $\partial h\cap S^1$ are mapped into the piece $P$ of $X$ non-trivially by $\alpha$. The image of $\partial h$ under quotient map $\rho_1$ is a simple closed curve where the closures of the three components of $\partial h\backslash (J_1\cup J_2\cup J_3)$ are elements of $H_{\partial h}$ and thus identified to points. }
\end{figure}

  Recall that $K = \bbd\backslash \bigcup\limits_{h\in H_\pc} \int_\bbd(h)$. So $K$ has countably many complementary components and $g$ takes the boundary of each into a single piece of $\pc(X)$. Let $\mco=\{O_1,O_2,O_3,\dots\}$ be an enumeration of the (possibly finite) collection of connected components of $\bbd\backslash K$.  Choose $h_i\in H_\pc$ such that $\int_\bbd(h_i) = O_i$. Since $h_i$ is convex, so is $O_i$, giving Condition (2).

The space $S^1/G_{S^1}$ is a Peano continuum and since $\eta$ is a light map $\eta^{-1}(V)$ is a totally disconnected subset of the dendrite $\bbd/H$. A non-degenerate fiber of $q': S^1/G_{S^1}\to \bbd/H$ must have the form $\rho_1(S^1\cap h_i)$ for some $i$, which is a simple closed curve since $\psi\circ\rho_1(S^1\cap h_i)=\rho_2(\partial h_i)$ where $\rho_2$ identifies closed connected subsets of $\partial h_i$. In particular, all fibers of $q'$ are path-connected. Additionally, since $h_i\in H_\pc$, $q'$ must map a non-degenerate fiber $\rho_1(S^1\cap h_i)$ to the point $\rho(S^1\cap h_i)\in \eta^{-1}(V)$. Note that since $q\circ g(\partial h_i)=\eta\circ \rho(S^1\cap h_i)\in V$ and $X$ is a disjointly tree-graded space, $g$ must map $\partial h_i$ into some piece of $X$, giving Condition (3) of the current lemma. Additionally, with all of the conditions of Lemma \ref{datalemma} now satisfied, $S^1/G_{S^1}$ inherits the structure of a disjointly tree-graded space with parameterization $(\bbd/H,\eta^{-1}(V),q')$. Since $S^1/G_{S^1}$ is compact, Lemma \ref{pclemma} gives that the diameters of the pieces of $S^1/G_{S^1}$ form a null sequence (with respect to any choice of compatible metric). Thus their image in $X$ under $g'$ must also form a null sequence and we have Condition (\ref{item: null sequence}).

We are now ready to prove the last conclusion of the lemma.  Suppose that $\alpha$ is an inessential loop. Then we can assume that $f_0= q\circ f''$ for a map $f'': \bbd\to X$. We will fix an identification of the upper hemisphere and lower hemisphere of $S^2$ with $\bbd$, which agree on the equator.  With this identification, we can view $H$ as an upper semicontinuous decomposition of the upper hemisphere of $S^2$ and $H'$ as an upper semicontinuous decomposition of the lower hemisphere of $S^2$.

Let $\widetilde K = S^2\backslash \bigcup\limits_{h\in H_\pc} \int_\bbd(h)\subseteq S^2$.  Then $\widetilde K$ is the two-sphere with a sequence of open discs removed from the upper hemisphere.

\begin{claim}\label{clm: 2-sphere separations}
   Let $l: \widetilde K \to X$ be given by $g$ on the upper hemisphere and $f''$ on the lower hemisphere.    Then every two elements of $H_\pc$  can be separated by a component of $l^{-1}(x)$  for some $x\in \tree(X)$.
\end{claim}

\begin{proof}
  Since $g|_{S^1} = \alpha = f''|_{S^1}$, the map $l$ is well-defined and continuous. Notice that for $x\in \tree(X)$ and $t = q(x)$, we have $q^{-1}(t)=\{x\}$ and thus \[l^{-1}(x) = g^{-1}(x) \cup f''^{-1}(x)= (q\circ g)^{-1}(t)\cup (q\circ f'')^{-1}(t) = f^{-1}(t) \cup f_0^{-1}(t).\]

  Let $h_1$ and $h_2$ be distinct elements of $H$ (identified with the corresponding subset of the upper hemisphere of $S^2$) and $C_1$ and $C_2$ be the corresponding elements of $H'$ (identified with the corresponding subset of the lower hemisphere of $S^2$).  By Claim \ref{clm: separate components}, there exists $C\in H_{\tree}'$ such that $C$ separates $C_1$ and $C_2$ in the lower hemisphere and $h=\hull(C\cap S^1)\in H_{\tree}$ separates $h_1$ and $h_2$ in the upper hemisphere. Let $t= f'(C)$ and note $t\in T\backslash V$. There is a unique $x\in \tree(X)$ such that $q(x)=t$. The final statement of Claim \ref{clm: separate components} ensures that the corresponding bounding component $C'$ of $f_0^{-1}(t)$ contained in $C$ also separates $C_1$ and $C_2$ in the lower hemisphere. Since $q$ is injective on $T\backslash V$ and $f_0(C') = f'(C) = f(h)=t$, we have that $f''(C) = g(h)=x$. Then $C'\cup h$ separates $h_1$ and $h_2$ in $S^2$ and  $C'\cup h\subseteq l^{-1}(x)$.

\end{proof}

Fix $i\in \bbn$ where $O_i$ is a component of $\bbd\backslash K$. Then $h_i\in H_\pc$ is such that $\int_\bbd(h_i) = O_i$. We will say that a subset of $S^2$ is ``bounded" if it is disjoint from $O_i$, which lies in the upper hemisphere. Let $\mcs$ be the set of all components of point-preimages of $l$ that separate $S^2$. For $D, D'\in\mathcal S$, we will say that $D< D'$ if $D$ is contained in a bounded component of $S^2\backslash D'$.

\begin{claim}\label{clm: maximal separate}
    Every element of $\mathcal S$ is less than or equal to a maximal one and any two distinct elements are disjoint.
\end{claim}

\begin{proof}[Proof of claim.]
   The second statement of the claim is trivial. We wish to apply Zorn's lemma; so we need only show that every chain is bounded.  Suppose that $\{D_\iota\}$ is a chain in $\mcs$, which we may assume does not have a maximal element. Let $U_\iota$ be the union of bounded components of $S^2\backslash D_\iota$. Note this implies that $U_\iota\cup D_\iota\subseteq U_\kappa$ whenever $\iota< \kappa$.  Let $U =\cup_\iota U_\iota$.  Fix $\epsilon>0$ and choose $\delta>0$ such that $d(x,y)<\delta$ implies that $d\bigl(l(x),l(y)\bigr)< \epsilon/2$.  Notice that $\partial U\subseteq \cup_\iota \mcn_\delta(U_\iota)$ where $\mcn_\delta(U_\iota)$ denotes the $\delta$-neighborhood of $U_{\iota}$. Since the collection $\{\mcn_\delta(U_\iota)\}_{\iota}$ is nested and $\partial U$ is compact, $\partial U\subseteq \mcn_\delta(U_\iota)$ for some $\iota$.  Then, for any $D_\kappa>D_\iota $, we have that $D_\kappa $ separates $U_\iota $ from $\partial U$ and, hence, that $\partial U\subseteq \mcn_\delta(D_\kappa)$. (This follows since we are using a length metric on $S^2$.) Then  $ \diam\bigl(l(\partial U)\bigr)\leq \diam\bigl(l\bigl( \mcn_\delta(D_\kappa)\bigr)\bigr)< \epsilon$. Since $\epsilon$ was arbitrary, $l|_{\partial U}$ is constant and there exists $x\in X$ such that $\partial U\subseteq l^{-1}(x)$.

We now check that $\partial U$ is connected. Suppose that $W$, $V$ is an open separation of $\partial U$.  Fix $a\in W\cap \partial U$ and $b\in V\cap \partial U$.  Since $W$ and $V$ are open and $\partial U$ is compact, we may choose $\delta>0$ such that $B_\delta(a)\subseteq W$, $B_\delta(b)\subseteq V$ and $\mcn_\delta(\partial U) \subseteq W\cup V$. Then $\overline U\backslash \mcn_\delta(\partial U) = U \backslash \mcn_\delta(\partial U) \subseteq \cup_\iota U_\iota$, which implies there exist $\iota_0$ such that $U\backslash \mcn_\delta(\partial U) \subseteq  U_{\iota_0}$. Thus, for all $D_\kappa > D_{\iota_0}$, we have $D_\kappa \subseteq U\backslash U_{\iota_0}\subseteq \mcn_\delta (\partial U) \subseteq W\cup V$.

    As in the first paragraph of the proof, there exists an $\kappa$ such that $D_\kappa> D_{\iota_0}$ and $\partial U\subseteq \mcn_\delta(D_\kappa)$. Thus $D_\kappa\cap W \neq \emptyset \neq D_\kappa\cap V$, which means that $W$, $V$ give an open separation of the connected set $D_\kappa$; a contradiction. Thus $\partial U$ is connected. We conclude that the component of $l^{-1}(x)$ containing $\partial U$ is an upper bound for the chain $\{D_\iota\}$ and the claim follows by Zorn's Lemma.
\end{proof}

    We can now define $\tilde l: S^2\backslash O_i \to X$, a nullhomotopy of $g|_{\partial O_i}$.  Let $U$ be the union of all bounded components of maximal elements of $\mcs$.  By Claim \ref{clm: 2-sphere separations}, $U$ contains $\bigcup_{j\neq i} \int_\bbd(h_j)$.  Let $\tilde l(x) = l(x)$ for $x\in S^2\backslash \bigl(O_i\cup U\bigr)$. For $x\in U$, let $\tilde l|_{D} = l(\partial D)$ where $x$ is contained in a bounded component of $S^2\backslash D$ and $D$ is a maximal element of $S$, i.e., $\tilde l$ collapses each component of $U$.  It is a trivial exercise to see that $\tilde l|_{\overline U}$ is continuous and agrees with $l$ on $\partial U$.  Thus $\tilde l$ is continuous and $\partial O_i$ is inessential in $X$.  Since each piece of the disjoint tree-grading is a retract, $\partial O_i$ is inessential in the piece that contains it.

\end{proof}

The next statement is a direct combination of Lemmas \ref{extensionlemma} and \ref{kextensionlemma} is the most important consequence of these technical results.

\begin{lemma}\label{simplyconnectedlemma}
If $(X,\scrp)$ is a disjointly tree-graded space such that every piece of $X$ is simply connected and $\pc(X)$ is uniformly $1$-$UV_0$, then $X$ is simply connected.
\end{lemma}

\begin{remark}
Lemmas \ref{extensionlemma} and \ref{kextensionlemma} can also be used to prove the non-trivial direction of the following statement: If $(X,\scrp)$ is a disjointly tree-graded space and $d$ is a compatible path-diameter metric on $X$, then $X$ is uniformly $1$-$UV_0$ if and only if $\pc(X)$ is uniformly $1$-$UV_0$ as a metric subspace. Since we do not use this fact directly, we omit the proof. However, it is worth noting that the hypothesis on $\pc(X)$ in Theorem \ref{mainthm1} may be replaced with the hypothesis that $X$ itself is uniformly $1$-$UV_0$.
\end{remark}

The next statement is a weak version of Corollary \ref{maincorollary} that is required in our proof of Theorem \ref{mainthm1}.

\begin{proposition}\label{pioneinjectiveprop}
Suppose that $(X,\scrp)$ and $(Y,\scrq)$ are disjointly tree-graded spaces where $\pc(X)$ is uniformly $1$-$UV_0$. If $f:X\to Y$ is a grade-preserving injective map such that $f(\tree(X))\subseteq \tree(Y)$ and such that whenever $f(P)\subseteq Q$ for $P\in\scrp$ and $Q\in\scrq$, the restriction $f|_{P}:P\to Q$ is $\pi_1$-injective, then $f$ is $\pi_1$-injective.
\end{proposition}

\begin{proof}
Let $\beta:S^1\to X$ be a map such that $\alpha=f\circ\beta:S^1\to Y$ is inessential. By Lemma \ref{kextensionlemma}, we may find a compact set $S^1\subseteq K\subseteq \bbd$ and map $g:K\to Y$ satisfying the conclusions of the statement. Since $\im(\beta)$ is compact and $f$ is a continuous injection, $f$ maps $\im(\beta)$ homeomorphically onto $\im(\alpha)$. Since $\im(g)=\im(\alpha)$, there exists $h:K\to X$ with $f\circ h=g$. Let $O$ be a connected component of $\bbd\backslash K$. Since $\partial O$ is a simple closed curve in $\bbd$ with $g(\partial O)\subseteq Q$ for some $Q\in\scrq$, the assumption that $f(\tree(X))\subseteq \tree(Y)$ ensures that $h(\partial O)\subseteq P$ for some $P\in\scrp$ with $f(P)\subseteq Q$. Since $g|_{\partial O}=f|_{P}\circ h|_{\partial O}$ is inessential in $Q$ and $f|_{P}:P\to Q$ is assumed to be $\pi_1$-injective, $h|_{\partial O}$ is inessential in $P$ (and thus is essential in $\pc(X)$). We conclude that $h|_{\partial O}$ is inessential for every component $O$ of $\bbd\backslash K$.

Since $\ds\lim_{n\to\infty}\diam(g(\partial O_n))=0$ in the compact metric space $\im(g)$, and $f$ maps $\im(h)$ homeomorphically onto $\im(g)$, it follows that $\ds\lim_{n\to\infty}\diam(h(\partial O_n))=0$ in $\im(h)\cap \pc(X)$. Since $\pc(X)$ is uniformly $1$-$UV_0$ it follows from Lemma \ref{extensionlemma} that there exists a map $H:\bbd\to X$ such that $H|_{K}=h$. Since $H|_{S^1}=\beta$, we conclude that $\beta$ is inessential.
\end{proof}

\begin{example}\label{bijectionexample}
Let $X$ be a subspace of $\bbr^3$ where a circle of fixed diameter is attached to $\ui$ at each point in the sequence $1,1/2,1/3,1/4,\dots$. Let $Y$ be similarly constructed but where the circles have null diameters (see Figure \ref{fig3}). Each of these spaces is disjointly tree-graded with pieces $\{0\}$ and the attached circles. However $Y$ is compact and $Y$ is not compact. The canonical continuous, grade-preserving, bijection $f:X\to Y$ is $\pi_1$-injective by Proposition \ref{pioneinjectiveprop} but is not an isomorphism since $\pi_1(X)$ is countable and $\pi_1(Y)$ is uncountable.
\end{example}

\begin{figure}[H]
\centering \includegraphics[height=1.1in]{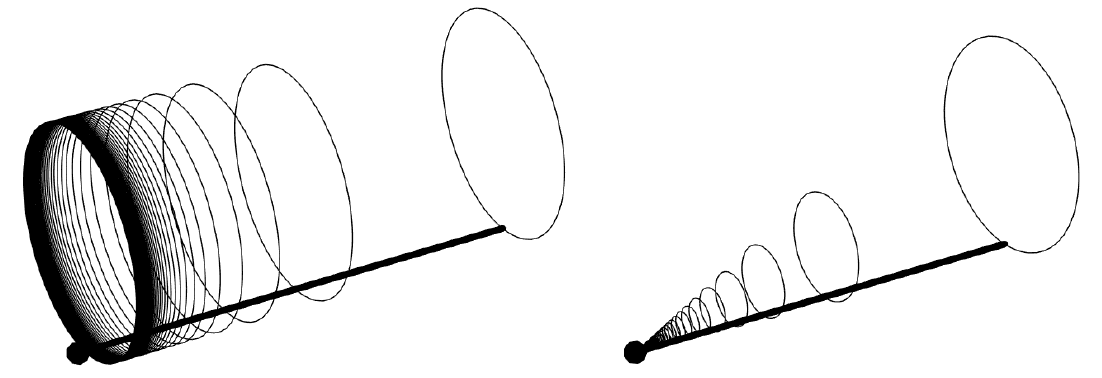}
\caption{\label{fig3} Disjointly tree-graded spaces $X$ (left) and $Y$ (right) with a canonical continuous, grade-preserving, bijection $f:X\to Y$.}
\end{figure}

\section{Generalized covering maps}\label{sectiongencovmaps}

An important tool in our proof is the generalized covering space theory of Fischer-Zastrow \cite{FZ07}, which uses the usual lifting property of covering maps as a definition.

\begin{definition}
A map $p:E\to X$ is a \textit{generalized covering map} if
\begin{enumerate}
\item $E$ is path-connected and locally path-connected,
\item for every path-connected, locally path-connected space $Y$, based map $f:(Y,y)\to (X,x)$, and point $e\in p^{-1}(x)$ such that $f_{\#}(\pi_1(Y,y))\leq p_{\#}(\pi_1(E,e))$, there exists a unique map $\wt{f}:(Y,y)\to (E,e)$ such that $p\circ \wt{f}=f$.
\end{enumerate}
If, in addition, $p_{\#}(\pi_1(E,e_0))$ is a normal subgroup of $\pi_1(X,p(e_0))$ for some $e_0\in E$, then we refer to $p$ as a \textit{generalized regular covering map}. If $E$ is simply connected and $X$ is path-connected, then we refer to $p$ as a \textit{generalized universal covering map} and $E$ as a \textit{generalized universal covering space} of $X$.
\end{definition}

\begin{remark}\label{fibersremark}
Note that every generalized covering map has the unique path lifting property and therefore all of it's fibers are $T_1$.
\end{remark}

\begin{remark}
Formally, we do not require that a generalized covering map be surjective since the definition permits $X$ to have multiple path-components. However, a generalized covering map $p:E\to X$ maps surjectively onto a path-component of $X$.
\end{remark}

\begin{remark}[A basis for the topology of $E$]
If $p:E\to X$ is a generalized covering map, the topology of $E$ can be identified with the so-called ``whisker topology" \cite[Lemma 5.10]{Brazcat}. Let $e\in E$ and $U$ be an open neighborhood of $p(e)$ in $X$. Then a basic neighborhood of $e$ in $E$ is given by \[N(e,U)=\{\wt{\gamma}(1)\mid \gamma:\ui\to U\text{ s.t. }\gamma(0)=p(e)\text{ and } \wt{\gamma}(0)=e\}\]
where $\wt{\gamma}$ denotes the unique lift of $\gamma$ starting at $e$. It is straightforward to check that $N(e,U)$ is path-connected.
\end{remark}

\begin{example}[Locally path-connected coreflections]\label{lpccoreflectionexample}
An important example of a generalized covering map is one which is bijective but need to be a homeomorphism. For any space $X$, let $\lpc(X)$ be the space with the same underlying set as $X$ but with the topology generated by the path-components of the open sets in $X$. Sometimes $\lpc(X)$ is referred to as the \textit{locally path-connected coreflection of} $X$ or the \textit{universal lpc-space of} $X$ \cite{BDLM08}. It is well-known that $\lpc(X)$ is locally path-connected, the identity $id:\lpc(X)\to X$ is continuous and has the property that if $f:Z\to X$ is a map from a locally path-connected space $Z$, then $f:Z\to \lpc(X)$ is continuous. Note that $\lpc(X)=X$ if and only if $X$ is locally path-connected. It follows that whenever $X$ is path-connected, the identity function $id:\lpc(X)\to X$ is a generalized covering map.
\end{example}

\begin{remark}[Discreteness of fibers]\label{evenlycoveredtreesremark}
If $p:E\to X$ is a generalized covering map, then $p$ is a Serre fibration with unique path lifting. However, the fibers of $p$ need not be homeomorphic to each other. In fact, some fibers will may be discrete while others are non-discrete. A fiber $p^{-1}(x)$ is discrete if and only if for every $e\in p^{-1}(x)$, there exists an open neighborhood $U_e$ of $x$ in $X$ such that if $i:U_e\to X$ is the inclusion, then $i_{\#}(\pi_1(U_e,x))\leq p_{\#}(\pi_1(E,e))$. Moreover, if there exists a path-connected and locally path-connected open neighborhood $U$ of $x$ such that $i_{\#}(\pi_1(U,x))\leq p_{\#}(\pi_1(E,e))$ for all $e\in p^{-1}(x)$, then $U$ is evenly covered by $p$ (it is straightforward to check that $N(e,U)\cap N(e',U)=\emptyset$ when $e\neq e'$ and $p$ maps $N(e,U)$ homeomorphically onto $U$). Hence, if $X$ is a locally path-connected space and $X$ is semilocally simply connected at $x$, then $x$ admits a neighborhood $U$ that is evenly covered by $p$. For example, if $(X,\scrp)$ is a disjointly tree-graded and $p:E\to X$ is a generalized covering map, then each connected component $U$ of $\tree(X)$ is evenly covered by $p$.
\end{remark}

\begin{remark}[pullbacks of generalized covering maps]\label{pullbackremark}
If $p:(E,e_0)\to (X,x_0)$ is a based generalized covering map with $H=p_{\#}(\pi_1(E,e_0))$ and $f:(Y,y_0)\to (X,x_0)$ is a based map from a path-connected space $Y$, then there is a pullback generalized covering $q:(E',e_0')\to (Y,y_0)$ with $q_{\#}(\pi_1(E',e_0'))=f_{\#}^{-1}(H)$ \cite[Lemma 2.34]{Brazcat}. We briefly recall the construction and some relevant facts. Let $Y\times_{X}E=\{(y,e)\in Y\times E\mid p(e)=f(y)\}$ be the ordinary pullback and $P$ be the path-component of $(y_0,e_0)$ in $Y\times_{X}E$. We set $E'=\lpc(P)$ and $e_0'=(y_0,e_0)$ and we let $q:E'\to Y$ and $g:E'\to E$ be the restrictions of the projection maps so the following diagram of continuous functions commutes (note that $g:E'\to Y$ is continuous since the identity function $id:E'\to P$ is continuous). \[\xymatrix{
(E',e_0') \ar[d]_-{q} \ar[r]^-{id} & (P,(y_0,e_0)) \ar[dl]^-{q} \ar[r]^-{g} & (E,e_0) \ar[d]^-{p}\\
(Y,y_0)\ar[rr]_-{f} && (X,x_0)
}\]
Using the lifting property of $p$ and the universal properties of pullbacks and locally path-connected coreflections, it is straightforward to check that $q_{\#}(\pi_1(E',e_0'))=f_{\#}^{-1}(H)$ and that $q$ has the following lifting property: if $(Z,z_0)$ is a path-connected and locally-path-connected based space and $h:(Z,z_0)\to (Y,y_0)$ and $k:(Z,z_0)\to (E,e_0)$ are maps such that $p\circ k=f\circ h$, then there is a unique based map $\wt{h}:(Z,z_0)\to (E',e_0')$ such that $q\circ \wt{h}=h$ and $g\circ \wt{h}=k$.
\[\xymatrix{
& E' \ar[r]_-{g} \ar[d]^-{q} & E \ar[d]^-{p}\\
Z \ar@/^2.7pc/[urr]^-{k} \ar@{-->}[ur]^-{\wt{h}} \ar[r]_-{h} & Y \ar[r]_-{f} & X
}\]
We are particularly interested in the following application of such pullbacks: If $p$ is a generalized universal covering map and $f$ is $\pi_1$-injective, then $q$ is also a generalized universal covering map. Moreover, if $f$ is an injective map, then the projection map $Y\times_{X}E\to E$ is injective and thus $g:E'\to E$ is injective as well. Note that if $f$ is bijective, then the full projection map $Y\times_{X}E\to E$ is bijective; however, the restricted projection $g:E'\to E$ need not be surjective.
\end{remark}

\begin{lemma}\label{restrictguclemma}
Let $X$ be path-connected and $p:\tX\to X$ be a based generalized universal covering map and $A$ be a retract of $X$. If $\wt{A}$ is a path-component of $p^{-1}(A)$, then $p|_{\wt{A}}:\wt{A}\to A$ is a generalized universal covering map.
\end{lemma}

\begin{proof}
Fix $x_0\in A$ and since $A$ is path-connected and $\wt{A}$ is a path-component of $p^{-1}(A)$, we may use path lifting to to choose $\wt{x}_0\in \wt{A}\cap p^{-1}(x_0)$. Let $i:A\to X$ be the inclusion and $r:X\to A$ be a retraction satisfying $r\circ i=id_{A}$. Since $i$ is $\pi_1$-injective, Remark \ref{pullbackremark} implies that the pullback based generalized covering map $q:(E,e_0)\to (A,x_0)$ of $p:(\tX,\wt{x}_0)\to (X,x_0)$ along $i:(A,x_0)\to (X,x_0)$ is a generalized universal covering map. Let $\wt{i}:E\to \tX$ be the other projection map so that $p\circ \wt{i}=i\circ q$. Since $\tX$ is simply connected, the map $r\circ p:(\tX,\wt{x}_0)\to (A,x_0)$ lifts to a map $\wt{r}:(\tX,\wt{x}_0)\to (E,e_0)$. Since $r\circ i=id_{A}$, uniqueness of lifts ensures that $\wt{r}\circ \wt{i}=id_{E}$. Thus $\wt{i}$ is a embedding. 
\[\xymatrix{
E  \ar[d]^{q} \ar[r]^-{\wt{i}} & \wt{X}   \ar[d]^{p} \ar@{-->}[r]^-{\wt{r}} & E \ar[d]^{q} \\
A \ar[r]_-{i} & X \ar[r]_-{r} & A
}\]
Lastly, we note that $\wt{i}(E)=\wt{A}$ follows from the assumption that $\wt{A}$ is a path-component of $p^{-1}(A)$ and our choice of basepoint $\wt{x}_0\in \wt{A}$. 
\end{proof}

When $(X,d)$ is a metric space and $p:E\to X$ is a generalized covering map, the metric on $X$ lifts naturally to a metric on $E$ as follows: Given $a,b\in E$, let $\mcp(E,a,b)$ denote the set of paths in $E$ from $a$ to $b$. Define \[\wt{d}(a,b)=\inf\{\diam_{d}(\im(p\circ\wt{\alpha}))\mid \wt{\alpha}\in\mcp(E,a,b) \}.\]

\begin{lemma}\label{liftedmetriclemma}
If $p:E\to X$ is a generalized covering map and $d$ is a compatible metric on $X$, then $\wtd$ is a compatible path-diameter metric on $E$ such that $p:(E,\wtd)\to (X,d)$ is non-expansive. Moreover, if $Y$ is a path-connected space, $\wt{f}:Y\to E$ is a map, and $f=p\circ\wt{f}$, then $\diam(\im(f))=\diam(\im(\wt{f}))$.
\end{lemma}

\begin{proof}
Since $\mcp(E,b,a)=\{\wt{\alpha}^{-1}\mid \wt{\alpha}\in \mcp(E,a,b)\}$, we have $\wtd(a,b)=\wtd(b,a)$. If $\wt{\alpha}\in \mcp(E,a,b)$ and $\wt{\beta}\in \mcp(E,b,c)$, then $\wt{\alpha}\cdot\wt{\beta}\in \mcp(E,a,c)$ where $\diam(\im(p\circ(\wt{\alpha}\cdot\wt{\beta})))\leq \diam(\im(p\circ\wt{\alpha}))+\diam(\im(p\circ\wt{\beta}))$. The triangle inequality follows. If $a=b$, then the constant path gives $\wtd(a,b)=0$. Suppose $a\neq b$ in $E$. If $p(a)\neq p(b)$, then any path between the points has diameter at least $d(p(a),p(b))$ and thus $\wtd(a,b)>0$. Suppose $p(a)=p(b)=x$ and, to obtain a contradiction, that $\wt{d}(a,b)=0$. Then there exists a sequence of paths $\wt{\alpha_n}:\ui\to E$, $n\in\bbn$ from $a$ to $b$ such that $\alpha_n=p\circ\wt{\alpha}_n$ is a loop based at $x$ and $\{\diam(\im(\alpha_n))\}_{n\in\bbn}\to 0$. Let $L_n\subseteq \bbr^2$ be the line segment with endpoints $(0,0)$ and $h_n=(1/n,1/n^2)$ and set $H=\bigcup_{n\in\bbn}L_n$. Define $f:H\to X$ to be the path $\alpha_n$ on $L_n$. Since the diameters of the paths $\alpha_n$ are null, $f$ is continuous. Moreover, since $H$ is a dendrite, there is a unique lift $\wt{f}:(H,(0,0))\to (E,a)$, which is defined to be $\wt{\alpha_n}$ on $L_n$. But $\{h_n\}\to (0,0)$ in $H$, $\wt{f}(0,0)=a$, and $\wt{f}(h_n)=b$ for all $n\in\bbn$. This contradicts the fact from Remark \ref{fibersremark} that $p^{-1}(x)$ is $T_1$. Thus $\wtd(a,b)>0$ and we conclude that $\wtd$ is a metric.

To see that $\wtd$ induces the topology of $E$, notice that for every $e\in E$ and $\delta>0$, we have $N(e,B_{\delta/2}(p(e)))\subseteq B_{\delta}(e)\subseteq N(e,B_{2\delta}(p(e)))$ where the middle set is the open ball in the metric $\wt{d}$.

To see that $p:(\tX,\wtd)\to (X,d)$ is non-expansive, note that $d(p(a),p(b))\leq \diam(\im(p\circ\wt{\alpha}))$ for all $\wt{\alpha}\in \mcp(E,a,b)$. Hence, $d(p(a),p(b))\leq \wtd(a,b)$.

For the last statement, we have $\diam(\im(f))\leq\diam(\im(\wt{f}))$ since $p$ is non-expansive. Let $y_0,y_1\in Y$ and find a path $\gamma:\ui\to Y$ from $y_0$ to $y_1$. Then $\wt{f}\circ \gamma\in \mcp(E,\wt{f}(y_0),\wt{f}(y_1))$ and $\diam(f\circ\gamma)\leq \diam(\im(f))$. Thus $\wtd(\wt{f}(y_0),\wt{f}(y_1))\leq \diam(\im(f))$, proving that $\diam(\im(\wt{f}))\leq\diam(\im(f))$.

Finally, we check that $\wtd$ is a path-diameter metric. Given $a,b\in E$ and $\wt{\alpha}\in \mcp(E,a,b)$, we have $\diam(\im(\wt{\alpha}))=\diam(\im(p\circ\wt{\alpha}))$ by the previous paragraph. Thus\[\wtd(a,b)=\inf\{\diam(\im(p\circ\wt{\alpha}))\mid \wt{\alpha}\in \mcp(E,a,b)\}=\inf\{\diam(\im(\wt{\alpha}))\mid \wt{\alpha}\in \mcp(E,a,b)\}.\]
\end{proof}

We refer to the metric $\wt{d}$ on $E$ as the \textit{lifted metric}. Shifting our attention to the existence of generalized covering spaces, we note that the $1$-$UV_0$ property provides a useful sufficient condition.

\begin{theorem}\cite[Theorem 6.9]{BFTestMap}\label{uvzeroimpliesgenunicovtheorem}
If $X$ is a $1$-$UV_0$ metrizable space, then there exists a generalized universal covering map $p:\tX\to X$.
\end{theorem}



The next two statements ensure that the (uniformly) $1$-$UV_0$ property always lifts to a generalized covering space. The first is straightforward to prove using Lemma \ref{uvsequencelemma} and the second follows easily from the last statement of Lemma \ref{liftedmetriclemma}.

\begin{lemma}
Suppose $(X,d)$ is a metrizable space and $p:E\to X$ is a generalized covering map. If $X$ is $1$-$UV_0$, then $E$ is $1$-$UV_0$. Moreover, the converse holds if $X$ is locally path-connected.
\end{lemma}

\begin{lemma}\label{uniformuvliftlemma}
Suppose $(X,d)$ is a metric space, $A\subseteq X$ is a metric subspace, and $p:E\to X$ is a generalized covering map. Then $A$ is uniformly $1$-$UV_0$ if and only if $p^{-1}(A)$ is uniformly $1$-$UV_0$ with respect to the subspace metric inherited from $(E,\wtd)$.
\end{lemma}

Now we focus our attention to generalized coverings of disjointly tree-graded spaces. One might suspect that if $p:E\to X$ is a generalized (regular) covering map where $X$ is disjointly tree-graded, then $E$ inherits a disjoint tree-grading where the pieces of $E$ are the path-components of the preimages $p^{-1}(P)$ of the pieces $P$ of $X$. However, this is false even in simple scenarios.

\begin{example}
 Let $X$ be the union of two disjoint circles $P_1,P_2$ adjoined by a single arc $A$ where the basepoint $x_0$ is the intersection point of $A$ and $P_1$. $X$ is disjointly tree-graded by taking the pieces to be the two circles. Let $a$ be the homotopy class of a loop traversing $P_1$ and let $b$ the homotopy class of the loop moving along $A$, traversing $P_2$ and then moving back along $A$ to the basepoint. Consider the 2-fold covering map $p:E\to X$ corresponding to the normal subgroup $H=\lb a^2,ba^{-1},a^{-1}b\rb$ of $\pi_1(X,x_0)$. Then $p^{-1}(P_1)$ and $p^{-1}(P_2)$ are both topological circles and $p^{-1}(A)$ consists of two arcs both having endpoints on these circles. Hence, $E$ admits no non-trivial disjoint tree-grading.
\end{example}

\begin{lemma}\label{notreducelemma}
Let $(X,\scrp)$ be a disjointly tree-graded space and $p:E\to X$ be a generalized covering map. If $\alpha:S^1\to X$ is an inessential loop whose image meets $\tree(X)$, then no lift of $\alpha$ to $E$ is injective.
\end{lemma}

\begin{proof}
Suppose $\alpha:S^1\to X$ is inessential, $\im(\alpha)\cap \tree(X)\neq \emptyset$, and that $\wt{\alpha}:S^1\to E$ is an injective lift of $\alpha$. Since each connected component $U$ of $\tree(X)$ is evenly covered by $p$ (recall Remark \ref{evenlycoveredtreesremark}), $\alpha$ must be tree-efficient in the sense of Definition \ref{treeefficientdef}. Let $J$ be a connected component of $\alpha^{-1}(\tree(X))$. Then there is a connected component $U$ of $\tree(X)$ such that $\alpha|_{\ov{J}}:\ov{J}\to \ov{U}$ is injective. Fix $s\in J$ and let $x=\alpha(s)$. Since $x$ separates $X$, there must be at least one other connected component $K$ of $\alpha^{-1}(\tree(X))$ and a point $t\in K$ such that $\alpha(t)=x$.

Since $\alpha$ is an inessential loop, there is a map $g:\bbd\to X$ such that $g|_{S^1}=\alpha$. Since $x$ separates $X$ and $g^{-1}(x)\cap S^1$ has at least two components (by the previous paragraph), some connected component $C$ of $g^{-1}(x)$ must separate $\bbd^2$. Pick points $a,b\in S^1$ that lie in distinct components of $C\cap S^1$. Let $A$ be one of the two arcs in $S^1$ with endpoints $a$ and $b$. Define a map $h:\bbd\to X$ to agree with $g$ on $S^1\cup g^{-1}(x)$ and also on any connected component of $\bbd\backslash g^{-1}(x)$ that meets $A$. If $V$ is a connected component of $\bbd\backslash g^{-1}(x)$ that does not meet $A$, then we set $h(V)=x$. Then $h:\bbd\to X$ is a continuous map that gives a null-homotopy of the loop $\alpha|_{A}$ based at $x$. However, this implies that the lift $\wt{\alpha}$ maps the endpoints of $A$ to a single point, contradicting the assumption that $\wt{\alpha}$ is injective.
\end{proof}

\begin{lemma}\label{finitegencovlemma}
Let $(X,\scrp)$ be a disjointly tree-graded space with finitely many non-degenerate pieces and let $p:\tX\to X$ be a generalized universal covering map. Then there exists a disjoint tree-grading $\wt{\scrp}$ on $\tX$ such that
\begin{enumerate}
\item a piece of $\tX$ is a path-component of the preimage $p^{-1}(P)$ of some piece $P$ of $X$,
\item every piece of $\tX$ is simply connected,
\item $p:(\tX,\wt{\scrp})\to (X,\scrp)$ is grade-preserving,
\item if $\wt{P}\in\wt{\scrp}$ is the path-component of $p^{-1}(P)$ for $P\in\scrp$, then the restriction $p|_{\wt{P}}:\wt{P}\to P$ is a generalized universal covering map.
\end{enumerate}
\end{lemma}

\begin{proof}
Defining the elements of $\wt{\scrp}$ to satisfy (1) in the statement of the lemma, we have $\bigcup\wt{\scrp}=p^{-1}(\pc(X))$. Thus $\bigcup\wt{\scrp}$ is closed in $\tX$. The elements of $\wt{\scrp}$ are pairwise-disjoint and path-connected by construction. Since the elements of $\scrp$ are the path-components of $\pc(X)$, it follows that the elements of $\wt{\scrp}$ are the path-components of $\bigcup\wt{\scrp}$. Suppose $\wt{\alpha}:S^1\to \tX$ is an injective loop. Since $\tX$ is simply connected, $\wt{\alpha}$ and $\alpha=p\circ\wt{\alpha}$ are inessential. However, Lemma \ref{notreducelemma} implies that $\im(\alpha)\subseteq \pc(X)$. In particular, $\alpha$ must have image in a single piece $P\in\scrp$. Thus $\wt{\alpha}$ must have image in some path-component of $p^{-1}(P)$. Hence, every simple closed curve in $\tX$ lies in some element of $\wt{\scrp}$.

Next, we show that the elements of $\wt{\scrp}$ are closed in $\tX$. Since the pieces of $X$ are closed, it suffices to show that if $\wt{P}\in\wt{\scrp}$ is a path-component of $p^{-1}(P)$ for $P\in\scrp$, then $\wt{P}$ is closed in $p^{-1}(P)$. Let $P_1,P_2,\dots,P_n$ be the non-degenerate pieces of $X$. If $P=\{x\}\in\scrp $ is a degenerate piece of $X$, then the connected component $U$ of $X\backslash \bigcup_{i=1}^{n}P_i$ containing $x$ is a topological $\bbr$-tree. According to Remark \ref{evenlycoveredtreesremark}, $p^{-1}(x)$ is discrete and the conclusion is clear. Therefore, we may focus our attention to a path-component $\wt{P}$ of $p^{-1}(P_i)$ for fixed non-degenerate piece $P_i$.

Before proceeding, we make two observations. First, note that we may describe a path-component $\wt{P}$ of $p^{-1}(P_i)$ explicitly as follows: Fix a point $e_0\in \wt{P}$. Then \[\wt{P}=\{\wt{\eta}(1)\mid \eta:\ui\to P_i\text{ is a path with }\wt{\eta}(0)=e_0 \}\]where $\wt{\eta}$ denotes the lift of $\eta$ starting at $e_0$. Second, let $V$ be the connected component of $X\backslash \bigcup_{j\neq i}P_j$ containing $P_i$. Then $V$ is open in $X$ and Corollary \ref{isocor} gives that the inclusion $P_i\to V$ induces an isomorphism on fundamental groups.

Consider a convergent sequence $\{\wt{a}_n\}_{n\in\bbn}\to \wt{b}$ in $p^{-1}(P_i)$ where $\wt{a}_n\in \wt{P}$ for all $n\in\bbn$. Write $x_0=p(\wt{b})$ and $a_n=p(\wt{a}_n)$ and note that $\{a_n\}_{n\in\bbn}\to x_0$ is a convergent sequence in $P_i$. Choose a point $\wt{a}\in \wt{P}\cap p^{-1}(x_0)$. Since $p^{-1}(V)$ is an open neighborhood of $p^{-1}(P_i)$ and $\tX$ is locally path-connected, we may find a path-connected neighborhood $W$ of $\wt{b}$ such that $W\subseteq p^{-1}(V)$. Fix an $n$ sufficiently large so that $\wt{a}_n\in W$. Let $\wt{\alpha}:\ui\to \wt{P}$ be a path from $\wt{a}$ to $\wt{a}_n$ and $\wt{\beta}:\ui\to W$ be a path from $\wt{a}_n$ to $\wt{b}$. Let $\wt{\gamma}=\wt{\alpha}\cdot\wt{\beta}$ and note that $\gamma=p\circ \wt{\gamma}$ is a loop in $V$ based at $x_0\in P$. Since the inclusion $P_i\to V$ is $\pi_1$-surjective, $\gamma$ is path-homotopic to a loop $\delta:\ui\to P_i$ based at $x_0$. But if $\wt{\delta}:\ui\to \tX$ is the lift of $\delta$ starting at $\wt{a}$, then $\wt{\delta}(1)=\wt{b}$. Thus $\wt{b}\in\wt{P}$, completing the proof that $\wt{P}$ is closed.

Having now established that $\wt{\scrp}$ is a disjoint tree-grading on $\wt{X}$, it is clear that $p$ is grade-preserving. Moreover, since $\tX$ is simply connected and the pieces of $\tX$ are retracts of $\tX$ by Lemma \ref{retractionlemma}, each piece of $\tX$ is simply connected. For (4), note that since each piece $\wt{P}$ of $\tX$ is a retract of $\tX$, $\wt{P}$ is locally path-connected and simply connected. If $\wt{P}$ is a path-component of $p^{-1}(P)$ for $P\in\scrp$, then it follows from Lemma \ref{restrictguclemma} that $p|_{\wt{P}}:\wt{P}\to P$ is a generalized universal covering map.
\end{proof}

\section{Inverse limits of disjointly tree-graded spaces}

Here, we show that certain inverse limits of disjointly tree-graded spaces inherit a natural disjoint tree-grading. A basic fact we will use is the following.

\begin{lemma}\label{limitsoftreeslemma}
Let $\varprojlim_{n\in\bbn}(T_n,g_{n+1,n})$ be the inverse limit of a sequence of uniquely arcwise-connected Hausdorff spaces $T_n$ and continuous functions $g_{n+1,n}:T_{n+1}\to T_n$. Then $\varprojlim_{n\in\bbn}(T_n,g_{n+1,n})$ contains no simple closed curve. Moreover, if $T_{lim}$ is a path component of $\varprojlim_{n\in\bbn}(T_n,g_{n+1,n})$, then $\lpc(T_{lim})$ is a topological tree.
\end{lemma}

\begin{proof}
Let $g_n$ denote the projection maps of the limit. If $C\subseteq \varprojlim_{n\in\bbn}(T_n,g_{n+1,n})$ is a simple closed curve, then $D_n=g_n(C)$ is a dendrite and $C$ is homeomorphic to the inverse limit of the sequence $(D_n,(g_{n+1,n})|_{D_{n+1}})$ of dendrites. However, standard results in Continuum Theorem such as \cite[Exercise 2.19]{Nadler} can be used to show that an inverse limit of dendrites cannot contain a simple closed curve.
\end{proof}

\begin{definition}
For each $n\in\bbn$, let $(X_{n},\scrp_n)$ be a disjointly tree-graded space and $f_{n+1,n}:(X_{n+1},\scrp_{n+1})\to (X_n,\scrp_n)$ be a grade-preserving map. We say that a sequence $\{P_n\}_{n\in\bbn}$ of pieces $P_n\in \scrp_n$ is \textit{coherent} if $f_{n+1,n}(P_{n+1})\subseteq P_n$ for all $n\in\bbn$.
\end{definition}

\begin{theorem}\label{inversesystemthm}
For each $n\in\bbn$, let $(X_{n},\scrp_n)$ be disjointly tree-graded space and $f_{n+1,n}:(X_{n+1},\scrp_{n+1})\to (X_n,\scrp_n)$ be a grade-preserving map. Suppose that for every coherent sequence of pieces $\{P_n\}_{n\in\bbn}$, $f_{n+1,n}$ maps $P_{n+1}$ homeomorphically onto $P_n$ for all but finitely many $n\in\bbn$.

Let $X_{lim}$ be the path-component of a fixed basepoint $(x_n)_{n\in\bbn}$ in $\varprojlim_{n}(X_n,f_{n+1,n})$ and let $X_{\infty}=\lpc(X_{lim})$. Then $X_{\infty}$ admits a disjoint tree-grading as follows: a piece of $X_{\infty}$ has the form $X_{lim}\cap\prod_{n\in\bbn}P_n $ where $\{P_n\}_{n\in\bbn}$ is a coherent sequence of pieces $P_n\in\scrp_n$. Moreover, the projection maps $f_n:X_{\infty}\to X_n$ are grade-preserving and, for each piece $P$ of $X_{\infty}$, $f_n$ maps $P$ homeomorphically onto a piece of $X_n$ for all but finitely many $n\in\bbn$.
\end{theorem}

\begin{proof}
Let $f_n:X_{lim}\to X_n$ denote the restriction of the $n$-th projection map of the limit. Since limits of inverse sequences of metrizable spaces are metrizable, $\varprojlim_{n}(X_n,f_{n+1,n})$ is metrizable. Recall that $X_{\infty}$ and $X_{lim}$ have the same underlying set and note that the continuous inclusion function $\lpc(X_{lim})\to \varprojlim_{n}(X_n,f_{n+1,n})$ is a continuous bijection onto $X_{lim}$. It follows from Lemma \ref{liftedmetriclemma} that $X_{\infty}$ is a path-connected, locally path-connected, metrizable space.

We construct the collection of pieces $\scrp_{\infty}$ of $X_{\infty}$ as described in the statement of the theorem: a piece $P\in\scrp_{\infty}$ has the form of a non-empty set $P=X_{lim}\cap\prod_{n\in\bbn}P_n$ where $\{P_n\}_{n\in\bbn}$ is a coherent sequence of pieces $P_n\in\scrp_n$. Since the identity function $id:X_{\infty}\to X_{lim}$ is a continuous bijection, if a set $S\subseteq X_{lim}$ is closed, then $S$ is also closed in $X_{\infty}$. Additionally, a set $S\subseteq X_{lim}$ is path-connected (resp. a simple closed curve) if and only if $S$ is path-connected (resp. a simple closed curve) as a subspace of $X_{\infty}$. Hence, the remainder of the conditions that must be checked (to verify that $\scrp_{\infty}$ is a disjoint tree-grading on $X_{\infty}$), may be checked within $X_{lim}$.

Suppose $\{Q_n\}_{n\in\bbn}$ is a coherent sequence of pieces such that $Q=X_{lim}\cap\prod_{n\in\bbn}Q_n $ and $P=X_{lim}\cap\prod_{n\in\bbn}P_n$ both contain a point $(a_n)_{n\in\bbn}$. Then $a_n\in P_n\cap Q_n$ for all $n$ but since the pieces of $\scrp_n$ are pairwise-disjoint, we must have $P_n=Q_n$ for all $n\in\bbn$. Thus $P=Q$. We conclude that the elements of $\scrp_{\infty}$ are pairwise-disjoint.

Next, note that since $P_n$ is closed in $X_n$ for all $n$, it follows that $P$ is closed in $X_{lim}$. Since the bonding maps $f_{n+1,n}$ map $P_{n+1}$ homeomorphically onto $P_n$ for all but finitely many $n$, $f_n:X_{lim}\to X_n$ maps $P$ homeomorphically onto $P_n$ for all but finitely many $n$. Find sufficiently large $n$ and embedding $i:P_n\to X_{lim}$ onto $P$ so that $f_n\circ i=id_{P_n}$. Since $P_n$ is locally path-connected, $i:P_n:\to X_{\infty}$ is also continuous. Thus $f_n:X_{\infty}\to X_n$ maps $P$ homeomorphically onto $P_n$ as well.

Since $\pc(X_n)$ is closed in $X_n$ for each $n$, the set $Y=X_{lim}\cap\prod_{n\in\bbn}\pc(X_n)$ is closed in $X_{lim}$. The inclusion $\bigcup\scrp_{\infty}\subseteq Y$ is clear. If $(a_n)_{n\in\bbn}\in Y$, then $a_n\in P_n$ for some piece $P_n$ of $X_n$. Since $f_{n+1,n}(a_{n+1})=a_{n}\in P_n$ and $f_{n+1,n}$ preserves grading, it must be that $f_{n+1,n}(P_{n+1})\subseteq P_n$. Hence $\{P_n\}_{n\in\bbn}$ is a coherent sequence for which $(a_n)_{n\in\bbn}\in P=  X_{lim}\cap\prod_{n\in\bbn}P_n$. Thus $(a_n)_{n\in\bbn}\in \bigcup\scrp_{\infty}$, proving that $\bigcup\scrp_{\infty}=Y$. We conclude that $\bigcup\scrp_{\infty}$ is closed in $X_{lim}$ and $f_{n}(\bigcup\scrp_{\infty})\subseteq \pc(X_n)$ for all $n\in\bbn$.

Next, suppose that $P$ and $Q$ are distinct elements of $\scrp_{\infty}$ as described above. Suppose that $\alpha:\ui\to \bigcup\scrp_{\infty}$ is a path with $\alpha(0)\in P$ and $\alpha(1)\in Q$. Then for each $n\in\bbn$, $f_{n}\circ\alpha:\ui\to X_n$ is a path in $\pc(X_n)$. Since paths in $\pc(X_n)$ must lie in a single piece of $X_n$, we must have $P_n=Q_n$ for all $n\in\bbn$. Hence, $P=Q$. This proves that the elements of $\scrp_{\infty}$ are the path-components of $\scrp_{\infty}$.

For the simple closed curve condition, we appeal to parameterizations. Let $(T_n,V_n,q_n)$ be a parameterization of $(X_n,\scrp_n)$. Recall from Lemma \ref{paramterizationlemma2} that we may assume that $q_n$ is a topological quotient map for each $n\in\bbn$. Then there exists induced maps $g_{n+1,n}:T_{n+1}\to T_n$ such that $g_{n+1,n}\circ q_{n+1}=q_n\circ f_{n+1,n}$ and $g_{n+1}(V_{n+1})\subseteq V_n$. Let $h_{n+1,n}:V_{n+1}\to V_n$ be the restriction of $g_{n+1,n}$. We have the following morphisms of inverse systems where $i_n:V_n\to T_n$ is the inclusion map.
\[\xymatrix{
X_{lim} \ar[d]^-{q_{\infty}} \ar[r] & \varprojlim_{n}X_n \ar[d]^-{\varprojlim_n q_n}  & \cdots \ar[r]^-{f_{4,3}} & X_3 \ar[d]^-{q_3} \ar[r]^-{f_{3,2}} & X_2 \ar[d]^-{q_2}  \ar[r]^-{f_{2,1}} & X_1 \ar[d]^-{q_1} \\
T_{lim} \ar[r] & \varprojlim_{n}T_n  & \cdots \ar[r]^-{g_{4,3}} & T_3 \ar[r]^-{g_{3,2}} & T_2 \ar[r]^-{g_{2,1}} & T_1\\
V_{lim} \ar[u]_-{i_{\infty}} \ar[r] & \varprojlim_{n}V_n \ar[u]_-{\varprojlim_n i_n}   & \cdots \ar[r]^-{h_{4,3}} & V_3 \ar[u]_-{i_3} \ar[r]^-{h_{3,2}} & V_2 \ar[u]_-{i_2}  \ar[r]^-{h_{2,1}} & V_1 \ar[u]_-{i_1} \\
}\]
On the left, we let $q_{\infty}$ be the restriction of $\varprojlim_n q_n$ to $X_{lim}$ and $T_{lim}=\im(q_{\infty})$. Note that $\varprojlim_n i_n$ is a closed embedding and so we define $V_{lim}=T_{lim}\cap \im (\varprojlim_n i_n)$. Then the three unlabeled maps are inclusion maps. In the next paragraph, we show that $q_{\infty}$ meets the hypotheses of Lemma \ref{monotonelemma2}.

Note that since each set $V_n$ is closed in $T_n$ and $V_{lim}=T_{lim}\cap \prod_{n\in\bbn}V_n$, that $V_{lim}$ is closed in $T_{lim}$. Also, since an inverse limit of a sequence of totally disconnected spaces is totally disconnected, $\varprojlim_{n}V_n$ and it's subspace $V_{lim}$ are totally disconnected. Using the above diagram, it is straightforward to check that $q_{\infty}^{-1}(V_{lim})=\bigcup\scrp_{\infty}$. Specifically, if $(v_n)_{n\in\bbn}\in V_{lim}$ and $P_n=q_{n}^{-1}(v_n)$, then $q_{\infty}^{-1}((v_n)_{n\in\bbn})=X_{lim}\cap\prod_{n\in\bbn}P_n$. Next, we check that $q_{\infty}$ maps $X_{lim}\backslash \bigcup\scrp_{\infty}$ injectively into $T_{lim}\backslash V_{lim}$. If $(a_n)_{n\in\bbn}\in X_{lim}\backslash \bigcup\scrp_{\infty}$, then there exists $N\in\bbn$ such that $a_n\in \tree(X_n)$ for all $n\geq N$. Since $q(a_N)\notin V_n$, we have $q_{\infty}((a_n)_{n\in\bbn})=(q_n(a_n))_{n\in\bbn}\notin V_{lim}$. Moreover, suppose $(b_n)_{n\in\bbn}\in X_{lim}\backslash \bigcup\scrp_{\infty}$ and $q_{\infty}((a_n)_{n\in\bbn})=q_{\infty}((b_n)_{n\in\bbn})$. Find $M\geq N$ such that for all $n\geq M$, we have $b_n\in \tree(X_n)$. If $n\geq M$, we have $q_n(a_n)=q_n(b_n)$ where $a_n,b_n\in \tree(X_n)$. But $q_n$ maps $\tree(X_n)$ bijectively onto $T_n\backslash V_n$. Thus $a_n=b_n$ for all $n\geq M$ and we conclude that $(a_n)_{n\in\bbn}=(b_n)_{n\in\bbn}$. Altogether, we see that the fibers of $q_{\infty}$ are always path-connected (either a point or an element of $\scrp_{\infty}$). By Lemma \ref{limitsoftreeslemma}, $T_{lim}$ contains no simple closed curve (and $T_{\infty}$ is a topological tree).

With the hypotheses of Lemma \ref{monotonelemma2} confirmed, it follows that if $\alpha:S^1\to X_{lim}$ is an injective loop, then $q_{\infty}\circ\alpha$ is constant. Every simple closed curve in $X_{lim}$ must have image in some non-degenerate fiber of $q_{\infty}$, which is an element of $\scrp_{\infty}$. With all of the conditions on the set of pieces verified for $X_{lim}$, we conclude that they also hold for $X_{\infty}$. This completes the proof that $\scrp_{\infty}$ is a disjoint tree-grading of $X_{\infty}$.
\end{proof}

In the remainder of this section, we fix a disjointly tree-graded space $(X,\scrp)$ with countably many non-degenerate pieces $P_1,P_2,P_3,\dots$ (and possibly many degenerate pieces). For each $n\in\bbn$, let $\scrf_n=\{P_1,P_2,\dots,P_n\}$ and let $X_{n}$ denote the metric quotient $X_{\scrf_n}$ from Section \ref{sectionmetricquotient} where all pieces $P\in\scrp\backslash\scrf_n$ are collapsed to a point. We identify $P_1,P_2,\dots, P_n$ with their homeomorphic images in $X_n$ and take $\scrp_n$ to denote the disjoint tree-grading on the metric quotient $X_n$ (including images of degenerate pieces). Let $F_n:X\to X_n$ and $f_{n+1,n}:X_{n+1}\to X_n$ denote the canonical grade-preserving metric quotient maps. Let $X_{lim}=\varprojlim_{n}(X_n,f_{n+1,n})$ be the inverse limit and $f_{n}:X_{lim}\to X_n$ be the projection maps of the limit. The maps $F_n$ agree with the bonding maps $f_{n+1,n}$ and thus induce a canonical map $F_{\infty}:X\to X_{lim}$ given by $F_{\infty}(x)=(F_n(x))_{n\in\bbn}$. Let $X_{\infty}=\lpc(X_{lim})$ as in Theorem \ref{inversesystemthm}. Since the disjointly tree-graded space $X$ is locally path-connected by definition, that the universal property of locally path-connected coreflections gives that $F_{\infty}:X\to X_{\infty}$ is continuous.
\[\xymatrix{
X \ar@/_1pc/[dr]^-{F_{n+1}} \ar@/_1.5pc/[ddr]_-{F_n} \ar[r]^-{F_{\infty}} & X_{\infty}  \ar[r]^-{id} & X_{lim} \ar@/^1pc/[dl]_-{f_{n+1}} \ar@/^1.5pc/[ddl]^-{f_n}\\
& X_{n+1} \ar[d]^-{f_{n+1,n}} \\
& X_n
}\]
It is clear that the inverse system of disjointly tree-graded spaces $(X_n,\scrp_n)$ meets the hypotheses of Theorem \ref{inversesystemthm} and so $X_{\infty}$ admits a canonical disjoint tree-grading, which we denote by $\scrp_{\infty}$. Since $X$ is locally path-connected, $F_{\infty}:X\to X_{\infty}$ is continuous. While $F_{\infty}$ need not be a homeomorphism (see Example \ref{nothomeoexample}), we show that it does preserve much of the structure of $X$. In particular, $X_{\infty}$ may be regarded as $X$ with a potentially coarser topology.

\begin{example}\label{nothomeoexample}
Recall the disjointly tree-graded spaces $X$ and $Y$ from Example \ref{bijectionexample}. Let $P_n$ denote the circle in $X$ that is attached to $\ui$ at the point $1/n$. In this case, $F_{\infty}:X\to X_{lim}$ is not a homeomorphism. Indeed, $X$ is not compact. However, $X_{lim}$ is compact since it is an inverse limit of the compact spaces $[0,1]\cup \bigcup_{k=1}^{n}P_k$. Thus, $X_{lim}$ is homeomorphic to $Y$.
\end{example}

\begin{lemma}\label{thetalemma}
The map $F_{\infty}:X\to X_{\infty}$ is bijective, grade-preserving, and maps the pieces of $X$ homeomorphically to the pieces of $X_{\infty}$. Moreover, if $\pc(X)$ is uniformly $1$-$UV_0$, then $F_{\infty}:X\to X_{\infty}$ is $\pi_1$-injective.
\end{lemma}

\begin{proof}
A straightforward analysis of the inverse limit shows that $F_{\infty}$ is bijective and that $F_{\infty}$ maps $P\in\scrp$ to the piece $F_{\infty}(P)=X_{lim}\cap \prod_{n\in\bbn}F_n(P)$ of $X_{\infty}$. Additionally, if $P_n$ is a non-degenerate piece of $X$ and $i:P_n\to X$ is the inclusion map, then $F_{\infty}\circ i:P_n\to X_{\infty}$ is a section with left-inverse given by $r\circ f_{n}:X_{\infty}\to P_n$ where $r:X_n\to P_n$ is the canonical non-expansive retraction.
\[\xymatrix{
X \ar@/^2pc/[rrr]^-{F_n} \ar[r]_-{F_{\infty}} & X_{\infty} \ar[r]_-{id} & X_{lim} \ar[r]_-{f_n} & X_n \ar[d]^-{r} \\
P_n \ar@{=}[rrr] \ar[u]^-{i} &&& P_n
}\]
Hence, $F_{\infty} $ maps the pieces of $X$ homeomorphically onto the pieces of $X_{\infty}$. Note that $F_{\infty}$ meets the required hypotheses of Proposition \ref{pioneinjectiveprop}. Therefore, if we assume that $\pc(X)$ is uniformly $1$-$UV_0$, then $F_{\infty}$ is $\pi_1$-injective by Proposition \ref{pioneinjectiveprop}.
\end{proof}

\begin{theorem}\label{shapeinjectivethm}
Fix $x_0\in X$ and set $x_n=F_n(x_0)$ for all $n\in\bbn$. If $\pc(X)$ is uniformly $1$-$UV_0$, then the canonical homomorphism \[\phi:\pi_1(X,x_0)\to \varprojlim_{n}(\pi_1(X_n,x_n),f_{n+1,n\#})\]
given by $\phi([\alpha])=([F_n\circ\alpha])_{n\in\bbn}$ is injective.
\end{theorem}

\begin{proof}
To verify the injectivity of $\phi$, we construct an inverse limit of generalized covering maps (most of this proof will be devoted to analyzing this construction). Since $\pc(X)$ is uniformly $1$-$UV_0$ and each piece of $X$ is a retract of $\pc(X)$ by a non-expansive retraction, each piece of $X$ is uniformly $1$-$UV_0$ by Corollary \ref{uniformretractcor}. For the moment, fix $n\in\bbn$. Since $X_n$ has finitely many non-degenerate pieces and all of them are $1$-$UV_0$, Lemma \ref{oneuvzerofinitetreelemma} implies that $X_n$ is $1$-$UV_0$. By Theorem \ref{uvzeroimpliesgenunicovtheorem}, there exists a generalized universal covering map $p_n:\tX_n\to X_n$ for each $n\in\bbn$. Lemma \ref{finitegencovlemma} then ensures that $\tX_n$ admits a disjoint tree-grading $\wt{\scrp}_n$ consisting of the path-components of the preimages of the pieces of $X_n$ under $p_n$. In particular, each piece of $\tX_n$ is simply connected and a non-degenerate piece of $\tX_n$ is a path-component of $p_{n}^{-1}(P_i)$ for some $i\in\{1,2,\dots,n\}$. Part (4) of Lemma \ref{finitegencovlemma} ensures that if $\wt{Q}$ is a path-component of $p_{n}^{-1}(P_i)$, $i\in\{1,2,\dots, n\}$, then $p_n$ restricts to a generalized universal covering map $\wt{Q}\to P_i$. Allowing $n$ to vary, we note that since $\wt{X}_{n+1}$ is locally path-connected and simply connected, the bonding map $f_{n+1,n}:X_{n+1}\to X_n$ lifts to a map $\wt{f}_{n+1,n}:\tX_{n+1}\to \tX_n$ satisfying $p_n\circ \wt{f}_{n+1,n}=f_{n+1,n}\circ p_{n+1}$. Since each $f_{n+1,n}$ is grade-preserving, it follows that each lift $\wt{f}_{n+1,n}$ is a grade-preserving map. 

Next, we choose basepoints in the generalized universal covering spaces. Choose $\wt{x}_1\in p_{1}^{-1}(x_1)$ arbitrarily. Recall from Corollary \ref{isocor2} that inclusion maps $P_i\to X_n$, $1\leq i\leq n$ (with some choice of basepoints for the pieces $P_i$ and an appropriate choice of conjugating paths) induce an isomorphism $\ast_{i=1}^{n}\pi_1(P_i)\to \pi_1(X_n)$. Making this identification for each $n$, the induced homomorphism $f_{n+1,n\#}$ becomes identified with the homomorphism $\ast_{i=1}^{n+1}\pi_1(P_i)\to \ast_{i=1}^{n}\pi_1(P_i)$ collapsing the $(n+1)$-st free factor. It follows that the bonding maps $f_{n+1,n}:X_{n+1}\to X_n$ are $\pi_1$-surjective and so each $\wt{f}_{n+1,n}$ maps the fiber $p_{n+1}^{-1}(x_{n+1})$ surjectively onto $p_{n}^{-1}(x_n)$. Hence, we may recursively choose points $\wt{x}_n\in \wt{X}_n$ so that $\wt{f}_{n+1,n}(\wt{x}_{n+1})=\wt{x}_n$ and $x_n=p_n(\wt{x}_n)$ for each $n\in\bbn$. Thus $\wt{x}_{\infty}=(\wt{x}_n)_{n\in\bbn}\in \varprojlim_{n}(\wt{X}_n,\wt{f}_{n+1,n})$. With these basepoints, we obtain the following inverse system of based generalized universal covering maps. Let $\varprojlim_{n}p_n$ be the corresponding inverse limit map (in the topological category).
\[\xymatrix{
\varprojlim_{n}(\wt{X}_n,\wt{f}_{n+1,n}) \ar[d]^-{\varprojlim_{n}p_n} & \cdots \ar[r] & \wt{X}_3  \ar[d]^-{p_3} \ar[r]^-{\wt{f}_{3,2}} & \wt{X}_2 \ar[d]^-{p_2} \ar[r]^-{\wt{f}_{2,1}} &\wt{X}_1 \ar[d]^-{p_1} \\
X_{lim} &\cdots \ar[r] & X_3 \ar[r]_-{f_{3,2}} & X_2 \ar[r]_-{f_{2,1}} & X_1 \\
}\]
Let $\wt{X}_{lim}$ be the path-component of $\wt{x}_{\infty}$ in $\varprojlim_{n}(\wt{X}_n,\wt{f}_{n+1,n})$ and $\wt{X}_{\infty}=\lpc(\wt{X}_{lim})$. 

Again fix $n\in\bbn$. Certainly, $\wt{f}_{n+1,n}$ maps a degenerate piece homeomorphically onto a degenerate piece. Since $f_{n+1,n}$ maps all non-degenerate pieces of $X_{n+1}$, except $P_{n+1}$, homeomorphically to a piece of $X_n$, it is straightforward to check that $\wt{f}_{n+1,n}$ maps each path-component of $p_{n+1}^{-1}(P_{i})$ for $1\leq i\leq n$ homeomorphically onto a piece of $\wt{X}_n$ (continuity of the inverse is essentially the same argument used in the proof of Lemma \ref{restrictguclemma}). Allowing $n$ to vary, note that if have a coherent sequence of pieces $\wt{Q}_n\in \wt{\scrp}_n$ (i.e. such that $\wt{f}_{n+1,n}(\wt{Q}_{n+1})\subseteq \wt{Q}_{n}$ for all $n\in\bbn$) then for sufficiently large $n$, $\wt{f}_{n+1,n}$ maps $\wt{Q}_{n+1}$ homeomorphically onto $\wt{Q}_n$. Hence, the hypotheses of Theorem \ref{inversesystemthm} are met and we conclude that $\wt{X}_{\infty}$ inherits the structure of a disjointly tree-graded space where pieces are non-empty sets of the form $\wt{Q}_{\infty}=\wt{X}_{lim}\cap \prod_{n\in\bbn}\wt{Q}_n$ where $\{\wt{Q}_n\}_{n\in\bbn}$ is a coherent family of pieces. Moreover, Theorem \ref{inversesystemthm} concludes that $\wt{f}_n$ maps $\wt{Q}_{\infty}$ homeomorphically onto $\wt{Q}_n$ for sufficiently large $n$.

Let $p_{lim}:\wt{X}_{lim}\to X_{lim}$ denote the restriction of $\varprojlim_{n}p_n$ to $\wt{X}_{lim}$, let $\wt{f}_n:\wt{X}_{lim}\to \wt{X}_n$ denote the restriction of the $n$-th projection map, and let $p_{\infty}=\lpc(p_{lim}):\wt{X}_{\infty}\to X_{\infty}$. This construction makes $p_{\infty}$ the limit of the sequence of based generalized covering maps $p_n:(\wt{X}_n,\wt{x}_n)\to (X_n,x_n)$ in the category of based generalized covering maps  and thus $p_{\infty}:\wt{X}_{\infty}\to X_{\infty}$ is a generalized covering map (this follows by directly combining the lifting property of generalized covering maps with the universal property of inverse limits but a proof can be found in \cite[Lemma 2.31]{Brazcat}).
\[\xymatrix{
\wt{X}_{\infty} \ar[d]_-{p_{\infty}} \ar[r]^-{id} & \wt{X}_{lim} \ar[d]^-{p_{lim}} & \cdots \ar[r] & \wt{X}_3  \ar[d]^-{p_3} \ar[r]^-{\wt{f}_{3,2}} & \wt{X}_2 \ar[d]^-{p_2} \ar[r]^-{\wt{f}_{2,1}} &\wt{X}_1 \ar[d]^-{p_1} \\
 X_{\infty} \ar[r]_-{id} & X_{lim} & \cdots \ar[r] & X_3 \ar[r]_-{f_{3,2}} & X_2 \ar[r]_-{f_{2,1}} & X_1 \\
}\]
In the next two propositions, we analyze the disjoint tree-grading structure of $\tX_{\infty}$.

\begin{proposition}\label{equalityprop1}
$\pc(\wt{X}_{\infty})=p_{\infty}^{-1}(\pc(X_{\infty}))$. In particular, the pieces of $\tX_{\infty}$ are precisely the path-components of the preimages of the pieces of $X_{\infty}$ under $p_{\infty}$.
\end{proposition}

\begin{proof}
Suppose $(\wt{a}_n)$ is an element of piece $\wt{X}_{lim}\cap \prod_{n\in\bbn}\wt{A}_n$ of $\tX_{\infty}$ where $\wt{A}_n$ is a path-component of $p_{n}^{-1}(A_n)$ for $A_n\in\scrp_n$. Then $\{A_n\}_{n\in\bbn}$ is a coherent sequence and we have $p_{\infty}((\wt{a}_n))\in X_{lim}\cap \prod_{n\in\bbn}A_n\subseteq \pc(X_{\infty})$. This proves one inclusion. For the other, suppose $p_{\infty}((\wt{a}_n))$ lies in piece $X_{lim}\cap \prod_{n\in\bbn}A_n$ of $X_{\infty}$ where $A_n\in \scrp_n$. Let $\wt{A}_n$ be the path-component of $p_{n}^{-1}(A_n)$ containing $\wt{f}_n(\wt{a}_n)$. Now $\{\wt{A}_n\}_{n\in\bbn}$ is a coherent sequence of pieces and thus $(\wt{a}_n)\in \wt{X}_{lim}\cap \prod_{n\in\bbn}\wt{A}_n\subseteq \pc(\tX_{\infty})$.

Since the pieces of $\tX_{\infty}$ are the path-components of $\pc(\wt{X}_{\infty})$ (by definition of disjoint tree-grading), the second statement follows from the first. 
\end{proof}

\begin{proposition}\label{piecesprop1}
If $\wt{A}_{\infty}$ is a piece of $\tX_{\infty}$, then $A_{\infty}=p(\wt{A}_{\infty})$ is a piece of $X_{\infty}$ and the restriction map $p_{\infty}|_{\wt{A}_{\infty}}:\wt{A}_{\infty}\to A_{\infty}$ is a generalized universal covering map. In particular, the pieces of $\wt{X}_{\infty}$ are simply connected.
\end{proposition}

\begin{proof}
That $A_{\infty}$ is a piece of $X_{\infty}$ and that $\wt{A}_{\infty}$ is a path-component of $p^{-1}(A_{\infty})$ follows from Proposition \ref{equalityprop1}. Find $n$ sufficiently large so that (1) $\wt{f}_n$ maps $\wt{A}_{\infty}$ homeomorphically onto a piece $\wt{A}_n$ of $\wt{X}_n$, (2) $f_n$ maps $A_{\infty}$ homeomorphically onto a piece $A_n$ on $X_n$, and (3) $A_n=p_n(\wt{A}_n)$. In the second paragraph of this proof, we established that the restriction map $p_n:\wt{A}_n\to A_n$ is a generalized universal covering map. The following diagram, in which we confuse maps with their restrictions, makes it clear that $p_{\infty}$ restricts to a generalized universal covering map on pieces.
\[\xymatrix{
\wt{A}_{\infty} \ar[d]_-{p_{\infty}} \ar[r]^-{\wt{f}_n}_-{\cong} & \wt{A}_n \ar[d]^-{p_n} \\
A_{\infty} \ar[r]_-{f_n}^-{\cong} & A_n
}\]
\end{proof}

Unfortunately, we cannot conclude that $\wt{X}_{\infty}$ is simply connected using Lemma \ref{simplyconnectedlemma} because it is not clear if $\pc(X_{\infty})$ is uniformly $1$-$UV_0$ and this appears tedious to verify. Hence, we will work to construct an analogous disjointly tree-graded space over $X$ itself. Applying Remark \ref{pullbackremark}, take $p:\wt{X}\to X$ to be the pullback based generalized covering map of $p_{\infty}:\wt{X}_{\infty}\to X_{\infty}$ along the continuous bijection $F_{\infty}:(X,x_0)\to (X_{\infty},(x_n))$. Let $\wt{F}_{\infty}:\wt{X}\to \wt{X}_{\infty}$ be the restriction of the projection map. Since $F_{\infty}$ is bijective, $\wt{F}_{\infty}$ is injective (note that $F_{\infty}$ need not be bijective because $\wt{X}$ is constructed by taking the locally path-connected coreflection of a path-component of the topological pullback).
\[\xymatrix{
\wt{X}  \ar[r]^{\wt{F}_{\infty}} \ar[d]_-{p} & \wt{X}_{\infty}  \ar[r]^-{id} \ar[d]^-{p_{\infty}} & \wt{X}_{lim} \ar[d]^-{p_{lim}} \ar[r]^-{\wt{f}_n} & \wt{X}_n \ar[d]^-{p_n}\\
X \ar@/_1.5pc/[rrr]_-{F_n} \ar[r]^-{F_{\infty}} & X_{\infty} \ar[r]^{id} & X_{lim} \ar[r]^-{f_n} & X_n
}\]

\begin{proposition}\label{equalityprop2}
$p^{-1}(\pc(X))=\wt{F}_{\infty}^{-1}(\pc(\wt{X}_{\infty}))$.
\end{proposition}

\begin{proof}
If $\wt{x}\in \tX$ and $p(\wt{x})$ lies in piece $P$ of $X$, then $p_{\infty}\circ \wt{F}_{\infty}(\wt{x})=F_{\infty}\circ p(\wt{x})$ lies in the piece $F_{\infty}(P)$ of $X_{\infty}$. Using Proposition \ref{equalityprop1}, we have $\wt{x}\in \wt{F}_{\infty}^{-1}(p_{\infty}^{-1}(\pc(X_{\infty})))= \wt{F}_{\infty}^{-1}(\pc(\wt{X}_{\infty}))$. For the other inclusion, suppose $\wt{F}_{\infty}(\wt{x})=(\wt{a}_n)$ lies in piece $\wt{X}_{lim}\cap \prod_{n}\wt{A}_n$ of $\wt{X}_{\infty}$ where $\wt{A}_n$ is the path-component of $p_{n}^{-1}(A_n)$ for piece $A_n$ of $X_n$. Then $(p_n(\wt{a}_n))$ lies in the piece $X_{lim}\cap \prod_{n}A_n$ of $X_{\infty}$. Let $A$ be the unique piece of $X$ that $F_{\infty}$ maps onto $X_{lim}\cap \prod_{n}A_n$. Since $F_{\infty}\circ p(\wt{x})=(p_n(\wt{a}_n))\in X_{lim}\cap \prod_{n}A_n=F_{\infty}(A)$ where $F_{\infty}$ is injective, we must have $p(\wt{x})\in A\subseteq \pc(X)$.
\end{proof}

\begin{proposition}\label{equalityprop3}
A path-component of $p^{-1}(\pc(X))$ is the pre-image of a piece of $\wt{X}_{\infty}$ under $\wt{F}_{\infty}$.
\end{proposition}

\begin{proof}
Let $\wt{A}$ be a path-component of $p^{-1}(A)$ where $A\in\scrp$. It follows from Proposition \ref{equalityprop2} that $\wt{F}_{\infty}(\wt{A})$ must be contained in a piece $\wt{A}_{\infty}$ of $\wt{X}_{\infty}$. We check that $\wt{F}_{\infty}(\wt{A})=\wt{A}_{\infty}$. Set $A_{\infty}=p_{\infty}(\wt{A}_{\infty})$. By Proposition \ref{equalityprop1}, $A_{\infty}$ is a piece of $X_{\infty}$. Hence, $F_{\infty}$ maps $A$ homeomorphically onto $A_{\infty}$. We have the following diagrams where we temporarily confuse the maps with their respective restrictions.
\[\xymatrix{
\wt{A} \ar[d]_-{p} \ar[r]^{\wt{F}_{\infty}}  & \wt{A}_{\infty} \ar[d]^-{p_{\infty}}\\
A \ar[r]_-{F_{\infty}}^{\cong} & A_{\infty}
}\]
Let $\wt{b}_{\infty}\in \wt{A}_{\infty}$. Fix a point $\wt{a}_0\in \wt{A}$ and set $\wt{a}_{\infty}=\wt{F}_{\infty}(\wt{a}_0)$. Find a path $\wt{\gamma}:\ui\to \wt{A}_{\infty}$ from $\wt{a}_{\infty}$ to $\wt{b}_{\infty}$. Using the diagram above, we have a unique path $\delta:\ui\to A$ such that $F_{\infty}\circ\delta=p_{\infty}\circ\wt{\gamma}$. Since $p(\wt{a}_0)=\delta(0)$, we may find a unique path $\wt{\delta}:\ui\to \wt{X}$ with $\wt{\delta}(0)=\wt{a}_0$ and $p\circ\wt{\delta}=\delta$. In particular, since $\delta$ has image in $A$, $\wt{\delta}$ must have image in $\wt{A}$. Now $\wt{F}_{\infty}\circ\wt{\delta}:\ui\to \wt{X}_{\infty}$ is a path starting at $\wt{a}_{\infty}$ and satisfying $p_{\infty}\circ \wt{F}_{\infty}\circ\wt{\delta}=p_{\infty}\circ\wt{\gamma}$. Since $p_{\infty}$ is a generalized covering map, we have $\wt{F}_{\infty}\circ\wt{\delta}=\wt{\gamma}$. Thus, $\wt{F}_{\infty}(\wt{\delta}(1))=\wt{b}_{\infty}$ for $\wt{\delta}(1)\in \wt{A}$, completing the proof that $\wt{A}_{\infty}\subseteq\wt{F}_{\infty}(\wt{A})$.
\end{proof}

Let $\wt{\scrp}$ denote the set of path-components of $p^{-1}(\pc(X))$. Certainly, $\tX$ is path-connected, locally path-connected, and metrizable (as a generalized covering space over a metrizable space). It follows from Proposition \ref{equalityprop3} that each element of $\wt{\scrp}$ is the preimage $\wt{A}=\wt{F}_{\infty}^{-1}(\wt{A}_{\infty})$ where $\wt{A}_{\infty}$ is a piece of $\wt{X}_{\infty}$. Since $\wt{A}_{\infty}$ is closed in $\tX_{\infty}$, $\wt{A}$ is closed in $\tX$. Moreover, since $\pc(\tX_{\infty})$ is closed in $\tX_{\infty}$, $p_{\infty}^{-1}(\pc(\wt{X}_{\infty}))=\bigcup\wt{\scrp}$ is closed in $\tX$. If $C\subseteq \wt{X}$ is a simple closed curve, then $F_{\infty}(C)$ is a simple closed curve in $\tX_{\infty}$ and therefore lies in some piece $\wt{A}_{\infty}$ of $\tX_{\infty}$. Thus $C$ must lie in the preimage $\wt{A}=\wt{F}_{\infty}^{-1}(\wt{A}_{\infty})$, which is an element of $\scrp$. Since all conditions of Definition \ref{deftreegraded} are met, we conclude that $\wt{\scrp}$ is a disjoint tree-grading on $\tX$. Next, we check that each piece of $\tX$ is simply connected.

\begin{proposition}\label{piecesprop2}
$\wt{F}_{\infty}$ maps each element of $\wt{\scrp}$ homeomorphically to a piece of $\wt{X}_{\infty}$. In particular, every element of $\wt{\scrp}$ is simply connected.
\end{proposition}

\begin{proof}
Let $\wt{A}\in \wt{\scrp}$ and $\wt{a}_0\in \wt{A}$. By Proposition \ref{equalityprop3}, $F_{\infty}$ maps $\wt{A}$ by a continuous bijection to a piece $\wt{A}_{\infty}$ of $\wt{X}_{\infty}$. Set $A=p(\wt{A})$ and $A_{\infty}=p_{\infty}(\wt{A}_{\infty})$ and fix basepoints $a_0=p(\wt{a}_0)\in A$, $\wt{b}_0=\wt{F}_{\infty}(\wt{a}_0)\in\wt{A}_{\infty}$, and $b_0=p_{\infty}(\wt{b}_0)\in A_{\infty}$. Since $F_{\infty}$ maps $A$ homeomorphically onto $A_{\infty}$, we may take $g:(A_{\infty},b_0)\to (A,a_0)$ to be the inverse map, which satisfies $F_{\infty}\circ g=id_{A_{\infty}}$. Consider the commutative diagram from the proof of Proposition \ref{equalityprop3}. By Proposition \ref{piecesprop1}, the right map $p_{\infty}:\wt{A}_{\infty}\to A_{\infty}$ is a generalized universal covering map. Since $p$ is a generalized covering map and $\wt{A}_{\infty}$ is simply connected, the map $g\circ p_{\infty}:(\wt{A}_{\infty},\wt{b}_0) \to (A,a_0)$ lifts uniquely to a map $\wt{g}:(\wt{A}_{\infty},\wt{b}_0)\to (\wt{X},\wt{a}_0)$ such that $p\circ \wt{g}=g\circ p_{\infty}$. Since $\wt{g}$ has image in $p^{-1}(A)$ and maps $\wt{b}_0$ to $\wt{a}_0$, it must have image in $\wt{A}$. Moreover, since $p\circ \wt{g}\circ \wt{F}_{\infty}=g\circ F_{\infty}\circ p=p$, uniqueness of lifts gives $\wt{g}\circ \wt{F}_{\infty}=id_{\wt{A}_{\infty}}$. Thus $\wt{F}_{\infty}$ maps $\wt{A}$ homeomorphically onto $\wt{A}_{\infty}$. 

In Proposition \ref{piecesprop1}, we established that the pieces of $\tX_{\infty}$ are simply connected. Hence, the current Proposition is proved.
\end{proof}

Summarizing, we have shown that $p:\tX\to X$ is a generalized covering map and that the pieces of $\tX$ are precisely the path-components of the preimages of the pieces of $X$ under $p$. Since $\pc(\tX)=p^{-1}(\pc(X))$ and $\pc(X)$ is assumed to be uniformly $1$-$UV_0$, Lemma \ref{uniformuvliftlemma} gives that $\pc(\tX)$ is uniformly $1$-$UV_0$ when $\tX$ is equipped with the lifted metric $\wtd$. By Proposition \ref{piecesprop2}, every piece of $\tX$ is simply connected. We now appeal to Lemma \ref{simplyconnectedlemma} to conclude that $\tX$ is simply connected. It follows that $p:\tX\to X$ is a generalized \textit{universal} covering map.

With the above construction completed and the properties of the generalized covering maps verified, we finally address the injectivity of $\phi$. For convenience, we will take loops to be based at the basepoints fixed in the second paragraph of the proof (of Theorem \ref{shapeinjectivethm}). Suppose that $\alpha:\ui\to X$ is a loop based at $x_0$ such that $\beta_n=F_n\circ\alpha$ is inessential in $X_n$ for all $n\in\bbn$. We seek to show that $\alpha$ is essential in $X$. 

Let $\beta=F_{\infty}\circ\alpha:S^1\to X_{\infty}$ and note that $f_n\circ\beta=\beta_n$ for all $n$. Each inessential loop $\beta_n$ lifts to a unique loop $\wt{\beta}_n:S^1\to \wt{X}_n$ based at $\wt{x}_n$. Uniqueness of lifting ensures that $\wt{f}_{n+1,n}\circ\wt{\beta}_{n+1}=\wt{\beta}_n$ for all $n$. Thus there is an induced loop $\wt{\beta}:S^1\to \wt{X}_{lim}$ based at $\wt{x}_{\infty}$ such that $\wt{f}_{n}\circ\wt{\beta}=\wt{\beta}_n$ for all $n\in\bbn$. The universal property of the inverse limit $X_{lim}$ ensures that $p_{lim}\circ \wt{\beta}=\beta$. 
\[\xymatrix{
& \wt{X}_{lim} \ar[d]^-{p_{lim}} \ar[r]_-{\wt{f}_n} & \wt{X}_n \ar[d]^-{p_n}\\
S^1 \ar@{-->}[ur]^-{\wt{\beta}} \ar[r]_-{\beta} \ar@/_2pc/[rr]_-{\beta_n} \ar@/^3pc/[urr]^-{\wt{\beta}_n} & X_{lim} \ar[r]_-{f_n} & X_n
}\]
Since $\wt{X}_{\infty}=\lpc(X_{lim})$, $p_{\infty}=\lpc(p_{lim})$ and $S^1$ is locally path-connected, $\wt{\beta}:S^1\to \wt{X}_{\infty}$ is continuous and satisfies $p_{\infty}\circ\wt{\beta}=\beta$. Since $F_{\infty}\circ\alpha =\beta=p_{\infty}\circ\wt{\beta}$ and $p$ was constructed as the pullback generalized covering map, there is a unique lift $\wt{\alpha}:S^1\to \tX$ based at $\wt{x}_0$ such that $p\circ \wt{\alpha}=\alpha$ and $\wt{F}_{\infty}\circ\wt{\alpha}=\wt{\beta}$.
\[\xymatrix{
 S^1  \ar@/_3pc/[drr]_-{\beta} \ar[dr]_-{\alpha}  \ar@/^2pc/[rr]^-(.84){\wt{\beta}} \ar@/^3pc/[rrr]^-{\wt{\beta}} \ar@{-->}[r]^-{\wt{\alpha}}   & \wt{X}  \ar[r]^{\wt{F}_{\infty}} \ar[d]_-{p} & \wt{X}_{\infty}  \ar[r]^-{id} \ar[d]^-{p_{\infty}} & \wt{X}_{lim} \ar[d]^-{p_{lim}}  \\
      & X \ar[r]^-{F_{\infty}} & X_{\infty} \ar[r]_{id} & X_{lim}
}\]
Since we have proved that $\tX$ is simply connected, $\wt{\alpha}$ is inessential. It follows that $\alpha=p\circ\wt{\alpha}$ is inessential. This completes the proof that $\phi$ is injective.
\end{proof}

A restatement of Theorem \ref{shapeinjectivethm}, including the assumed hypotheses on $X$, is the following. Note that this is precisely Theorem \ref{mainthm1} in the case where $X$ has countably many non-degenerate pieces.

\begin{corollary}\label{countablecase}
Let $(X,\scrp)$ be a disjointly tree-graded space with countably many non-degenerated pieces and such that $\pc(X)$ is uniformly $1$-$UV_0$. If $\alpha:S^1\to X$ is an essential loop, then there exists a finite set $\scrf\subseteq \scrp$ such that if $\Gamma_{\scrf}:X\to X_{\scrf}$ is the metric quotient map, then $\Gamma_{\scrf}\circ\alpha:S^1\to X_{\scrf}$ is essential.
\end{corollary}

\section{A Proof of Theorem \ref{mainthm1}}\label{sectionproofofmain}

Finally, we have all the ingredients to prove our main theorem. Recall the statement from the introduction.

\begin{proof}[Proof of Theorem \ref{mainthm1}]
The ``only if" direction is clear. For the converse, we fix a compatible path-diameter metric $d$ on $X$. Suppose that $\alpha:S^1\to X$ is essential. By Lemma \ref{expansionofpclemma}, there exist a sectional graded-subspace $(Y,\scrq)$ of $(X,\scrp)$ such that $\im(\alpha)\subseteq Y$, $Y$ is closed in $X$, $Y$ has countably many non-degenerate pieces, and such that $\alpha:S^1\to Y$ is essential in $Y$. Since $\pc(X)$ is assumed to be uniformly $1$-$UV_0$ in the subspace metric, Lemma \ref{sectionaluniformlemma} gives that $\pc(Y)$ is uniformly $1$-$UV_0$ in the subspace metric.

Corollary \ref{countablecase} gives the existence of a finite subset $\scrf\subseteq \scrq$ such that if $\Gamma_{Y,\scrf}:Y\to Y_{\scrf}$ is the metric quotient map that collapses all elements of $\scrq\backslash \scrf$ to a point, then $\Gamma_{Y,\scrf}\circ\alpha:S^1\to Y_{\scrf}$ is essential.

Since the inclusion $i:Y\to X$ and the quotient map $\Gamma_{Y,\scrf}$ are non-expansive and $\Gamma_{\scrf}\circ i$ is constant on the fibers of $\Gamma_{Y,\scrf}$, the universal property of $\Gamma_{Y,\scrf}$ induces a non-expansive map $i_{\scrf}:Y_{\scrf}\to X_{\scrf}$ that makes the left square commute in the diagram below. Additionally, if $r:X\to Y$ is the canonical non-expansive retraction from Lemma \ref{retractionlemma3}, then $\Gamma_{Y,\scrf}\circ r$ is non-expansive and constant on the fibers of $\Gamma_{\scrf}$. Hence, there exists a unique non-expansive map $r_{\scrf}:X_{\scrf}\to Y_{\scrf}$ for which the right square commutes. Since $r\circ i=id_{Y}$, the universal property of $\Gamma_{Y,\scrf}$ ensures that $r_{\scrf}\circ i_{\scrf}$ is the identity of $Y_{\scrf}$.
\[\xymatrix{
Y \ar[d]_-{\Gamma_{Y,\scrf}} \ar[r]^-{i} & X \ar[d]^-{\Gamma_{\scrf}} \ar[r]^-{r} & Y  \ar[d]^-{\Gamma_{Y,\scrf}}\\
Y_{\scrf} \ar[r]_{i_{\scrf}} & X_{\scrf} \ar[r]_-{r_{\scrf}} & Y_{\scrf}
}\]
We conclude that $Y_{\scrf}$ is a retract of $X_{\scrf}$ with section $i_{\scrf}$. Since $\Gamma_{Y,\scrf}\circ\alpha:S^1\to Y_{\scrf}$ is essential and $i_{\scrf}$ is $\pi_1$-injective, $\Gamma_{\scrf}\circ\alpha= \Gamma_{\scrf}\circ i\circ \alpha=i_{\scrf}\circ\Gamma_{Y,\scrf}\circ\alpha$ is essential.
\end{proof}

We also note that the topological quotient maps $\Gamma_{\scrf}:X\to X_{\scrf}^{qt}$ may also be used to detect essential loops.

\begin{corollary}\label{maincorollary}
Let $(X,\scrp)$ be a disjointly tree-graded space where $\pc(X)$ is uniformly $1$-$UV_0$. Then a loop $\alpha:S^1\to X$ is essential if and only if there exists a finite set of pieces $\scrf\subseteq \scrp$ such that $\Gamma_{\scrf}\circ \alpha:S^1\to X_{\scrf}^{qt}$ is essential in the topological quotient space $X_{\scrf}^{qt}$.
\end{corollary}

\begin{proof}
By Theorem \ref{mainthm1}, there exists finite subset $\scrf\subseteq\scrp$ such that $\Gamma_{\scrf}\circ \alpha$ is essential in the quotient metric space $X_{\scrf}$. By Lemma \ref{identitypoioneisomorphismlemma}, the continuous identity function $X_{\scrf}^{qt}\to X_{\scrf}$ induces an isomorphism of fundamental groups. Thus $\Gamma_{\scrf}\circ \alpha$ is essential in $X_{\scrf}^{qt}$.
\end{proof}


\section{Applications of Theorem \ref{mainthm1}}\label{sectionapplications}

In our primary application of Theorem \ref{mainthm1}, we identify conditions sufficient to conclude that a grade-preserving map is $\pi_1$-injective. 

\begin{remark}
For a disjointly tree-graded space $(X,\scrp)$, each piece $P$ of $X$ is a retract of $X$ and a path-component of $\pc(X)$. Thus the inclusion $\pc(X)\to X$ is $\pi_1$-injective. It follows that if a grade-preserving map $f:(X,\scrp)\to (Y,\scrq)$ is $\pi_1$-injective, then the restriction $f|_{\pc(X)}:\pc(X)\to \pc(Y)$ is $\pi_1$-injective. The converse need not hold since $f$ might map distinct pieces of $X$ into the same piece of $Y$ in an overlapping fashion.  
\end{remark}

\begin{proof}[Proof of Theorem \ref{mainapplicationcorollary}]
The assumptions on $f$ ensure that $f$ is a grade-preserving map. Let $\alpha:S^1\to X$ be an essential loop based at a point $x_0\in X$. By Corollary \ref{maincorollary}, there exists a finite set of pieces $\scrf\subseteq\scrp$ such that $\Gamma_{\scrf}\circ \alpha:\ui\to X_{\scrf}^{qt}$ is essential. Let $Q_P$ denote the piece of $Y$ such that $f(P)\subseteq Q_P$. Let $\scrg=\{Q_P\mid P\in\scrp\}$. By assumption, the function $\scrf\to \scrg$ given by $P\mapsto Q_P$ is bijective. Since $\Gamma_{\scrg}\circ f:X\to Y_{\scrg}^{qt}$ is constant on the fibers of $\Gamma_{\scrf}:X\to X_{\scrf}^{qt}$, there is an induced map $g:X_{\scrs}^{qt}\to Y_{\scrg}^{qt}$ such that $g\circ \Gamma_{\scrf}=\Gamma_{\scrg}\circ f$. We claim that $g$ is $\pi_1$-injective. Once this is done, it will follow that $g\circ \Gamma_{\scrf}\circ \alpha=\Gamma_{\scrg}\circ f\circ\alpha$ is essential and thus $f\circ\alpha$ is essential.

Set $y_0=f(x_0)$, $x_{\scrf}=\Gamma_{\scrf}(x_0)$, and $y_{\scrg}=g(x_{\scrf})$. For each $P\in\scrp$, fix point $x_P\in P$ and let $y_P=f(x_P)\in Q_P$. By Corollary \ref{isocor2}, there are paths $\alpha_P:\ui\to X_{\scrf}$, $P\in\scrf$ from $x_{\scrf}$ to $x_P$ so that the homomorphism $\phi:\ast_{P\in\scrf}\pi_1(P,x_P)\to \pi_1(X_{\scrf},x_{\scrf})$, defined by $\phi([\gamma])=[\alpha_P\cdot\gamma\cdot\alpha_{P}^{-1}]$ for $[\gamma]\in \pi_1(P,x_P)$ is an isomorphism. Let $\beta_P=g\circ \alpha_P:\ui\to Y_{\scrg}$. By the last statement of Corollary \ref{isocor2}, the homomorphism $\psi:\ast_{P\in\scrf}\pi_1(Q_P,y_P)\to \pi_1(Y_{\scrg},y_{\scrg})$, defined by $\phi([\delta])=[\beta_P\cdot\delta\cdot\beta_{P}^{-1}]$ for $[\delta]\in \pi_1(Q_P,y_P)$ is injective (but need not be surjective). Finally, recall that each restriction $g|_{P}:P\to Q_{P}$ agrees with $f|_{P}:P\to Q_P$ and is $\pi_1$-injective by assumption. Therefore, the free product $\ast_{P\in\scrf}(g|_{P})_{\#}:\ast_{P\in\scrf}\pi_1(P,x_P)\to \ast_{P\in\scrf}\pi_1(Q_P,y_P)$ is injective. Since the following diagram commutes, it follows that $g_{\#}$ is injective.
\[\xymatrix{
\pi_1(X_{\scrf},x_{\scrf}) \ar[rr]^-{g_{\#}} && \pi_1(Y_{\scrg},y_{\scrg}) \\
\ast_{P\in\scrf}\pi_1(P,x_P) \ar[u]^-{\phi}_-{\cong} \ar@{^{(}->}[rr]_-{\ast_{P\in\scrf}(g|_{P})_{\#}} && \ast_{P\in\scrf}\pi_1(Q_P,y_P) \ar@{^{(}->}[u]_-{\psi}
}\]

\end{proof}

\begin{definition}
Let $(X,\scrp)$ be a disjointly tree-graded space $(X,\scrp)$. We say that $(X,\scrp)$ is a \textit{string-light space} if it has the property that every piece $P$ of $X$ meets $\ov{X\backslash P}$ at a single point $x_P$ (called the \textit{attachment point} of $P$). If $\scrq\subseteq \scrp$ is the set of non-degenerate pieces of $X$, then we refer to $T=X\backslash \bigcup\{P\backslash\{x_p\}\mid P\in\scrq\}$ as the \textit{wire} of $(X,\scrp)$.
\end{definition}

The wire $T$ of a string-light space $(X,\scrp)$ may be identified with the metric quotient $X_{\scrp}$ where each piece of $X$ is identified to a point. Hence, $T$ is a topological $\bbr$-tree. In the next application of our main result, we consider maps $X\to X/T$ that collapse the wire of a string-light space while permitting the quotient $X/T$ to be either a topological or metric quotient (formally $X/T$ may be equipped with any topology between the quotient topology and the topology induced as a subspace of the direct product $\prod_{P\in\scrq}P$).

\begin{corollary}\label{quotientcor1}
Let $(X,\scrp)$ be a string-light space with wire $T$ and such that $\pc(X)$ is uniformly $1$-$UV_0$. Suppose that $f:X\to Y$ is a map satisfying the following:
\begin{enumerate}
\item $f(T)$ is a point,
\item $f$ is bijective on $X\backslash T$,
\item if $P$ is a non-degenerate piece of $X$ and $r_P:X\to P$ is the canonical retraction, then there is a map $R_P:Y\to P$ such that $R_P\circ f=r_P$.
\end{enumerate}
Then $f$ is $\pi_1$-injective.
\end{corollary}

\begin{proof}
First, we make an observation about the hypotheses on $f$. Let $\scrq$ denote the set of non-degenerate pieces of $X$ and let $\sigma_P:P\to X$ denote the inclusion map for each piece $P$ of $X$. Then $R_P\circ f\circ \sigma_P=r_P\circ \sigma_P=id_{P}$ when $P\in\scrq$. Therefore, $f$ maps each $P\in\scrq$ homeomorphically onto $f(P)$. Let $\zeta_P=f\circ \sigma_P:P\to Y$ denote the resulting embedding.

Suppose $\alpha:S^1\to X$ is an essential loop based at $x_0\in T$ and let $y_0\in Y$ be the image of $T$. Using Corollary \ref{maincorollary}, find a finite set of a $\scrf=\{P_1,P_2,\dots,P_n\}\subseteq \scrp$ such that if $\Gamma_{\scrf}:X\to X_{\scrf}^{qt}$ is the topological quotient map, then $\Gamma_{\scrf}\circ\alpha:S^1\to X_{\scrf}^{qt}$ is essential. Let $x_i$ denote the attachment point of $P_i$, which we also take as the basepoint of $P_i$, and let $\bigvee_{i=1}^{n}P_i$ be the finite wedge with wedgepoint $w_0$. Let $g:X_{\scrf}^{qt}\to \bigvee_{i=1}^{n}P_i$ be the quotient map that sends the topological tree $\Gamma_{\scrf}(T)$ in $X_{\scrf}^{qt}$ to $w_0$. The hypotheses on $f$ imply that $f$ is constant on the fibers of $g\circ \Gamma_{\scrf}$ and so there is an induced map $\zeta:\bigvee_{i=1}^{n}P_i\to Y$ such that $\zeta\circ g\circ \Gamma_{\scrf}=f$. We claim that both $g$ and $\zeta$ are $\pi_1$-injective from which it will follow that $f\circ\alpha$ is essential.

It follows from Corollary \ref{isocor2} that the inclusion maps $P_i\to X_{\scrf}^{qt}$ induce an isomorphism $\theta_1:\ast_{i=1}^{n}\pi_1(P_i,x_{i})\to \pi_1(X_{\scrf}^{qt},x_0)$ (here, we choose the conjugating paths from $x_{i}$ to $x_0$ to have image entirely in $T$). The inclusion maps $P_i\to \bigvee_{i=1}^{n}P_i$ also induce a canonical homomorphism  $\theta_2:\ast_{i=1}^{n}\pi_1(P_i,x_{i})\to \pi_1(\bigvee_{i=1}^{n}P_i,w_0)$, which is known to always be injective \cite[Proposition 5]{Edaonepointunions}. Because of our choice of conjugating paths for $\theta_1$, we have $\theta_2=g_{\#}\circ\theta_1$. It follows that $g_{\#}$ is injective.
\[\xymatrix{
\pi_1(X,x_0) \ar[rr]^-{f_{\#}} \ar[d]_-{\Gamma_{\scrf\#}} && \pi_1(Y,y_0) \\
\pi_1(X_{\scrf},x_0) \ar[rr]_-{g_{\#}} && \pi_1(\bigvee_{i=1}^{n}P_i,w_0) \ar[u]_-{\zeta_{\#}}\\
& \ast_{i=1}^{n}\pi_1(P_i,x_i) \ar[ul]^-{\theta_1}_{\cong} \ar@{^{(}->}[ur]_-{\theta_2}
}\]
Let $\rho_i: \prod_{i=1}^{n}P_i\to P_i$ be the product projection maps. Note that there is a canonical embedding $\mu:\bigvee_{i=1}^{n}P_i\to \prod_{i=1}^{n}P_i$ for which $\rho_i\circ \mu=R_{P_i}\circ\zeta$ for each $i$. The maps $R_{P_i}:Y\to P_i$ together induce a map $R:Y\to \prod_{i=1}^{n}P_i$ for which $\rho_i\circ R=R_{P_i}$. Since $\rho_i\circ \mu = R_{P_i}\circ \zeta= \rho_i\circ R\circ\zeta$ for each $i$, we have $R\circ \zeta=\mu$. Thus $\mu$ is an embedding. It follows that $\zeta$ is a section and is therefore $\pi_1$-injective.
\[\xymatrix{
P_i \ar@{_{(}->}[dr] \ar[r]^-{\zeta_{P_i}} &Y \ar@/^1.5pc/[rr]^-{R_{P_i}} \ar[r]_-{R} & \prod_{i=1}^{n}P_i \ar[r]_-{\rho_i} & P_i \\
& \bigvee_{i=1}^{n}P_i \ar[u]^-{\zeta} \ar@{^{(}->}[ur]_-{\mu}
}\]
\end{proof}

\end{document}